\documentclass[aap,preprint]{imsart}
\usepackage{amsmath,amssymb}
\usepackage[T1]{fontenc}      % Police contenant les caractères français
\usepackage{geometry}         % Définir les marges
\usepackage[francais,english]{babel}  % Placez ici une liste de langues, la
\usepackage{mathrsfs}
\newtheorem{Theorem}{Theorem}[part]
\newtheorem{Definition}{Definition}[part]
\newtheorem{Proposition}{Proposition}[part]
\newtheorem{Assumption}{Assumption}[part]
\newtheorem{Lemma}{Lemma}[part]
\newtheorem{Corollary}{Corollary}[part]
\newtheorem{Remark}{Remark}[part]

\def \F{I\!\!F}
\def \H{I\!\!H}

\def \N{I\!\!N}
\def \R{I\!\!R}

\def\Fc{{\cal F}}

\def\L{{\cal L}}

\def\T{{\cal T}}

\def\Dzw1#1{\frac{\partial^2 #1}{\partial z \partial w_1}}

\def\Dzb1#1{\frac{\partial^2 #1}{\partial z \partial b_1}}

\newcommand{\fproof}{\hfill $\square$ \bigskip}

\newtheorem{definition}{Definition}[section]
\newtheorem{example}{Example}[section]
\newtheorem{theorem}[definition]{Theorem}

\newtheorem{Property}[definition]{Property}

\newtheorem{lemma}[definition]{Lemma}

\def\RB{\mathbb{R}}

\def\BC{\mathcal{B}}
\def\FC{\mathcal{F}}
\def\PC{\mathcal{P}}

\def\R{{\bf R}}

\def\1B{\text{1\!\!I}}

\def\tN{\tilde{N}}

\def\RB{\mathbb{R}}

\def\FC{\mathcal{F}}
\def\PC{\mathcal{P}}

\def\R{{\bf R}}

\def\1B{\text{1\!\!I}}

\def\tN{\tilde{N}}

\def\RB{\mathbb{R}}

\newcommand{\Q}{\mathbb{Q}}
\newcommand{\cf}{\mathcal{F}}

\newcommand{\ct}{\mathcal{T}}
\newcommand{\stopt}{\mathcal{T}_{t,T}}
\newcommand{\stops}{\mathcal{T}_{S,T}}
\newcommand{\stopo}{\mathcal{T}_{0,T}}
\DeclareMathOperator*{\esssup}{ess\,sup}
\DeclareMathOperator{\e}{e}

\newcommand{\dbarY}{\overline{\overline{Y}}}
\newcommand{\barY}{\overline{Y}}

\newcommand{\vertiii}[1]{{\left\vert\kern-0.25ex\left\vert\kern-0.25ex\left\vert #1 
    \right\vert\kern-0.25ex\right\vert\kern-0.25ex\right\vert}}

\newcommand{\vvertiii}[1]{{\vert\kern-0.25ex\vert\kern-0.25ex\vert #1 
    \vert\kern-0.25ex\vert\kern-0.25ex\vert}}

\newcommand{\triple}{\vert\kern-0.25ex\vert\kern-0.25ex\vert}

\newcommand{\BOX}{\ensuremath\Box}
\newenvironment{proof}{{\vskip\baselineskip\noindent\textbf{Proof:}}}%
{\hspace*{.1pt}\hspace*{\fill}\BOX\vskip\baselineskip}

\newcommand{\limsupn}{\limsup_{n\rightarrow\infty}}

\newcommand{\limup}{\lim_n\!\smash{\uparrow}}

\DeclareMathAlphabet{\mathpzc}{OT1}{pzc}{m}{it}
\begin{document}
\begin{frontmatter}
%\title{A Sample Document\thanksref{T1}}

%\thankstext{T1}{Footnote to the title with the ``thankstext'' command.}

\title{Reflected BSDEs when the obstacle is not right-continuous and optimal stopping} 
\runtitle{RBSDEs with irregular obstacle and optimal stopping}

\begin{aug}

%\author{Miryana Grigorova \thanks{
%Institut für Mathematik,
%Humboldt Universität zu Berlin, Unter den Linden 6, 10099 Berlin, Germany   
%Universit\'e Paris 7 Denis Diderot, Boite courrier 7012, 75251 Paris Cedex 05, France, 
%and  INRIA Paris-Rocquencourt, Domaine de Voluceau, Rocquencourt, BP 105, Le Chesnay Cedex, 78153, France, 
%email: {\tt quenez@math.univ-paris-diderot.fr}
%}
\author{\fnms{Miryana} \snm{Grigorova}\thanksref{m1}
\ead[label=e1]{miryana.grigorova@hu-berlin.de}},
\author{\fnms{Peter} \snm{Imkeller}\thanksref{m1}\ead[label=e2]{imkeller@math.hu-berlin.de}},
\author{\fnms{Elias} \snm{Offen}\thanksref{m2}\ead[label=e3]{offen@mopipi.ub.bw}},
\author{\fnms{Youssef} \snm{Ouknine}\thanksref{m3},
\ead[label=e4]{ouknine@ucam.ac.ma}
\ead[label=u1,url]{http://www.foo.com}}
\and
\author{\fnms{Marie-Claire} \snm{Quenez}\thanksref{m4}
\ead[label=e5]{quenez@math.univ-paris-diderot.fr}}
%\thankstext{t1}{Some comment}
%\thankstext{t2}{First supporter of the project}
%\thankstext{t3}{Second supporter of the project}
\runauthor{M. Grigorova et al.}
\affiliation{Humboldt University-Berlin \thanksmark{m1},  University of Botswana \thanksmark{m2}, Université Cadi Ayyad \thanksmark{m3}, and Université Paris-Diderot \thanksmark{m4} }
\address{Institut für Mathematik\\
Humboldt Universität zu Berlin\\
Unter den Linden 6\\
10099 Berlin, Germany\\
\printead{e1}\\
\phantom{E-mail:\ }\printead*{e2}\\}
\address{University of Botswana\\
 Private Bag UB 00704\\
 Gaborone, Botswana\\
\printead{e3}\\
%\printead{u1}
}
\address{Département de Mathématiques\\
 Faculté des Sciences Semlalia\\
 Université Cadi Ayyad\\
 B.P. 2390\\
 Marrakech, Morocco\\
\printead{e4}\\}
%\phantom{E-mail:\ }\printead*{e2}}
\address{Laboratoire de Probabilités\\
 et Modèles Aléatoires\\
 Université Paris-Diderot\\
 Boîte courrier 7012\\
 75251 Paris Cedex 05, France\\
\printead{e5}\\
%\printead{u1}
}

\end{aug}

%%\and Peter Imkeller \footnotemark[1]
%%\thanks{Institut für Mathematik,
%%Humboldt Universität zu Berlin,  Unter den Linden 6, 10099 Berlin, Germany   
%Universit\'e Paris 7 Denis Diderot, Boite courrier 7012, 75251 Paris Cedex 05, France, 
%and  INRIA Paris-Rocquencourt, Domaine de Voluceau, Rocquencourt, BP 105, Le Chesnay Cedex, 78153, France, 
%email: {\tt quenez@math.univ-paris-diderot.fr}
%}
%%\and Elias Offen
%\thanks{University of Botswana, Private Bag UB 00704, Gaborone, Botswana}
%\and Youssef Ouknine
%\thanks{Département de Mathématiques, Faculté des Sciences Semlalia, Université Cadi Ayyad, Marrakech, Morocco}
%\and Marie-Claire Quenez
%\thanks{Laboratoire de Probabilités et Modèles Aléatoires, Université Paris-Diderot,  Boîte courrier 7012, 75251 Paris Cedex 05, France}
%\thanks{ , email: {\tt }}
%}
%\date{\today}
%\date{}
%\maketitle

\begin{abstract}
$\quad$In the first part of the paper, we study reflected backward stochastic differential equations (RBSDEs) with lower obstacle which is assumed to be  right upper-semicontinuous but not necessarily right-continuous. We prove existence and uniqueness of the solutions to such RBSDEs in appropriate Banach spaces. The result is established by using
some results from optimal stopping theory, some tools from the general theory of processes such as Mertens decomposition of optional strong supermartingales,  as well as an appropriate generalization of Itô's formula due to Gal'chouk and Lenglart. In the second part of the paper, we provide some links between the RBSDE studied in the first part and an optimal stopping problem in which the risk of a financial position $\xi$ is assessed by an $f$-conditional expectation $\mathcal{E}^f(\cdot)$ (where $f$ is a Lipschitz driver). We characterize the "value function" of the problem in terms of the solution to our RBSDE. Under an additional assumption of left upper-semicontinuity along stopping times on $\xi$, we show the existence of an optimal stopping time. We also provide a generalization of Mertens decomposition to the case of strong $\mathcal{E}^f$-supermartingales.   
\end{abstract}

\begin{keyword}[class=MSC]
\kwd[Primary ]{60G40, 93E20, 60H30}
%\kwd{60K35}
\kwd[; secondary ]{60G07, 47N10}
\end{keyword}

\begin{keyword}
\kwd{backward stochastic differential equation, reflected backward stochastic differential equation, optimal stopping, $f$-expectation, strong optional supermartingale, Mertens decomposition, dynamic risk measure, strong $\mathcal{E}^f$-supermartingale}
\end{keyword}

%%%\noindent{\bf AMS 1991 subject classifications~:} 

\end{frontmatter}

\section{Introduction}\label{sec1}

%THE CASE WHERE THERE IS ALSO AN INDEPENDENT POISSON PROCESS: SOME MODIFICATIONS

Backward stochastic differential equations (BSDEs) have been introduced by Bismut (\cite{Bismut2}) in the case of a linear driver. The general theory of existence and uniqueness of solutions to BSDEs has been developed by Pardoux and Peng \cite{Pape90}. Through a result of Feynman-Kac-type, these authors have linked the theory of BSDEs to that of quasilinear parabolic partial differential equations (cf.\cite{Pape92}). BSDEs have found number of applications in finance, among which  pricing  and hedging of European options and recursive utilities (cf., for instance, \cite{KPQ}, \cite{EQ96}). Also, a useful family of operators, the family of so-called $f$-conditional expectations, has been defined through the notion of BSDEs and used in the literature on dynamic risk measures (cf., for instance, \cite{BEK}, \cite{Pe04}, \cite{QuenSul}, \cite{Gianin}, \cite{Bayraktar-2} 
among others).  %Let $T>0$ be a final horizon. 
We recall that the $f$-conditional expectation at time $t\in[0,T]$ (where $T>0$ is a fixed final horizon) is  the operator which maps  a given square-integrable terminal condition $\xi_T$ to the position at time $t$ of (the first component of) the solution to the BSDE with parameters $(f,\xi_T)$. The operator is denoted $\mathcal{E}^f_{t,T}(\cdot).$ \\ %has been used in the risk measurement literature. This operator is denoted by $\mathcal{E}_{t,T}^f(\cdot)$.\\ 
Reflected backward stochastic differential equations (RBSDEs) can be seen as a variant of BSDEs in which the (first component of the) solution is constrained to remain greater than or equal to a given process called the obstacle. Compared to the case of (non-reflected) BSDEs, there is an additional nondecreasing predictable process which keeps the (first component of the) solution above the obstacle. 

RBSDEs have been introduced  by El Karoui et al. \cite{ElKaroui97} in the case of a Brownian filtration and a continuous obstacle. %Recall that  
%there are two approaches to solve reflected BSDEs: the first one is based on results of the optimal stopping theory and a fixed point theorem, 
%whereas the second one consists to approximate the solution by a sequence of penalized non reflected BSDEs.
 In \cite{EQ96}, El Karoui and Quenez also study their links with the (non linear) pricing of American options. There have been several extensions of these works to the case of a discontinuous obstacle and/or a larger stochastic basis than the Brownian one  (cf. \cite{Ham}, \cite{CM}, \cite{HO1}, \cite{Essaky}, \cite{HO2}, \cite{QuenSul2}). In all these extensions an assumption of right-continuity on  the obstacle is made. % There have been several extensions of this work to the case of a discontinuous obstacle and/or a larger stochastic basis than the Brownian one  (cf. \cite{Ham}, \cite{CM}, \cite{HO1}, \cite{Essaky}, \cite{HO2}, \cite{QuenSul2}). In all these extensions an assumption of right-continuity on  the obstacle is made.  
%%Let us mention also that some  links have been established between, on one hand, solutions of %%%RBSDEs with tight-continuous obstacle, and on the other hand,  optimal stopping problems in %%%which one aims at stopping optimally a process evaluated by an $f$-conditional expectation (cf. %%%…)%, have been established (cf. ……)  
%%%Let us mention that optimal stopping problems involving « non-linear expectations » have also %%%been studied  in …. (or by….).

In the first part of the present paper we consider a further extension of the theory of RBSDEs to the case where the obstacle is not necessarily right-continuous. Compared to the right-continuous case, the additional nondecreasing process, which "pushes" the (first component of the) solution to stay above the obstacle, is no longer right-continuous.  %%For the sake of simplicity, we place ourselves in the framework of a Brownian filtration, but we note that our results can be generalized to the case of a larger stochastic basis (cf. Section \ref{sec_extension}). We establish existence and uniqueness of the solution in  appropriate Banach spaces.
 %%%We also prove a characterization of the first component of the solution in terms of the « value process » of an optimal stopping problem. 
%We adopt the first approach already mentioned, the second one appearing less adapted to our case because of the lack of regularity of the obstacle.
To prove our results we use some tools from 
the optimal stopping theory (cf.  \cite{Maingueneau},  \cite{EK}, \cite{KS2}, \cite{Kob}), some tools from the general theory of processes (cf. \cite{DM2}) such as {\em Mertens decomposition} of strong optional (but not necessarily right-continuous)  supermartingales (generalizing Doob-Meyer decomposition), a result from the potential theory (cf. \cite{DM2}), and a generalization of Itô's formula
 to the case of {\em strong optional semimartingales} in the vocabulary of \cite{Galchouk} (but not necessarily right-continuous)  due to Gal'chouk and Lenglart (cf. \cite{Lenglart}).

In the second part of the paper, we make some links between the RBSDEs studied in the first part and optimal stopping with $f$-conditional expectations. More precisely, we are interested in the following optimization problem: we are given a process $\xi$ modelling a dynamic financial position. The risk of $\xi$ is assessed by a dynamic risk measure which (up to a minus sign) is given by  an $f$-conditional expectation. The process $\xi$ is assumed to be right upper-semicontinuous, but not necessarily right-continuous. We aim at stopping the process $\xi$ in such a way that the risk be minimal. 
%%More precisely, we consider the following problem:
%%$$V(S):= \essinf-\mathcal{E}^f_{S,\tau}(\xi_\tau),$$  where $\stopo$ is the set of all stopping %%%times valued in the interval $[0,T].$ 
We characterize the value of the problem in terms of the unique solution to the RBSDE associated with obstacle $\xi$ and driver $f$ studied in the first part. We show the existence of an optimal stopping time for the problem under an additional assumption of left upper-semicontinuity along stopping times on $\xi$, and the existence of an $\varepsilon$-optimal stopping time in the more general case where this assumption is not made.  We provide an optimality criterion characterizing the optimal stopping times for the problem in terms of properties of the "value process". We thus extend some results of  \cite{QuenSul2} to the case where the optimized process $\xi$ is not cadlag. 
We also establish  a comparison principle for the RBSDEs studied in the first part of our paper, as well as a generalization of Mertens decomposition to the case of  $\mathcal{E}^f$-strong supermartingales.

The remainder of the paper is organized as follows: 

In Section \ref{sec2} we give some preliminary definitions and properties.  In Section \ref{sec3} we define our RBSDE and we prove existence and uniqueness of the solution. 
Section \ref{sec-charact} is dedicated to our optimal stopping problem with $f$-conditional expectations. In Subsection \ref{sec-charact_subsec_1}  we formulate and motivate the problem. In Subsection \ref{sec-charact_subsec_2} we characterize the value function of the problem in terms of the solution of the RBSDE studied in Section \ref{sec3}; we also give an optimality criterion and address the question of existence of $\varepsilon$-optimal  and optimal stopping times. In Section \ref{sec_add} we derive some useful additional results: comparison principle for our RBSDEs (Subsection \ref{subsec_comp}) and "generalized" Mertens decomposition for $\mathcal{E}^f$-strong supermartingales (Subsection \ref{subsec_Mertens}). In Section \ref{sec_extension} we briefly present some further extensions of our work. 
In the Appendix we recall some useful results ("classical" Mertens decomposition, a result from  potential theory, Gal'chouk-Lenglart's change of variables formula), we give the proofs of three results (Prop. \ref{Prop_Banach_space}, Prop. \ref{pro}, and Prop. \ref{compref}) used in the main part of the paper, and we  also give some examples. % for completeness, we also give a short proof of a property used in the main part.         

%%%%%recall the notion of strong martingale/supermartingale

%\section{Existence and Uniqueness results for  RBSDEs with jumps} \label{sec2}
\section{Preliminaries} \label{sec2}
Let $T>0$ be a fixed positive real number. Let $(E, \mathscr{E})$ be a measurable space 
equipped with a $\sigma$-finite positive measure $\nu$.
Let $(\Omega,  \Fc, P)$ be a probability space equipped with a  one-dimensional Brownian motion $W$ and with an independent  Poisson random measure $N(dt,de)$ with compensator $dt\otimes\nu(de)$. 
We denote by $\tilde N(dt,de)$ the compensated process, i.e. $\tilde N(dt,de):= N(dt,de)-dt\otimes\nu(de).$
Let $\F = \{\Fc_t \colon t\in[0,T]\}$ 
be  the (complete) natural filtration associated with $W$ and $N$. For $t\in [0,T],$ we denote by $\stopt$ the set of stopping times $\tau$ such that $P(t \leq\tau\leq T)=1.$ More generally, for a given stopping time $\nu\in \stopo$, we denote by $\ct_{\nu,T}$ the set of stopping times $ \tau$ such that $P(\nu \leq\tau\leq T)=1.$   \\

%Let $(\Omega, \cf, \F, P)$ be a Wiener filtered probability space satisfying the usual conditions of right-continuity and completeness. Let $W$ be a brownian motion on  $(\Omega, \cf, \F, P).$ For $t\in [0,T],$ we denote by $\stopt$ the set of stopping times $\tau$ such that $P(t \leq\tau\leq T)=1.$ More generally, for a given stopping time $\nu\in \stopo$, we denote by $\ct_{[\nu,T]}$ the set of stopping times $ \tau$ such that $P(\nu \leq\tau\leq T)=1.$   \\

%Let $(\Omega,  \Fc, P)$ be a probability space. 
%Let  $W$ be   a  one-dimensional Brownian motion. Let $(U, {\cal U})$ be a measurable space 
%equipped with a $\sigma$-finite positive measure $\nu$.
% Let  $N(dt,du)$ be a Poisson random measure with compensator $\nu(du)dt$ (see e.g. \cite{IW}). 
 %such that $\nu$ is a $\sigma$-finite measure on $ U$, equipped with its Borel field ${\cal B}( U)$. 
% Let $\tilde N(dt,du)$ be its compensated process.
%Let  $\F = \{\Fc_t , t \geq 0 \}$ 
%be  the natural filtration associated with $W$ and $N$. 

%\subsection{Notation and classical results on BSDES with jumps} 
%\subsection{Notation}

%\paragraph{Notation.}

%We adopt the following notation:
 %For each $T>0$ and 
 We use the following notation: 
\begin{itemize}
\item ${\cal P}$ is  the predictable $\sigma$-algebra
on $ \Omega\times [0,T]$.
\item $\mathrm{Prog}$ is the progressive $\sigma$-algebra
on $ \Omega\times [0,T]$.
\item ${\cal B}(\R)$ (resp. ${\cal B}(\R^2)$)   is the Borel 
$\sigma$-algebra on $\R$ (resp. $\R^2$). 
\item 
 $L^2({\cal F}_T)$  is the set of random variables which are  $\Fc 
_T$-measurable and square-integrable.
\item $L^2_\nu$ is the set of $(\mathscr{E}, \mathcal{B}(\R))$-measurable functions $\ell:  E \rightarrow \R$ such that  $\|\ell\|_\nu^2:= \int_{ E}  |\ell(e) |^2 \nu(de) <  \infty.$
For $\ell\in\L^2_\nu$, $\mathpzc{k}\in\L^2_\nu$, we define $\langle \ell, \mathpzc{k}\rangle_\nu:=\int_E \ell(e)\mathpzc{k} (e) \nu(de)$.
\item ${\cal B}(L^2_\nu ) $ is the Borel 
$\sigma$-algebra on $L^2_\nu $.
\item    $\H^{2}$ is the set of 
$\R$-valued predictable processes $\phi$ with
 $\| \phi\|^2_{\H^{2}} := E \left[\int_0 ^T |\phi_t| ^2 dt\right] < \infty.$ 
%\item $L^p_\nu$ is the set of measurable functions $\ell:  U \rightarrow \R$ such that  $\int_{ U}  |\ell(u) |^p \nu(du) < + \infty.$

\item $\H_{\nu}^{2}$ is  the set of $\R$-valued processes $l: (\omega,t,e)\in(\Omega\times[0,T] \times  E)\mapsto l_t(\omega, e)$ which are {\em predictable}, that is $(\PC \otimes {\mathscr{E}},\BC(\R))$-measurable,
 %$l : (\Omega\times[0,T] \times  E,\; \PC \otimes {\mathscr{E}}) \rightarrow (\R\;,  \BC(\R)); \quad 
%(\omega,t,e) \mapsto l_t(\omega, e)
%$
and such that $\| l \|^2_{\H_{\nu}^{2}} :=E\left[ \int_0 ^T \|l_t\|_{\nu}^2 \,dt   \right]< \infty.$

%%%pour l'instant nous enlevons cette partie
%The set  $L^2_\nu$ is a 
% Hilbert space equipped with the scalar product 
%$\langle \delta    , \, \ell \rangle_\nu := \int_{ U} \delta(u) \ell(u) \nu(du)$ for all $  \delta  , \, \ell \in L^2_\nu \times L^2_\nu,$ 
%and the norm $\|\ell\|_\nu^2 :=\int_{ U}  |\ell(u) |^2 \nu(du) < + \infty.$

%\item   ${\cal P}$ denotes  the progressive $\sigma$-algebra
%on $[0,T]  \times \Omega$ and ${\cal B}(\R^2)$ (resp ${\cal B}(L^2_\nu ) $)  is the Borelian 
%$\sigma$-algebra on $\R^2$  (resp.  on $L^2_\nu $). 

%%\item $\H_{\nu}^{p,T}$ is  the set of processes $l$ which are {\em predictable}, that is, measurable\\
 %$l : ([0,T]  \times \Omega \times  U,\; \PC \otimes {\cal U})  \rightarrow (\R\;,  \BC(\R)); \quad 
%(\omega,t,u) \mapsto l_t(\omega, u)
%$
%with $$\| l \|^p_{\H_{\nu}^{p,T}} :=E\left[( \int_0 ^T \|l_t\|_{\nu}^2 \,dt ) ^{\frac{p}{2}} \right]< \infty.$$

%\item  ${\cal S}^{2}$ is the set of $\R$-valued optional 
% processes $\phi$ such that
%$\vertiii{\phi}^2_{{\cal S}^{2}} := E[\esssup_{\tau\in\stopo} |\phi_\tau |^2] <  \infty.$ 
%\item  $\T_{0}$ is the set of 
%stopping times $\tau$ such that $\tau \in [0,T]$ a.s
%\item For $S$ in $\T_{0}$,    $\T_{S}$  is the set of 
%stopping times 
%$\tau$ such that $S \leq \tau \leq T$ a.s. 
\end {itemize}
We  introduce the vector space $ {\cal S}^{2}$ defined as the space of $\R$-valued optional 
(not necessarily cadlag)
 processes $\phi$ such that
$\vertiii{\phi}^2_{{\cal S}^{2}} := E[\esssup_{\tau\in\stopo} |\phi_\tau |^2] <  \infty.$ \footnote{Note that when $\phi$ is right-continuous, $\vertiii{\phi}^2_{{\cal S}^{2}}= E[\sup_{t\in [0,T]} |\phi_t |^2]$ 
(cf. Remark \ref{esssup_cadlag_process}). }

\begin{Proposition}\label{Prop_Banach_space}
The map 
$\vertiii{\cdot}_{{\cal S}^{2}}$
 %$\phi \mapsto \vertiii{\phi}^2_{{\cal S}^{2}}$ 
 is a norm on $\mathcal{S}^2$. 
In particular, if $\phi \in {\cal S}^{2}$ is such that $\vertiii{\phi}^2_{{\cal S}^{2}}=0$, then $\phi$ is indistinguishable
 from the null process, that is  $\phi_t =0$, $0\leq t \leq T$ a.s. Moreover, the space $\mathcal{S}^2$ endowed with the norm $\vertiii{\cdot}_{{\cal S}^{2}}$ is a Banach space.
\end{Proposition}
%\begin{proof}
The proof is given in the Appendix.
%\end{proof}
%\footnote{Contrary to $\vertiii{\cdot}_{{\cal S}^{2}}$,  the map $\|\cdot\|_{\H^2}$ is only a semi norm on $\mathcal{S}^2$ (see Remark \ref{indistingable} for more details).}

%\begin{Proposition}\label{Prop_Banach_space}
%The map $\vertiii{\cdot}_{{\cal S}^{2}}$ is a norm on $\mathcal{S}^2$ and the space $\mathcal{S}^2$ endowed with this norm  is a Banach space.
%
%\end{Proposition}
%\begin{proof}
%The proof is given in the Appendix.\\
%\end{proof}

\noindent
We will also use the following notation:\\
Let $\beta>0$. For $\phi\in\H^2$, $\|\phi\|_\beta^2:= E[\int_0^T \e^{\beta s} \phi_s^2 ds ].$ 
We note that  on the space $\H^2$  the norms $\|\cdot\|_\beta$ and $\|\cdot\|_{\H^2}$ are equivalent.
For $l\in\H^2_\nu$, $\|l\|_{\nu,\beta}^2:=E[\int_0^T \e^{\beta s} \|l_s\|_\nu^2  ds].$ 
On the space $\H^2_\nu$  the norms $\|\cdot\|_{\nu,\beta}$ and $\|\cdot\|_{\H^2_\nu}$ are equivalent.
For $\phi\in{\cal S}^2$, we define $\vertiii{\phi}^2_\beta:=E[\esssup_{\tau \in \stopo}\e^{\beta \tau}\phi_\tau^2]$. We note that $\vertiii{\cdot}_\beta$ is a norm on ${\cal S}^{2}$  equivalent to the norm $\vertiii{\cdot}_{{\cal S}^{2}}$.   

\begin{Remark}By a slight abuse of notation, we shall also write $\| \phi\|^2_{\H^{2}}$ (resp. $ \|\phi\|^2_\beta$) for  $E \left[\int_0 ^T |\phi_t| ^2 dt\right]$ (resp. $E \left[\int_0 ^T \e^{\beta t}|\phi_t| ^2 dt\right]$) in the case of a progressively measurable 
real-valued process $\phi$ (cf., e.g., Remark 2.1 in \cite{CM} for the same notation). 
\end{Remark}

\begin{definition}[Driver, Lipschitz driver]\label{defd}
A function $f$ is said to be a {\em driver} if 
\begin{itemize}
\item  
$f: \Omega  \times [0,T]  \times \R^2 \times L^2_\nu \rightarrow \R $\\
$(\omega, t,y, z, \mathpzc{k}) \mapsto  f(\omega, t,y, z, \mathpzc{k})  $
  is $ \mathrm{Prog} \otimes {\cal B}(\R^2)  \otimes {\cal B}(L^2_\nu) 
- $ measurable,  
\item $E[\int_0^T f(t,0,0,0)^2dt] < + \infty$.
\end{itemize} 
A driver $f$ is called a {\em Lipschitz driver} if moreover there exists a constant $ K \geq 0$ such that $dP \otimes dt$-a.e.\,, 
for each $(y_1, z_1, \mathpzc{k}_1)\in \R^2 \times L^2_\nu$, $(y_2, z_2, \mathpzc{k}_2)\in \R^2 \times L^2_\nu$, 
$$|f(\omega, t, y_1, z_1, \mathpzc{k}_1) - f(\omega, t, y_2, z_2, \mathpzc{k}_2)| \leq 
K (|y_1 - y_2| + |z_1 - z_2| +   \|\mathpzc{k}_1 - \mathpzc{k}_2 \|_\nu).$$
A Lipschitz driver  $f$ is called {\em predictable } if moreover $f$ is $ \mathcal{P}\otimes {\cal B}(\R^2)  \otimes {\cal B}(L^2_\nu) 
- $ measurable.
\end{definition}

%A function $f$ is said to be a {\em driver} if 
%\begin{itemize}
%\item  
%$f: \Omega \times [0,T]\times \R^2  \rightarrow \R $\\
%$(\omega, t,y,z) \mapsto  f(\omega, t, y, z) $
%  is $ {\cal P} \otimes {\cal B}(\R^2)  
%- $ measurable,  
%\item $E \left[(\int_0 ^T |f(t,0,0)| ^2 dt)\right]<\infty$.
%\end{itemize} 
%A driver $f$ is called a {\em Lipschitz driver} if moreover there exists a constant $ K \geq 0$ such that $dP \otimes dt$-a.e.\,, 
%for each $(y_1, z_1)$, $(y_2, z_2)$, 
%$$|f(\omega, t, y_1, z_1) - f(\omega, t, y_2, z_2)| \leq 
%K (|y_1 - y_2| + |z_1 - z_2|).$$
%\end{definition}
For real-valued random variables $X$ and $X_n$, $n \in \N$, the notation   "$X_n\uparrow X$" will stand  for "the sequence $(X_n)$ is nondecreasing and converges to $X$ a.s.".  \\
For a ladlag process $\phi$, we denote by $\phi_{t+}$ and $\phi_{t-}$ the right-hand and left-hand limit of $\phi$ at  $t$. We denote by $\Delta_+ \phi_t:=\phi_{t_+}-\phi_t$ the size of the right jump of $\phi$ at $t$, and by   $\Delta \phi_t:=\phi_t-\phi_{t-}$ the size of the left jump of $\phi$ at $t$.

\begin{definition} \label{defr} 
%An optional process $(\phi_t)$ is said to be {\em right upper-semicontinuous (r.u.s.c.)  along stopping times} (resp. {\em left upper-semicontinuous (l.u.s.c.) along stopping times})  if for all $\tau \in \stopo$ and for all nonincreasing (resp. nondecreasing) sequence of stopping times $ (\tau_n)$ such that $\tau^n \downarrow \tau$ (resp. $\tau^n \uparrow \tau$)   a.s.\,,
%%\begin{equation}\label{usc}
%$\phi_{\tau} \geq \limsup_{n\to \infty} \phi_{\tau_n} \;\; \rm{a.s.}$
%
 Let $\tau \in \stopo$. An optional process $(\phi_t)$ is said to be {\em right upper-semicontinuous (r.u.s.c.)  } (resp. {\em left upper-semicontinuous (l.u.s.c.)}) {\em along stopping times} at the stopping time $\tau$  if for all nonincreasing (resp. nondecreasing) sequence of stopping times $ (\tau_n)$ such that $\tau^n \downarrow \tau$ (resp. $\tau^n \uparrow \tau$)   a.s.\,,
%\begin{equation}\label{usc}
$\phi_{\tau} \geq \limsup_{n\to \infty} \phi_{\tau_n} \;\; \rm{a.s.}$.  The process $(\phi_t)$ is said to be r.u.s.c. (resp. l.u.s.c.)   along stopping times if it is   r.u.s.c. (resp. l.u.s.c.)   along stopping times at each $\tau \in \stopo$. The right- (resp. left-) continuity property of  an optional process $(\phi_t)$ along stopping times at a stopping time $\tau$ is defined similarly.
%\end{equation} 
\end{definition}

%\begin{definition} \label{defr} An optional process $(\phi_t)$ is said to be {\em right upper-semicontinuous (r.u.s.c.)  along stopping times} (resp. {\em left upper-semicontinuous (l.u.s.c.) along stopping times})  if for all $\tau \in \stopo$ and for all nonincreasing (resp. nondecreasing) sequence of stopping times $ (\tau_n)$ such that $\tau^n \downarrow \tau$ (resp. $\tau^n \uparrow \tau$)   a.s.\,,
%%\begin{equation}\label{usc}
%$\phi_{\tau} \geq \limsup_{n\to \infty} \phi_{\tau_n} \;\; \rm{a.s.}$
%%\end{equation} 
%\end{definition}

\begin{Remark}\label{Rmk_final_00}
 Note that if $(\phi_t)$ is an optional process and $\tau$ is a totally inaccessible stopping time, then   $(\phi_t)$ is  left-continuous along stopping times at $\tau$.
\\ 
If the process $( \phi_{t})$ has left limits, $( \phi_{t})$ is l.u.s.c. (resp. left-continuous) along stopping times if and only if for each predictable stopping time $\tau \in \stopo$,  $\phi_{\tau-} \leq \phi_{\tau}$ (resp. 
$\phi_{\tau-} = \phi_{\tau}$)
a.s. 

%\end{Remark}

%\begin{Remark}\label{ruscbis}
Note, moreover, that if an optional process $(\phi_t)$ is right upper-semicontinuous (r.u.s.c.), then it is r.u.s.c. along stopping times. The converse also holds true; it is a difficult result  of the general theory of processes proved in \cite[Prop. 2, page 300]{DelLen2}.
\end{Remark}

\section{Reflected BSDE whose obstacle is  not cadlag}\label{sec3}

Let $T>0$ be a fixed terminal time. Let $f$ be  a  driver. 
Let $\xi= (\xi_t)_{t\in[0,T]}$ be a left-limited  process in ${\cal S}^2$. %, that is, such that $E(\esssup_{\tau\in\stopo}  |\xi_\tau |^p) <  \infty$ 
%%%(of class (D)). 
We suppose moreover that the process $\xi$ is r.u.s.c. A process $\xi$ satisfying the previous properties will be called a  \emph{barrier}, or an  \emph{obstacle}.

\begin{Remark}\label{Rmk_left_uppersemicontinuous envelope}
Let us note that in the following definitions and results we can relax the assumption of existence of left limits for the obstacle $\xi$. All the results still hold true provided we replace the process  $(\xi_{t-})_{t\in]0,T]}$ by the  process $(\underline {\xi}_{t})_{t\in ]0,T]}$ defined by $\underline{\xi}_t:= \limsup_{s \uparrow t, s< t} \xi_s,$ for all $t\in]0,T].$ We recall that $\underline \xi$ is a predictable process (cf. \cite[Thm. 90, page 225]{DM1}). We call the process $\underline \xi$ the left upper-semicontinuous envelope of $\xi$.
\end{Remark}

%%%%%%définition améliorée
\begin{definition} \label{def_solution_RBSDE}
A process $(Y,Z,k,A,C)$ is said to be a solution to the reflected BSDE with parameters $(f,\xi)$, where $f$ is a driver and $\xi$ is an obstacle, if 
\begin{align}\label{RBSDE}
&(Y,Z,k,A,C)\in {\cal S}^2 \times \H^2 \times \H^2_\nu \times {\cal S}^2\times {\cal S}^2  \text{and}  \text{ a.s.} \text{ for all } t\in[0,T] \nonumber \\
&Y_t=\xi_T+\int_{t}^T f(s,Y_s,  Z_s,k_s)ds-\int_{t}^T  Z_s dW_s-\int_{t}^T \int_E k_s(e) \tilde N(ds,de) +A_T-A_t+C_{T-} -C_{t-},\\
%& \text{ for all } t\in[0,T]  \text{ a.s.,} \\
%& Y_\tau=\xi_T+\int_{\tau}^T f(t,Y_t,  Z_t,k_t)dt-\int_{\tau}^T  Z_t dW_t-\int_{\tau}^T\int_E  k_t(e) \tilde N(dt,de) +A_T-A_\tau+C_{T-} -C_{\tau-}, \\
%& \text{ where the equality holds } \text{ a.s. for all } \tau\in\stopo, \nonumber\\
%-dY_t = f(t,Y_t,  Z_t)dt +dA_t+dC_{t-} - Z_t  dW_t; \quad  Y_T = \xi_T, \\
&  Y_t \geq \xi_t   \text{ for all } t\in[0,T]  \text{ a.s.,} \label{RBSDE_inequality_barrier}\\
%& % (Y_t +\int_0^t f(s, Y_s, Z_s) ds)_{t\in[0,T]} \text{ is a strong supermartingale},   \label{RBSDE_strong_supermartingale}\\
& A \text{ is a nondecreasing right-continuous predictable process 
with } A_0= 0 \text{ and such that } \nonumber\\
& \int_0^T {\bf 1}_{\{Y_t > \xi_t\}} dA^c_t = 0 \text{ a.s. and } \; (Y_{\tau-}-\xi_{\tau-})(A^d_{\tau}-A^d_{\tau-})=0 \text{ a.s. for all predictable }\tau\in\stopo, \label{RBSDE_A}\\   
%%%%%%%%%%%%%%% \Delta A_t^d=  \Delta A_t^d \,{\bf 1}_{\{Y_{t^-} = \xi_{t^-}\}}  \quad   \text{a.s.}%%%%%%%%%%%
& C \text{ is a nondecreasing right-continuous adapted purely discontinuous process with } C_{0-}= 0  \nonumber\\
& \text{ and such that }
(Y_{\tau}-\xi_{\tau})(C_{\tau}-C_{\tau-})=0 \text{ a.s. for all }\tau\in\stopo. \label{RBSDE_C}
%%%\Delta C_t=\Delta C_t{\bf 1}_{\{Y_t= \xi_t\}}= (Y_t-Y_{t+}){\bf 1}_{\{Y_t= \xi_t\}}=- \Delta_+ Y_t{\bf 1}_{\{Y_t= \xi_t\}}\quad \rm{a.s.}\label{RBSDE_C}
\end{align}
\end{definition}
Here  $A^c$ denotes the continuous part of the process $A$ and $A^d$  its discontinuous part.\\
Equations \eqref{RBSDE_A} and \eqref{RBSDE_C} are referred to as \emph{minimality conditions} or \emph{Skorokhod conditions}.  
%\begin{Remark}\label{Rmk_def_RBSDE}
%We note that, {\bf by Proposition \ref{optional_section}}, a process $(Y,Z,k,A,C)\in {\cal S}^2 \times \H^2 \times \H^2_\nu\times {\cal S}^2\times {\cal S}^2$ satisfies equation \eqref{RBSDE} in the above definition if and only if
%$$Y_t=\xi_T+\int_{t}^T f(s,Y_s,  Z_s,k_s)ds-\int_{t}^T  Z_s dW_s-\int_{t}^T \int_E k_s(e) \tilde N(ds,de) +A_T-A_t+C_{T-} -C_{t-}, $$
%for all  $t\in[0,T]$ a.s.
%\end{Remark}
We note that, by a classical result of the general theory of processes (\cite[Theorem IV.84]{DM1}), a process $(Y,Z,k,A,C)\in {\cal S}^2 \times \H^2 \times \H^2_\nu\times {\cal S}^2\times {\cal S}^2$ satisfies equation \eqref{RBSDE} in the above definition if and only if $Y_\tau=\xi_T+\int_{\tau}^T f(t,Y_t,  Z_t,k_t)dt-\int_{\tau}^T  Z_t dW_t-\int_{\tau}^T\int_E  k_t(e) \tilde N(dt,de) +A_T-A_\tau+C_{T-} -C_{\tau-}$, where the equality holds  $\text{ a.s. for all } \tau\in\stopo$.
Let us also emphasize that if  $(Y,Z,k,A,C)$ satisfies the above definition, then the process $Y$
has left and right limits.
%\begin{Remark}
%\label{Rmk_def_RBSDE_strong_supermartingale}  
% Note that equality 
%\eqref{RBSDE} still holds with $f(t, Y_t, Z_t,k_t)$ replaced by $f(t, Y_{t^-}, Z_t,k_t)$.   
%Moreover, the process $(Y_t +\int_0^t f(s, Y_s, Z_s,k_s) ds)_{t\in[0,T]}$ is a  strong supermartingale (cf. Definition \ref{Def_surmartingale_forte}).
%\end{Remark}

\begin{Remark}\label{Rmk_the_jumps_of_C}
If $(Y,Z,k,A,C)$ is a solution to the RBSDE defined above, then $\Delta C_t(\omega)=Y_t(\omega)-Y_{t+}(\omega)$ for all $(\omega,t)\in\Omega\times[0,T)$ outside an evanescent set. This observation is a consequence of equation \eqref{RBSDE}.
It follows that $Y_t\geq Y_{t+}$, for all $t\in[0,T)$, which implies that $Y$ is necessarily r.u.s.c.
% For  the size $\Delta Y$ of the left jumps  of $Y$  %: $\Delta Y_t=Y_t - Y_{t^-}$, $t \in (0,T]$.
% we have the following:  

Moreover, since in our framework the filtration is quasi-left-continuous, martingales have only totally inaccessible jumps. Hence, if $\tau$ is a predictable stopping time,  we have  $Y_{\tau}-Y_{\tau-}= - (A_{\tau}-A_{\tau-})$ a.s.\, 
From this, together with Remark \ref{Rmk_final_00}, it follows that the process $Y$ is left-continuous along stopping times at a 
%predictable 
stopping time $\tau$ if and only 
if $\Delta A_{\tau}=0$ a.s.

We also note that equality 
\eqref{RBSDE} still holds with $f(t, Y_t, Z_t,k_t)$ replaced by $f(t, Y_{t-}, Z_t,k_t)$.   
Furthermore, the process $(Y_t +\int_0^t f(s, Y_s, Z_s,k_s) ds)_{t\in[0,T]}$ is a  strong supermartingale (cf. Definition \ref{Def_surmartingale_forte}).

\end{Remark}
%Also, $Y$ is necessarily r.u.s.c.; indeed, Remark  \ref{Rmk_the_jumps_of_C} and the non-decreasingness of $C$ imply that $Y_t\geq Y_{t+}$, for all $t\in[0,T)$.
%
\begin{Remark}[the particular case of a right-continuous obstacle]
In the particular case of  a right-continuous obstacle $\xi$, we have that $Y$ is right-continuous. Indeed, observe that  $Y_t\geq Y_{t+}\geq \xi_{t+}= \xi_{t}$ (due to the right upper semicontinuity of $Y$  and to inequality \eqref{RBSDE_inequality_barrier}). Hence, if $t$ is such that $Y_t=\xi_t$, then $Y_t=Y_{t+}=\xi_t.$ If $t$ is such that $Y_t>\xi_t$, then $Y_t-Y_{t+}=C_t-C_{t-}=0$ (due to Remark \ref{Rmk_the_jumps_of_C} and to \eqref{RBSDE_C}). Thus, in both cases,  $Y_t=Y_{t+}$; so, $Y$ is right-continuous. \\
Moreover, the right-continuity of $Y$ combined with Remark   \ref{Rmk_the_jumps_of_C} give $C_t=C_{t-}$, for all $t$.  As $C$ is right-continuous, purely discontinuous and such that $C_{0-}=0$, we deduce $C\equiv 0$. Thus, we recover the usual formulation of RBSDE with right-continuous obstacle. 
\end{Remark}

%\begin{Remark}\label{Rmk_left_uppersemicontinuous envelope}
%\end{Remark}
%The process $A$ can be interpreted as the minimal push  which allows the solution to stay above the obstacle.

%Suppose that $f$ does not depend on $y,z,k$, that is $f(\omega,t, y,z, k(\cdot)) = f(\omega,t)$, where $f$ 
%is in $\H^2 := \H^{2,T}$. The obstacle is given by  a RCLL adapted process $\xi_.$ in ${\cal S}^2$.%, that is, such that $E(\sup_{0\leq t \leq T} |\phi_t |^2) <  \infty$.

%The following lemma  provides a link between RBSDEs whose driver does not depend on the solution  and some optimal stopping problems.
%Before presenting our general existence and uniqueness result, we  give a simple introductory example where a solution to our RBSDE (from Definition \ref{RBSDE}) can be explicitly computed. The details of our toy example are presented in the Appendix (cf. Example \ref{exe_intro}).

A simple introductory example where a solution to our RBSDE (from Definition \ref{RBSDE}) 
can be explicitly computed is presented in the Appendix (cf. Example \ref{exe_intro}).

Let us now  investigate the question of existence and uniqueness of the solution to the RBSDE defined above in the case where the driver $f$ does not depend on $y$, $z$, and $\mathpzc{k}$. %We prove the following lemma.
To this purpose, we first state a lemma which will be used in the sequel.

%MODIFICATIONS DANS LE CAS AVEC UN PROCESSUS DE POISSON: \\
%Dans la suite, $N$ désigne un processus de Poisson avec intensité $\lambda>0$, et $\tilde N$ désigne le processus compensé, i.e. $\tilde N_t=N_t-\lambda t.$ 

\begin{lemma}[A priori estimates]\label{Lemma_estimate}
%Suppose that $f^1$ and $f^2$ do not depend on $y,z$, that is $f(\omega,t, y,z) = f(\omega,t)$, where $f$ 
%is a predictable \emph{process in $\H^2 := \H^{2,T}$}. Let $(\xi_t)$ be an obstacle. 
Let $(Y^1, Z^1, k^1, A^1, C^1)\in {\cal S}^2 \times \H^2 \times \H^2_\nu \times {\cal S}^2\times {\cal S}^2$ (resp. $(Y^2, Z^2, k^2, A^2, C^2)\in {\cal S}^2 \times \H^2 \times \H^2_\nu \times {\cal S}^2\times {\cal S}^2$) be  a  solution to the RBSDE  associated with  driver $f^1(\omega, t)$ (resp. $f^2(\omega, t)$) and with obstacle $\xi$.  There exists $c>0$ such that for all $\varepsilon>0$,  for all $\beta\geq\frac 1 {\varepsilon^2}$ we have 
%\begin{equation}\label{eq_initial_Lemma_estimate}
\begin{eqnarray}\label{eq_initial_Lemma_estimate}
\|k^1-k^2\|^2_{\nu,\beta} \, \leq  \, \varepsilon^2  \|f^1-f^2\|^2_\beta\,;  &  &
\|Z^1-Z^2\|^2_\beta \, \leq \, \varepsilon^2  \|f^1-f^2\|^2_\beta\,;  \nonumber\\
  \quad \vvertiii{Y^1-Y^2}^2_\beta  &\leq& 4\varepsilon^2(1+6c^2)  \|f^1-f^2\|^2_\beta.
\end{eqnarray}
%\end{equation}

%, if we define the process $$ K_t = Y_t- Y_0 - \int^t_0 f(s, U_s, V_s ) \, ds + \int^t_0 ( Z_s   , \, d B_s ), \, 0 \leq t \leq T, $$ then 
\end{lemma}
\textbf{Proof of Lemma \ref{Lemma_estimate}:}
Let $\beta>0$ and  $\varepsilon>0$ be such that $\beta\geq\frac 1 {\varepsilon^2}$.
We  set $\tilde Y:=Y^1-Y^2$, $\tilde Z:=Z^1-Z^2$, $\tilde A:=A^1-A^2$, $\tilde C:=C^1-C^2$,  $\tilde k:=k^1-k^2$, and $\tilde f(\omega, t):=f^1(\omega, t)-f^2(\omega, t)$.
We note that $\tilde Y_T=\xi_T-\xi_T=0;$ moreover, 
$$\tilde Y_\tau=\int_{\tau}^T \tilde f(t)dt-\int_{\tau}^T  \tilde Z_t dW_t -
\int_{\tau}^T \int_{E} \tilde k_t(e) \tilde N(dt,de) +\tilde A_T-\tilde A_\tau+\tilde C_{T-} -\tilde C_{\tau-} \text{ a.s. for all }\tau\in\stopo.$$
%%$$-d\tilde Y_t=\tilde f(t) dt +d\tilde A_t+d\tilde C_{t-}-\tilde Z_t dW_t,$$ for $t\in[0,T]$. 
%

Thus we see that $\tilde Y$ is an {\em optional strong  semimartingale} in the vocabulary of \cite{Galchouk} (cf. Theorem \ref{Thm_Ito} for the definition) 
%
%Thus  we see that $\tilde Y$ is an optional (strong)  semimartingale (in the vocabulary of \cite{Galchouk})
 with decomposition 
$\tilde Y= \tilde Y_0+ M+A+B$, where
$M_t:=\int_0^t \tilde Z_s dW_s+\int_0^t \int_{E} \tilde k_s(e) \tilde N(ds,de) $, $A_t:= -\int_0^t  \tilde f(s) ds- \tilde A_t$ and 
$B_t:=-\tilde C_{t-}$ (the notation is that of Theorem \ref{Thm_Ito} and Corollary \ref{Cor_Ito} from the Appendix).
Applying Gal'chouk-Lenglart's formula (more precisely Corollary \ref{Cor_Ito}) to $\e^{\beta t}\tilde Y_t^2$ gives: almost surely, for all $t\in[0,T]$,
\begin{equation*}
\begin{aligned}
\e^{\beta T}\tilde Y_T^2&=\e^{\beta t}\tilde Y_t^2+\int_{]t,T]}\beta\e^{\beta s} (\tilde Y_{s})^2 ds
-2\int_{]t,T]} \e^{\beta s}\tilde Y_{s-}\tilde f(s) ds -2\int_{]t,T]} \e^{\beta s}\tilde Y_{s-}d\tilde A_s \\
&+2\int_{]t,T]} \e^{\beta s}\tilde Y_{s-}\tilde Z_s d W_s+
2\int_{]t,T]} \e^{\beta s}\int_E \tilde Y_{s-}\tilde k_s(e) \tilde N(ds,de) +\int_{]t,T]} \e^{\beta s} \tilde Z_s^2 ds\\
&+ \sum_{t<s\leq T}\e^{\beta s}(\tilde Y_s-\tilde Y_{s-})^2-
\int_{[t,T[} 2\e^{\beta s}\tilde Y_{s}d(\tilde C)_{s+}
+\sum_{t\leq s<T}\e^{\beta s}(\tilde Y_{s+}-\tilde Y_{s})^2.
\end{aligned}
\end{equation*}
Thus, we get (recall that $\tilde Y_T=0$): almost surely, for all $t\in[0,T]$,
\begin{equation}\label{eq0_lemma_estimate}
\begin{aligned}
\e^{\beta t}\tilde Y_t^2+ \int_{]t,T]} \e^{\beta s} \tilde Z_s^2 ds&=  -\int_{]t,T]}\beta\e^{\beta s} (\tilde Y_{s})^2 ds+
2\int_{]t,T]} \e^{\beta s}\tilde Y_{s}\tilde f(s) ds+2\int_{]t,T]} \e^{\beta s}\tilde Y_{s-}d\tilde A_s\\
&-2\int_{]t,T]} \e^{\beta s}\tilde Y_{s-}\tilde Z_s d W_s-
2\int_{]t,T]} \e^{\beta s}\int_E \tilde Y_{s-}\tilde k_s(e) \tilde N(ds,de)\\
& -\sum_{t<s\leq T}\e^{\beta s}(\tilde Y_s-\tilde Y_{s-})^2 +2\int_{[t,T[} \e^{\beta s}\tilde Y_{s}d\tilde C_{s}-\sum_{t\leq s<T}\e^{\beta s}(\tilde Y_{s+}-\tilde Y_{s})^2.%\\
%&\leq 2\int_{]0,T]} \e^{\beta s}\tilde Y_{s}\tilde f(s) ds+2\int_{]0,T]} \e^{\beta s}\tilde Y_{s-}d\tilde A_s\\
%&-2\int_{]0,T]} \e^{\beta s}\tilde Y_{s-}\tilde Z_s d W_s +2\int_{[0,T[} \e^{\beta s}\tilde Y_{s}d(\tilde C)_{s+}.
\end{aligned}
\end{equation}
%Hence, a.s. for all $t\in[0,T]$,
%\begin{equation}\label{eq0_lemma_estimate}
%\begin{aligned}
%\e^{\beta t}\tilde Y_t^2+\int_{]t,T]} \e^{\beta s} \tilde Z_s^2 ds&\leq -\int_{]t,T]}\beta\e^{\beta s} (\tilde Y_{s})^2 ds+
%2\int_{]t,T]} \e^{\beta s}\tilde Y_{s}\tilde f(s) ds+2\int_{]t,T]} \e^{\beta s}\tilde Y_{s-}d\tilde A_s\\
%&-2\int_{]t,T]} \e^{\beta s}\tilde Y_{s-}\tilde Z_s d W_s-
%2\int_{]t,T]} \e^{\beta s}\int_E \tilde Y_{s-}\tilde k_s(e) \tilde N(ds,de)\\&
% +2\int_{[t,T[} \e^{\beta s}\tilde Y_{s}d\tilde C_{s}
%-\sum_{t<s\leq T}\e^{\beta s}(\tilde Y_s-\tilde Y_{s-})^2.
%\end{aligned}
%\end{equation}
We give hereafter an upper bound for some of the terms appearing on the right-hand side (r.h.s. for short) of the above equality.

%POUR CERTAINS DES TERMES A DROITE DE L'INEGALITE PRECEDENTE, NOUS AVONS LE MEME RAISONNEMENT QU'AVANT.\\
Let us first consider the sum of \textbf{the first and the  second term } on the r.h.s. of equality  \eqref{eq0_lemma_estimate}. By applying the inequality $2ab\leq (\frac a \varepsilon)^2+\varepsilon^2 b^2$, valid for all $(a,b)\in\R^2$, we get: a.s. for all $t\in[0,T]$,
\begin{equation*}
\begin{aligned}
-\int_{]t,T]}\beta\e^{\beta s} (\tilde Y_{s})^2 ds+2\int_{]t,T]}  \e^{\beta s}\tilde Y_{s}\tilde f(s) ds&\leq -\int_{]t,T]}\beta\e^{\beta s} (\tilde Y_{s})^2 ds+\frac 1 {\varepsilon^2} \int_{]t,T]}  \e^{\beta s}\tilde Y_{s}^2ds\\
&\quad+\varepsilon^2 \int_{]t,T]}  \e^{\beta s}\tilde f^2(s) ds\\
&=(\frac 1 {\varepsilon^2}-\beta)\int_{]t,T]}\e^{\beta s} (\tilde Y_{s})^2 ds+\varepsilon^2 \int_{]t,T]}  \e^{\beta s}\tilde f^2(s) ds.
\end{aligned}
\end{equation*}
As $\beta\geq\frac 1 {\varepsilon^2}$, we have $(\frac 1 {\varepsilon^2}-\beta)\int_{]t,T]}\e^{\beta s} (\tilde Y_{s})^2 ds \leq 0$, for all $t\in[0,T]$ a.s. %By combining the above observations with equation \label{eq0_lemma_estimate}, we obtain 
%\begin{equation}\label{eq_new_lemma_estimate}
%\begin{aligned}
%\e^{\beta t}\tilde Y_t^2+\int_{]t,T]} \e^{\beta s} \tilde Z_s^2 ds &\leq \varepsilon^2 \int_{]t,T]}  \e^{\beta s}\tilde f^2(s) ds+2\int_{]t,T]} \e^{\beta s}\tilde Y_{s-}d\tilde A_s\\
%&-2\int_{]t,T]} \e^{\beta s}\tilde Y_{s-}\tilde Z_s d W_s +2\int_{[t,T[} \e^{\beta s}\tilde Y_{s}d(\tilde C)_{s+}.
%\end{aligned}
%\end{equation}

For \textbf{the third  term} on the r.h.s. of \eqref{eq0_lemma_estimate} it can be shown that, a.s. for all $t\in[0,T]$,  $2\int_{]t,T]} \e^{\beta s}\tilde Y_{s-}d\tilde A_s\leq 0$. The proof uses property \eqref{RBSDE_A} of the definition of the RBSDE and the property  $Y^i\geq \xi$, for $i=1,2$; the details are similar to those in the case of a cadlag obstacle and are left to the reader (cf., for instance,  \cite[proof of Prop. A.1]{QuenSul2}).

%By taking the expectation, we  get
%\begin{equation}\label{eq0_lemma_estimate}
%\begin{aligned}
%\beta\|\tilde Y\|^2_\beta +\|\tilde Z\|^2_\beta\leq 2E\left[\int_{]0,T]}  \e^{\beta s}\tilde Y_{s}\tilde f(s) ds\right] +
%2E\left[\int_{]0,T]} \e^{\beta s}\tilde Y_{s-}d\tilde A_s\right] +2E\left[\int_{[0,T[} \e^{\beta s}\tilde Y_{s}d(\tilde C)_{s+}\right].
%\end{aligned}
%\end{equation}
For the \textbf{the last but one term} on the r.h.s. of \eqref{eq0_lemma_estimate} we show that,  a.s. for all $t\in[0,T]$, $2\int_{[t,T[} \e^{\beta s}\tilde Y_{s}d\tilde C_{s}\leq 0.$ Indeed, a.s. for all $t\in[0,T]$,
$\int_{[t,T[} \e^{\beta s}\tilde Y_{s}d\tilde C_{s}=
\sum_{t\leq s<T} \e^{\beta s}\tilde Y_{s} \Delta \tilde C_s.$
Now, a.s. for all $s\in[0,T]$,
%\begin{equation}\label{eq1_lemma_estimate}
$\tilde Y_{s} \Delta \tilde C_s=(Y_s^1-Y_s^2)\Delta C_s^1-(Y_s^1-Y_s^2)\Delta C_s^2.$
%\end{equation}
We use property \eqref{RBSDE_C}, the non-decreasingness of (almost all trajectories of) $C^1$, and the fact that $Y^2\geq \xi$ to obtain: a.s. for all $s\in[0,T]$,
$$(Y_s^1-Y_s^2)\Delta C_s^1=(Y_s^1-\xi_s)\Delta C_s^1-
(Y_s^2-\xi_s)\Delta C_s^1= -(Y_s^2-\xi_s)\Delta C_s^1\leq 0.$$
Similarly, we obtain: a.s. for all $s\in[0,T]$,
%$$(Y_s^1-Y_s^2)\Delta C_s^2=(Y_s^1-\xi_s)\Delta C_s^2+(\xi_s-Y_s^2)\Delta C_s^2=
$(Y_s^1-Y_s^2)\Delta C_s^2=(Y_s^1-\xi_s)\Delta C_s^2\geq 0.$
We conclude that, a.s. for all $t\in[0,T]$, $2\int_{[t,T[} \e^{\beta s}\tilde Y_{s}d\tilde C_{s}$ $\leq 0$. 
%By combining the above observations with equation \eqref{eq0_lemma_estimate}, we get: a.s. for all $t\in[0,T]$,
The above observations, together with equation \eqref{eq0_lemma_estimate}, lead to the following inequality: a.s., for all $t\in[0,T]$,  
\begin{equation}\label{eq7_lemma_estimate}
\begin{aligned}
\e^{\beta t}\tilde Y_t^2+\int_{]t,T]} \e^{\beta s} \tilde Z_s^2 ds &\leq \varepsilon^2 \int_{]t,T]}  \e^{\beta s}\tilde f^2(s) ds-
2\int_{]t,T]} \e^{\beta s}\tilde Y_{s-}\tilde Z_s d W_s\\
&-2\int_{]t,T]} \e^{\beta s}\int_E \tilde Y_{s-}\tilde k_s(e) \tilde N(ds,de)-\sum_{t<s\leq T}\e^{\beta s}(\tilde Y_s-\tilde Y_{s-})^2.
%& %-\int_{]t,T]} \e^{\beta s}( 2\tilde Y_{s-}\tilde k_s+\tilde k_s^2)  d\tilde N_s.
\end{aligned}
\end{equation}
From the above inequality we derive first an estimate for $\|\tilde Z\|^2_\beta$ and 
$\|\tilde k\|^2_{\nu,\beta}$, and then an estimate for $\vvertiii{\tilde Y}^2_\beta.$
\paragraph*{Estimate for $\|\tilde Z\|^2_\beta$ and 
$\|\tilde k\|^2_{\nu,\beta}$}
Note first that we have: 
\begin{equation*}
\begin{aligned}
 \int_{]t,T]} \e^{\beta s} ||\tilde k_s||_{\nu}^2 ds-\sum_{t<s\leq T}\e^{\beta s}(\tilde Y_s-\tilde Y_{s-})^2
% &=
% \int_{]t,T]} \e^{\beta s} \tilde k_s^2 \lambda  ds-\sum_{t< s\leq T}\e^{\beta s}(\tilde k_s \Delta N_s-\Delta \tilde A_s )^2\\
& = \int_{]t,T]} \e^{\beta s} ||\tilde k_s||_{\nu}^2 ds- \int_{]t,T]} \e^{\beta s}\int_E \tilde k_s^2(e)  N(ds,de)\\
&\quad-\sum_{t< s\leq T}\e^{\beta s}(\Delta \tilde A_s) ^2\\
&=-\sum_{t\leq s<T}\e^{\beta s}\Delta \tilde A_s ^2-\int_{]t,T]} \e^{\beta s}\int_E \tilde k_s^2(e)  \tilde N(ds,de),
\end{aligned}
\end{equation*}
where, in order to obtain the first equality, we have used the fact that the processes $A_\cdot$ and $N(\cdot,de)$ "do not have jumps in common" (recall that the process $A$ jumps only at predictable stopping times, while the process $N(\cdot,de)$ does not jump at predictable stopping times).\\ 
By adding the term $\int_{]t,T]} \e^{\beta s}|| \tilde k_s||_{\nu}^2 ds$ on both sides of inequality \eqref{eq7_lemma_estimate} and by using the above computation, we derive that almost surely, for all $t\in[0,T]$, 
\begin{equation}\label{eq_last_000}
\begin{aligned}
\e^{\beta t}\tilde Y_t^2+\int_{]t,T]} \e^{\beta s} \tilde Z_s^2 ds +\int_{]t,T]} \e^{\beta s}|| \tilde k_s||_{\nu}^2 ds
&\leq \varepsilon^2 \int_{]t,T]}  \e^{\beta s}\tilde f^2(s) ds-
2\int_{]t,T]} \e^{\beta s}\tilde Y_{s-}\tilde Z_s d W_s\\
& -\int_{]t,T]} \e^{\beta s}\int_E ( 2\tilde Y_{s-}\tilde k_s(e)+\tilde k_s^2(e))  \tilde N(ds,de).
\end{aligned}
\end{equation}
Let us show that the stochastic integral "with respect to $dW_s$" has  zero expectation. Note first that  
%\begin{description}
%\item[(i)] 
\begin{equation}\label{inegalite-essup}
\sup_{t\in ]0,T]}  \tilde{Y}_{t-}^2= \sup_{t\in\Q \cap ]0,T]} \tilde Y_{t-}^2\leq \esssup_{\tau\in\stopo} 
\tilde Y_{\tau}^2 \quad {\rm a.s}
\end{equation}
 where we have used the left-continuity of the process $(\tilde {Y}_{s-})$ to obtain the  equality. 
 %
 %
%%Note first that we have, 
%\begin{description}
%\item[(i)] 
%for all $s\in(0,T]$, for a.e. $\omega\in\Omega$, 
%\begin{equation}\label{inegalite-essup}
%(\tilde{Y}_{s-}(\omega))^2\leq \sup_{t\in\Q} (\tilde Y_{t-}(\omega))^2\leq (\esssup_{\tau\in\stopo} \tilde Y_{\tau}^2)(\omega),
%\end{equation}
% where we have used the left-continuity of the process $(\tilde {Y}_{s-})$ to obtain the first inequality. 
 %
From this property together with Cauchy-Schwarz inequality, we get   $E\left[ \sqrt{\int_{0}^T \e^{2\beta s}\tilde Y^2_{s-}\tilde Z^2_s ds}\right] 
 \leq  \vvertiii{\tilde Y}_{{\cal S}^2} \|\tilde Z\|_{2\beta}<\infty.$ By standard arguments, we deduce 
$E\left[ \int_{0}^T \e^{\beta s}\tilde Y_{s-}\tilde Z_s d W_s\right] =0.$  By similar arguments, the last term on the r.h.s. of  inequality \eqref{eq_last_000}  has also zero expectation. 
 % in virtue of Property \ref{property_add} $(ii)$. % To this purpose, let us show first that $E\left[ \sqrt{\int_{0}^T \e^{2\beta s}\tilde Y^2_{s-}\tilde Z^2_s ds}\right]<\infty.$
%%%% rédaction de cette partie abrégée %%%%%

%For the first term on the r.h.s. of equation \eqref{eq0_lemma_estimate} we apply the inequality $2ab\leq (\frac a \varepsilon)^2+\varepsilon^2 b^2$, valid for all $a\in\R$, for all $b\in\R$. We get
%\begin{equation*}
%\begin{aligned}
%2E\left[\int_{]0,T]}  \e^{\beta s}\tilde Y_{s}\tilde f(s) ds\right]\leq \frac 1 {\varepsilon^2} E\left[\int_{]0,T]}  \e^{\beta s}\tilde Y_{s}^2ds\right]+\varepsilon^2 E\left[\int_{]0,T]}  \e^{\beta s}\tilde f^2(s) ds\right]=
%\frac 1 {\varepsilon^2}\|\tilde Y\|^2_\beta +\varepsilon^2\|\tilde f\|^2_\beta
%\end{aligned}
%\end{equation*}
%Thus, after taking expectations in \eqref{eq7_lemma_estimate}, we obtain
%$$    $$
By applying \eqref{eq_last_000} with $t=0$, and by taking expectations on both sides of the resulting inequality, we obtain
$\tilde Y_0^2+ \|\tilde Z\|^2_\beta+\|\tilde k\|^2_{\nu,\beta}\leq \varepsilon^2\|\tilde f\|^2_\beta.$          
We deduce the desired estimates:
\begin{equation}\label{eq8_lemma_estimate} 
\|\tilde Z\|^2_\beta\leq \varepsilon^2\|\tilde f\|^2_\beta \text{  and  } \|\tilde k\|^2_{\nu,\beta}\leq \varepsilon^2\|\tilde f\|^2_\beta.          
\end{equation}

%DIFFERENCE PAR RAPPORT AU CAS BROWNIEN:\\
\paragraph*{Estimate for $\vvertiii{\tilde Y}^2_\beta$}
%From inequality \eqref{eq7_lemma_estimate} we derive that a.s., for all $t\in[0,T]$,
%\begin{equation}\label{eq7a_lemma_estimate}
%\e^{\beta t}\tilde Y_t^2 \leq \varepsilon^2 \int_{]t,T]}  \e^{\beta s}\tilde f^2(s) ds-
%2\int_{]t,T]} \e^{\beta s}\tilde Y_{s-}\tilde Z_s d W_s-2\int_{]t,T]} \e^{\beta s}\int_E \tilde Y_{s-}\tilde k_s(e) \tilde N(ds,de).
%\end{equation}
%%%%%%%%%%%%
%Hence, for all $\tau\in\stopo$, 
%
From inequality \eqref{eq7_lemma_estimate} we derive that a.s., for all $t\in[0,T]$,
\begin{equation}\label{eq7a_lemma_estimate}
\e^{\beta t}\tilde Y_t^2 \leq \varepsilon^2 \int_{]t,T]}  \e^{\beta s}\tilde f^2(s) ds-
2\int_{]t,T]} \e^{\beta s}\tilde Y_{s-}\tilde Z_s d W_s-2\int_{]t,T]} \e^{\beta s}\int_E \tilde Y_{s-}\tilde k_s(e) \tilde N(ds,de).
\end{equation}
%%%%%%%%%%%
From this, together with Chasles' relation for  stochastic integrals, we get, for all $\tau\in\stopo$, 
\begin{equation*}
\begin{aligned}
\e^{\beta \tau}\tilde Y_\tau^2 &\leq \varepsilon^2 \int_{]0,T]}  \e^{\beta s}\tilde f^2(s) ds 
-2\int_{]0,T]} \e^{\beta s}\tilde Y_{s-}\tilde Z_s d W_s+2\int_{]0,\tau]} \e^{\beta s}\tilde Y_{s-}\tilde Z_s d W_s\\
&-2\int_{]0,T]} \e^{\beta s}\int_E \tilde Y_{s-}\tilde k_s(e) \tilde N(ds,de)+2\int_{]0,\tau]} \e^{\beta s}
\int_E \tilde Y_{s-}\tilde k_s(e) \tilde N(ds,de) \text{ a.s.}
\end{aligned}
\end{equation*}

%
%Using the non negativity of the terms $\int_{]t,T]} \e^{\beta s} \tilde Z_s^2 ds$ and $\sum_{t<s\leq T}\e^{\beta s}(\tilde Y_s-\tilde Y_{s-})^2$ as well as  
%Chasles' relation for the stochastic integrals, we derive from the inequality \eqref{eq7_lemma_estimate} that 
%
By taking first the essential supremum over $\tau\in\stopo$ and then the expectation on both sides of the above inequality, we obtain   
\begin{equation}\label{eq9_lemma_estimate} 
E[\esssup_{\tau \in \stopo}\e^{\beta \tau}\tilde Y_\tau^2]\leq \varepsilon^2 \|\tilde f\|^2_\beta  
+2E[\esssup_{\tau \in \stopo}|\int_0^\tau \e^{\beta s}\tilde Y_{s-}\tilde Z_s d W_s|]+
2E[\esssup_{\tau \in \stopo}|\int_{]0,\tau]} \e^{\beta s}\int_E \tilde Y_{s-}\tilde k_s(e) \tilde N(ds,de)|].
\end{equation}
%By using the continuity of a.e. trajectory of the process $(\int_0^t \e^{\beta s}\tilde Y_{s-}\tilde Z_s d W_s)_{t\in[0,T]}$ (cf. 
%From Prop. \ref{esssup_cadlag_process} and Burkholder-Davis-Gundy inequalities (applied with $p=1$), we get 
%\begin{equation}\label{eq10_lemma_estimate} 
%2E[\esssup_{\tau \in \stopo}|\int_0^\tau \e^{\beta s}\tilde Y_{s-}\tilde Z_s d W_s|]=
%2E[\sup_{t \in [0,T]}|\int_0^t \e^{\beta s}\tilde Y_{s-}\tilde Z_s d W_s|]\leq 
%2cE\left[ \sqrt{\int_{0}^T \e^{2\beta s}\tilde Y^2_{s-}\tilde Z^2_s ds}\right],
%\end{equation}
%where $c$ is a positive "universal" constant (which does not depend on the other parameters).
%%The same reasoning as that used to obtain equation \eqref{eq_lemma_estimate_intermediate}  leads to
%We have 
%$\sqrt{\int_{0}^T \e^{2\beta s}\tilde Y^2_{s-}\tilde Z^2_s ds}\leq 
%\sqrt{\esssup_{\tau\in\stopo} \e^{\beta\tau}(\tilde Y_{\tau})^2\int_{0}^T \e^{\beta s}\tilde Z^2_s ds}\text{ a.s.,} $ where we have used Remark   \ref{property_add}.   
%Combining the two previous inequalities with the inequality $ab\leq \frac 1 2 a^2+\frac 1 2 b^2$   gives
%\begin{equation}\label{eq30_lemma_estimate} 
%2E[\esssup_{\tau \in \stopo}\int_0^\tau \e^{\beta s}\tilde Y_{s-}\tilde Z_s d W_s]\leq 
%\frac 1 2 E[\esssup_{\tau\in\stopo} \e^{\beta\tau}(\tilde Y_{\tau})^2]+ \frac 1 2 4c^2E[\int_{0}^T \e^{\beta s}\tilde Z^2_s ds]=\frac 1 2 \vvertiii{\tilde Y}^2_\beta+ 2c^2\|\tilde Z\|^2_{\beta}.
%\end{equation}
%By using similar arguments, 
Let us  consider the last term in \eqref{eq9_lemma_estimate}.
By using Remark
\ref{esssup_cadlag_process} 
applied to the right-continuous process  
$(\int_{]0,t]} \e^{\beta s}\int_E \tilde Y_{s-}\tilde k_s(e) \tilde N(ds,de))_{t\in[0,T]}$ 
  and Burkholder-Davis-Gundy inequalities, we get
\begin{equation}\label{eq20_lemma_estimate} 
2E[\esssup_{\tau \in \stopo}|\int_0^\tau \e^{\beta s}\int_E \tilde Y_{s-}\tilde k_s(e) \tilde N(ds,de) |]
\leq 
2cE\left[ \sqrt{\int_{0}^T \e^{2\beta s}\int_E \tilde Y_{s-}^2\tilde k_s^2(e) N(ds,de) }\right],
\end{equation}
where $c>0$ is a positive "universal" constant (which does not depend on the other parameters).  %The same reasoning again as that used to obtain equation \eqref{eq_lemma_estimate_intermediate}

The inequality \eqref{inegalite-essup}   and the trivial inequality $ab\leq \frac 1 2 a^2+\frac 1 2 b^2$ lead to  

%
%\begin{equation}\label{add_eq_550}
\begin{align}\label{add_eq_550}
2cE\left[ \sqrt{\int_{0}^T \e^{2\beta s}\int_E \tilde Y_{s-}^2\tilde k_s^2(e) N(ds,de) }\right] \leq 
E\left[ \sqrt{\frac 1 2\esssup_{\tau\in\stopo}\e^{\beta \tau}\tilde Y_{\tau}^2}\sqrt{8c^2
\int_0^T \e^{\beta s}\int_E \tilde k_s^2(e)N(ds,de) }\right] \nonumber\\
\leq  \frac 1 4 E \left[ \esssup_{\tau\in\stopo}\e^{\beta \tau}\tilde Y_{\tau}^2 \right]+4c^2 E \left[ \int_0^T \e^{\beta s}\int_E \tilde k_s^2(e)N(ds,de) \right]\nonumber
= \frac 1 4 \vvertiii{\tilde Y}^2_\beta + 
%4c^2 E[\int_0^T \e^{\beta s}\int_E \tilde k_s^2(e) \tilde N(ds,de)]+ 
 4c^2\|\tilde k\|^2_{\nu,\beta}.
\end{align}
%\end{equation}
%From this, together with \eqref{eq20_lemma_estimate}, we get
%\begin{equation}\label{eq21_lemma_estimate}
%\begin{aligned}
%2E[\esssup_{\tau \in \stopo}|\int_0^\tau \e^{\beta s}\int_E \tilde Y_{s-}\tilde k_s(e) \tilde N(ds,de) |]
%&\leq \frac 1 4 E[\esssup_{\tau \in \stopo}\e^{\beta \tau}\tilde Y_{\tau}^2]+
%4c^2  E[\int_0^T \e^{\beta s}\int_E \tilde k_s^2(e)N(ds,de)]\\
%&= \frac 1 4 \vvertiii{\tilde Y}^2_\beta + 4c^2 E[\int_0^T \e^{\beta s}\int_E \tilde k_s^2(e) \tilde N(ds,de)]+ 
% 4c^2\|\tilde k\|^2_{\nu,\beta}, 
% \end{aligned} 
%\end{equation}
Here, the   equality has been obtained by adding and subtracting $4c^2\|\tilde k\|^2_{\nu,\beta}$ (on the left-hand side) and 
by using the fact that $E[\int_0^T \e^{\beta s}\int_E \tilde k_s^2(e) \tilde N(ds,de)]=0$.  
By using similar arguments, we obtain that the last but one term in \eqref{eq9_lemma_estimate} satisfies 
\begin{equation}\label{eq30_lemma_estimate} 
2E[\esssup_{\tau \in \stopo}\int_0^\tau \e^{\beta s}\tilde Y_{s-}\tilde Z_s d W_s]\leq 
\frac 1 2 \vvertiii{\tilde Y}^2_\beta+ 2c^2\|\tilde Z\|^2_{\beta},
\end{equation}
where $c$ is the same universal constant as above. 
By \eqref{eq9_lemma_estimate}, we thus derive that
% by combining equations \eqref{eq21_lemma_estimate} and \eqref{eq30_lemma_estimate} with \eqref{eq9_lemma_estimate}, we obtain 
$\frac 1 4 \vvertiii{\tilde Y}^2_\beta \leq \varepsilon^2 \|\tilde f\|^2_\beta+2c^2  \|\tilde Z\|^2_\beta+ 4c^2\|\tilde k\|^2_{\nu,\beta}.$
This inequality, together with the estimates from  \eqref{eq8_lemma_estimate}, gives
$ \vvertiii{\tilde Y}^2_\beta \leq 4\varepsilon^2(1+6c^2) \|\tilde f\|^2_\beta,$
which is the desired result.   
  {\hspace*{.1pt}\hspace*{\fill}\BOX\vskip\baselineskip}     
%This inequality, combined with the estimate \eqref{eq8_lemma_estimate}  on $\|\tilde Z\|^2_\beta$, gives
%$$\vvertiii{\tilde Y}^2_\beta \leq 2\varepsilon^2(1+2c^2)  \|\tilde f\|^2_\beta.$$  
  %%%%%%%%%%%ù
%%%By combining the above observations with equation \eqref{eq0_lemma_estimate}, we get
%%%%$$\beta\|\tilde Y\|^2_\beta +\|\tilde Z\|^2_\beta\leq \frac 1 {\varepsilon^2}\|\tilde Y\|^2_\beta +\varepsilon^2\|\tilde f\|^2_\beta.$$
%%%%We conlude that equation \eqref{eq_initial_Lemma_estimate} holds for all $\beta\geq 1+ \frac 1 {\varepsilon^2}$, %%%%%which is the desired conclusion.
%\end{proof}

In the following lemma, we prove existence and uniqueness of the solution to the RBSDE from Definition \ref{def_solution_RBSDE} (in the case where the driver $f$ does not depend on $y$, $z$ and $\mathpzc{k}$) and we characterize the first component of the solution as the "value process" of an optimal stopping problem.
\begin{lemma}\label{f}
Suppose that $f$ does not depend on $y,z,\mathpzc{k}$, that is $f(\omega,t, y,z,\mathpzc{k}) = f(\omega,t)$, where $f$ 
is a progressive process with $E[\int_0^T f(t)^2dt] < + \infty$. Let $(\xi_t)$ be an obstacle. %cadlag adapted process in ${\cal S}^2$.
Then, the RBSDE from Definition \ref{def_solution_RBSDE} admits a unique solution $(Y,Z,k,A,C)\in {\cal S}^2 \times \H^2 \times \H^2_\nu\times {\cal S}^2\times {\cal S}^2$, and  for each $S \in \stopo$,  we have 
\begin{eqnarray}\label{deux-2}
 Y_S= \esssup_{\tau \in\stops } E[ \xi_{\tau} + \int_S^\tau f(t)dt \mid \Fc_S] \quad   \rm{a.s.}
\end{eqnarray} 
Moreover, the following property holds:
\begin{equation}\label{max}
Y_S=\xi_S\vee Y_{S+} \quad   \rm{a.s.}
\end{equation}
We also have  $Y_{S+}= \esssup_{\tau >S} E[ \xi_{\tau} + \int_S^\tau f(t)dt \mid \Fc_S] \quad   \rm{a.s.}$, for all $S \in \stopo$.\\
 If, furthermore, the obstacle $(\xi_t)$ is l.u.s.c.   along stopping times, then $(A_t)$ is continuous.
 %Moreover,  for each $\varepsilon >0$ and each $S \in \stopo$, the stopping time $\tau^{\varepsilon}_S$ defined by
%\begin{equation}\label{eq_tau_epsilon_S}
%\tau^{\varepsilon}_S:= \inf \{ t \geq S,\,\,Y_t \leq \xi_t  + \varepsilon\}
%\end{equation}
% is $\varepsilon$-optimal for \eqref{deux-2}, that is 
%$$Y_S \leq E[ \xi_{\tau^{\varepsilon}_S} + \int_S^{\tau^{\varepsilon }_S}f(u)du \mid \Fc_S]  + \varepsilon  \quad   \rm{a.s.}
%$$
\end{lemma}

%\begin{remark}
%{\bf A simple example, which illustrates this result is given in the Appendix.}
%\end{remark}

%\begin{Remark}\label{Rmk_A_continuous}
%If the obstacle $(\xi_t)$ is l.u.s.c.   along stopping times, then $(A_t)$ is continuous (cf., for instance, the last statement in Thm. 20 of \cite[page 429]{DM2}, or \cite{Kob}).
%\end{Remark}
%%%%%%%%%%%%
The proof of the lemma is divided in several steps. First, we exhibit a "natural candidate" $\barY$ to be the first component of the  solution to the RBSDE with parameters $(f, \xi)$; we prove that $\barY$ belongs to the space ${\cal S}^{2}$ and we give an estimate of $\vertiii{\barY}_{{\cal S}^2}^2$ in terms of $\vertiii{\xi}_{{\cal S}^{2}}^2$ and  $\| f\|^2_{\H^{2}}$. % where for simplicity, $\|f\|_{\H^2}^2$ denotes $E[\int_0^T f(t)^2dt]$, even if the process $f$ is not necessarily predictable. 
In the second step,  we exhibit "natural candidates" for the processes $A$ and $C$, and a "natural candidate" $M$ for the martingale part of the solution to the RBSDE with parameters $(f, \xi)$. In the third step, we prove that the processes $A$ and $C$ belong to ${\cal S}^{2}$ and we give an estimate of $\vertiii{A+C_-}_{{\cal S}^2}^2$. In the fourth step, we apply the martingale representation theorem to $M$, which gives the second component $Z\in\H^2$ and the third component $k\in\H^2_\nu$ of the solution. In the fifth step, we show the uniqueness of the solution. 
Finally, we prove property \eqref{max}
 and the last two assertions of the lemma.

%As a by-product of this last step we obtain the $\varepsilon$-optimality of the stopping time $\tau^{\varepsilon}_S$ defined in \eqref{eq_tau_epsilon_S}.

%In the first step, we exhibit a "natural candidate" $(\barY, Z, A, C)$ to be the solution to the RBSDE with parameters $(f, \xi)$. In the second step, we prove that $\barY$ belongs to the space ${\cal S}^{p}$ and we give an estimate of $\vertiii{\barY}_{{\cal S}^p}^p$ in terms of $\vertiii{\xi}_{{\cal S}^{p}}^p$ and  $\| f\|^p_{\H^{p}}$. In the third step, we prove that the processes $A$ and $C$ belong to ${\cal S}^{p}$ and we give an estimate of $\vertiii{A}_{{\cal S}^p}^p$ and $\vertiii{C}_{{\cal S}^p}^p$. In the fourth step, we prove that $Z\in\H^p$ and, again,  we provide an estimate of its norm. In the final step, we show the uniqueness of the solution. As a by-product of this last step we obtain the $\varepsilon$-optimality of the stopping time $\tau^{\varepsilon}_S$ defined in \eqref{eq_tau_epsilon_S}.

\begin{proof}
For $S \in \stopo$, we define $\overline Y  (S)$ by 
\begin{eqnarray}\label{ddeux-2}
\overline Y  (S):= \esssup_{\tau \in \stops} E[ \xi_{\tau} + \int_S^\tau f(u)du \mid \Fc_S].
\end{eqnarray} 
%and $\dbarY(S)$ by 
%\begin{equation}\label{Lemma_eq_numero_ajoute}
%\dbarY(S):=\barY(S)+\int_0^S f(u)du=\esssup_{\tau \in \stops} E[ \xi_{\tau} + \int_0^\tau f(u)du \mid \Fc_S].
%\end{equation}
%We note that the process $(\xi_t+\int_0^{t} f(u)du)_{t\in[0,T]}$ is progressive. Therefore, the family $(\dbarY(S))_{S\in\stopo}$ is a supermartingale family (cf. \cite[Remark 1.2 with Prop. 1.5]{Kob}). This observation, combined with \cite[Remark (b), page 435]{DM2}, gives the existence of a strong optional supermartingale (which we denote again by $\dbarY$) such that $\dbarY_S=\dbarY(S)\, \rm{ a.s.}$ for all $S\in\stopo$.  
By Proposition \ref{pro} in the Appendix, there exists a  ladlag optional  process $(\barY_t)_{t\in[0,T]}$ which \emph{aggregates} the family $(\barY(S))_{S\in\stopo},$ that is, 
\begin{equation}\label{eq_aggreg}
\barY_S=\barY(S) \text{ a.s. for all } S\in\stopo.
\end{equation}

%strong optional ladlag supermartingale  (which we denote again by $\dbarY$) such that $\dbarY_S=\dbarY(S)\, \rm{ a.s.}$ for all $S\in\stopo$, and we also have 
%$\dbarY_S= (\xi _S + \int_0^S f(u)du)\vee \dbarY_{S+} \quad {\rm a.s.},$  for all $S$.
%Thus, we have $\barY(S)=\dbarY(S)-\int_0^S f(u) du=\dbarY_S-\int_0^S f(u) du\, \rm{ a.s.} $ for all $S\in\stopo$.  
%%On the other hand, we know that almost all trajectories of the strong optional supermartingale $\dbarY$ are ladlag (cf. \cite{DM2}). 
%Thus, we get that the ladlag optional  process $(\barY_t)_{t\in[0,T]}=(\dbarY_t-\int_0^t f(u) du)_{t\in[0,T]}$ aggregates the family $(\barY(S))_{S\in\stopo}, $ {\bf and that $\barY_S=\xi_S\vee \barY_{S+}$ a.s. for all $S\in\stopo$.
%}
%We  give a short proof of the existence based on some fine results of Optimal Stopping theory.
%We introduce the following optimal stopping problem.
\paragraph{Step 1} 
By using  %the definition of $\barY$ (cf. \eqref{ddeux-2}), 
Jensen's inequality  and the triangular inequality,  we get 
\begin{equation}\label{eq_inequality_Y_X}
\begin{aligned}
|\barY_S|&\leq \esssup_{\tau \in \stops} E[| \xi_{\tau}| + |\int_S^\tau f(u)du| \mid \Fc_S]\leq E[\esssup_{\tau \in \stops} | \xi_{\tau}| + \int_0^T |f(u)|du \mid \Fc_S]=E[ X|\Fc_S],
\end{aligned}
\end{equation} 
%%%%%%
%%%%%%%
%%%%  partie raccourcie%%%%%
%\begin{equation*}
%\begin{aligned}
%|\barY_S|&=|\esssup_{\tau \in \stops} E[ \xi_{\tau} + \int_S^\tau f(u)du \mid \Fc_S]|\leq 
%\esssup_{\tau \in \stops} E[| \xi_{\tau} + \int_S^\tau f(u)du| \mid \Fc_S]\leq\\
%&\leq 
%\esssup_{\tau \in \stops} E[| \xi_{\tau}| + |\int_S^\tau f(u)du| \mid \Fc_S].
%\end{aligned}
%\end{equation*} 
%
%Hence,
%\begin{equation*}
%\begin{aligned} 
%|\barY_S|&\leq \esssup_{\tau \in \stops} E[| \xi_{\tau}| + \int_S^\tau |f(u)|du \mid \Fc_S]\leq
% \esssup_{\tau \in \stops} E[| \xi_{\tau}| + \int_0^T |f(u)|du \mid \Fc_S]\\
%&\leq  E[\esssup_{\tau \in \stops} | \xi_{\tau}| + \int_0^T |f(u)|du \mid \Fc_S].
%\end{aligned}
%\end{equation*} 
%where, in the first line, we have used Jensen's inequality, and, in the last line, we have exchanged essential supremum and  conditional expectation  in virtue of... in \cite{Neveu}.
%%Thus, we obtain
%Hence, 
%\begin{equation}\label{eq_inequality_Y_X}
%\begin{aligned} 
%|\barY_S|&\leq E[ X|\Fc_S],
%\end{aligned}
%\end{equation} 
a.s., for all $S\in\stopo$, where we have set 
\begin{equation}\label{eq_def_X}
X:=  \int_0^T |f(u)|du+\esssup_{\tau \in \stopo} | \xi_{\tau}|.
\end{equation}
We apply Cauchy-Schwarz inequality to obtain %We note that $X\in L^2 $ and that  %Indeed, $E[X^p]\leq C_pE[(\int_0^T |f(u)|du)^p+\esssup_{\tau \in \stopo} | \xi_{\tau}|^p],$ where $C_p$ is a positive constant depending on $p$. Applying Cauchy-Schwarz inequality gives
%$E[X^p]\leq C_pE[T^{p/2}(\int_0^T |f(u)|^2du)^{p/2}+\esssup_{\tau \in \stopo} | \xi_{\tau}|^p],$ i.e.
\begin{equation}\label{eq_Lp-estimate_X}
E[X^2]\leq c T\|f\|_{\H^2}^2+c\vertiii{\xi}^2_{{\cal S}^{2}},
\end{equation}
where
$c>0$ is a positive constant, which, in the sequel, is allowed to differ from line to line.   From \eqref{eq_inequality_Y_X},    %leads to
%$|\barY_S|^2\leq |E[ X|\Fc_S]|^2.$ By taking the essential supremum over $S\in\stopo$, 
we get
$\esssup_{S\in\stopo}|\barY_S|^2\leq \esssup_{S\in\stopo}|E[ X|\Fc_S]|^2=\sup_{t\in[0,T]}|E[ X|\Fc_t]|^2,$
where the equality follows from the right-continuity of the process $(E[ X|\Fc_t])_{0\leq t\leq T}$, together with  Remark \ref{esssup_cadlag_process}, %By using Proposition \ref{esssup_cadlag_process} of the Appendix, we get
%$\esssup_{S\in\stopo}|\barY_S|^2\leq \sup_{t\in[0,T]}|E[ X|\Fc_t]|^2.$ 
By using  this and Doob's martingale inequalities in $L^2$, we obtain %(cf., for instance, \cite[Theorem I.20]{Pro05})
\begin{equation}\label{eq_estimate_Y_f_xi}
E[\esssup_{S\in\stopo}|\barY_S|^2]\leq E[\sup_{t\in[0,T]}|E[ X|\Fc_t]|^2]\leq cE[X^2]\leq cT\|f\|_{\H^2}^2+c\vertiii{\xi}^2_{{\cal S}^{2}}, 
\end{equation}
where the last inequality follows from \eqref{eq_Lp-estimate_X}.
%%%%
%%where $c$ is a positive constant different from the above one.
%Finally, combining inequalities \eqref{eq_inequality_Y_X_second} and \eqref{eq_Lp-estimate_X} gives
%\begin{equation}\label{eq_estimate_Y_f_xi}
%E[\esssup_{S\in\stopo}|\barY_S|^2]\leq cT\|f\|_{\H^2}^2+c\vertiii{\xi}^2_{{\cal S}^{2}}.
%\end{equation}
%%where we have again allowed  the positive constant $c$ to differ from the above ones.
%

%%%%%%%%%%%%%%%%%%%%%%%%%%%%%%%%%%%%%%%%
\paragraph{Step 2}  By Proposition \ref{pro}, the process $(\barY_t + \int_0^t f(u) du)_{t\in[0,T]}$ is a strong supermartingale. Due to the previous step and to the assumption $f\in\H^2$, it is of class (D). Applying
Mertens decomposition (cf. Theorem \ref{thm_Mertens_decomposition}) % cf; also DM1 Thm 20 page 429 combined with Rk 3b page 205
%%%%%%%%%%%%
%%%%
  gives the following
\begin{equation}\label{eq_Mertens_decomposition_Y}
\barY_t=-\int_0^t f(u) du+M_t-A_t-C_{t-} \text{ for all $t\in[0,T]$ a.s.}, 
\end{equation}
where $M$ is a cadlag uniformly integrable martingale, $A$ is a nondecreasing right-continuous predictable process 
such that $ A_0= 0$, $E(A_T)<\infty$, and $C$ is a nondecreasing right-continuous adapted purely discontinuous process such that  $C_{0-}= 0$, $E(C_T)<\infty$.  
Let $\tau \in\stopo$. 
%We set $\Delta^+ \barY_{\tau}:= \barY_{\tau^+} - \barY_{\tau}$.
%Since $\barY_\tau=\xi_\tau \vee \barY_{\tau+}$ a.s.\,, it follows that 
%$\Delta^+ \barY_{\tau}= {\bf 1}_{\{\barY_\tau= \xi_\tau\}}\Delta^+ \barY_{\tau}$a.s.\, 
By Remark \ref{sautY}, 
$\Delta_+ \barY_{\tau}= {\bf 1}_{\{\barY_\tau= \xi_\tau\}}\Delta_+ \barY_{\tau}$ a.s.\, Now, by \eqref{eq_Mertens_decomposition_Y}, $\Delta C_{\tau}= - \Delta_+ \barY_{\tau}$ a.s. It follows that 
$\Delta C_{\tau}= {\bf 1}_{\{\barY_\tau= \xi_\tau\}}\Delta C_{\tau}$ a.s.\,
In other terms,  the process $C$ satisfies the minimality condition \eqref{RBSDE_C} (with $Y$ replaced by $\barY$).  
Moreover, thanks to a result from optimal stopping theory due to El Karoui %for strong optional supermartingales 
(cf. \cite[Prop. 2.34]{EK};   cf. also \cite{Kob}),  for each predictable stopping time $\tau$, we have 
$\Delta A_{\tau}= {\bf 1}_{\{\barY_{\tau-}= \, \xi_{\tau-}\}}\Delta A_{\tau}$ a.s.\, For the continuous part $A^c$ of $A$, again by a result from optimal stopping theory (cf. \cite{ERRATUM}), we have 
$\int_0^T {\bf 1}_{\{\barY_t > \xi_t\}} dA^c_t = 0$ a.s.\, The process $A$ thus satisfies the minimality condition 
\eqref{RBSDE_A} (with $Y$ replaced by $\barY$). 
%Due to a result from optimal stopping theory %for strong optional supermartingales 
%(cf. \cite[Prop. 2.34]{EK}  or \cite{Kob}),  the process $A$ satisfies the minimality condition \eqref{RBSDE_A}   and the process $C$ satisfies the minimality condition \eqref{RBSDE_C}. 
We have $\barY_T=\barY(T)= \xi_T\,$ a.s.  (due to \eqref{ddeux-2} and \eqref{eq_aggreg}).
 % Combining this with equation \eqref{eq_Mertens_decomposition_Y} gives equation \eqref{RBSDE}.   
Also, from \eqref{ddeux-2} and \eqref{eq_aggreg}, we have $\barY_S=\barY(S)\geq \xi_S \,$ a.s. 
 for all $S\in\stopo$, which, along with a classical result of the general theory of processes 
 (cf. \cite[Theorem IV.84]{DM1}) implies that
  $\barY_t\geq \xi_t,$ $0 \leq t \leq T$, a.s. 
% Proposition \ref{optional_section}, shows that $\barY$ satisfies inequality  \eqref{RBSDE_inequality_barrier}. %In order to conclude that the process $(\barY,Z,A,C)$ is a solution to the RBSDE with parameters $(f, \xi)$, it remains to show that $(A,C)$ belongs to the space ${\cal S}^2\times {\cal S}^2$ and to apply the martingale representation theorem to the martingale $M$, which we do in the following steps.
\paragraph{Step 3} Let us consider the {\em Mertens process} associated with the strong supermartingale $\barY_\cdot  + \int_0^\cdot f(u) du$, that is the process $(A_t + C_{t-})$, where the processes $(A_t)$ and $(C_{t-})$ are given by \eqref{eq_Mertens_decomposition_Y}.  We show that $A_T + C_{T-} \in L^2$. By arguments similar to those used in the proof of  \eqref{eq_inequality_Y_X}, we see that 
$|\overline Y_S+ \int_0^S f(u) du|\leq E[ X|\Fc_S]$,
where $X$ is the random variable defined in  \eqref{eq_def_X}.
This observation, together with a result from potential theory (cf. Corollary \ref{corollary_DM}), gives $E\left[\left(A_T + C_{T-}\right)^2\right]\leq c E[X^2],$ where $c>0$.
%Let us define the process $K$ as the sum of the two non-decreasing processes of  Mertens decomposition of the strong supermartingale $\barY  + \int_0^\cdot f(u) du$. More precisely, we set
%$K_t:=A_t+C_{t-},$
%where the processes $(A_t)$ and $(C_{t-})$ are given by \eqref{eq_Mertens_decomposition_Y}. We show that $K\in{\cal S}^{2}$. %(We emphasize that the process $\dbarA$ is neither left-continuous nor right-continuous.) 
%By arguments similar to those used in the proof of  \eqref{eq_inequality_Y_X}, we see that 
%$|\overline Y_S+ \int_0^S f(u) du|\leq E[ X|\Fc_S]$,
%where $X$ is the random variable defined in  \eqref{eq_def_X}.
%{\bf This observation, together with an important result of the Potential Theory (see Corollary \ref{corollary_DM}), gives $E\left[\left(K_T\right)^2\right]\leq c E[X^2],$ where $c>0$.} 
By combining this inequality with inequality \eqref{eq_Lp-estimate_X} , we obtain
\begin{equation}\label{eq_norm_A_T}
E\left[\left(A_T + C_{T-}\right)^2\right]\leq cT\|f\|_{\H^2}^2+c\vertiii{\xi}^2_{{\cal S}^{2}}, 
\end{equation}
%\begin{equation}\label{eq_norm_A_T}
%E\left[\left(K_T\right)^2\right]\leq cT\|f\|_{\H^2}^2+c\vertiii{\xi}^2_{{\cal S}^{2}}, 
%\end{equation}
%$E\left[\left(K_T\right)^p\right]\leq C_pT^{p/2}\|f\|_{\H^p}^p+C_p\vertiii{\xi}^p_{{\cal S}^{p}}, $
 where we have again allowed the positive constant $c$ to vary from line to line. 
%$ \vertiii{K}^p_{{\cal S}^{p}} =E\left[\left(K_T\right)^p\right]$
We conlude that $A_T + C_{T-} \in L^2$. Hence, $A_T$  and $C_{T-}(= C_T)$ are square integrable, which, due to the nondecreasingness of $A$ and $C$, is equivalent to $A \in\mathcal{S}^2$ and $C \in\mathcal{S}^2$.
%the process $K$ belongs to ${\cal S}^{2}$ %(cf. Remark \ref{Rmk_norms_increasing_process})
\paragraph{Step 4} 
The martingale $M$ from the decomposition \eqref{eq_Mertens_decomposition_Y} belongs to $\mathcal{S}^2$; this is a consequence of Step 1., Step 3., and the fact that the process $(\int_0^t f(u) du)_{t\in[0,T]}$ is in $\mathcal{S}^2$ (since $f\in\H^2$). By the  martingale representation theorem  (cf., e.g., Lemma 2.3 in \cite{Tang}) there exists a unique predictable process $Z\in\H^2$ and a unique \emph{predictable} $k\in\H^2_\nu$ such that  
$dM_t = Z_t  dW_t + \int_{ E} k_t(e) \tilde{N}(dt,de).$ Combining this step with the previous ones gives that $(\overline Y,Z,k,A,C)$ is a solution to the RBSDE with parameters $f$ and $\xi$. 
\paragraph{Step 5}
Let us now prove the uniqueness of the solution. 
Let  $Y$ be the first component of a  solution to the RBSDE with driver $f$ and obstacle $\xi$. 
Then, by  the previous Lemma \ref{Lemma_estimate} (applied with $f^1=f^2=f$) we obtain  $Y=\overline{Y}$ in $\mathcal{S}^2$, where $\overline{Y}$ is given by \eqref{ddeux-2}. The uniqueness of the other components follows  from the uniqueness of Mertens 
decomposition of strong optional supermartingales and from the uniqueness of the martingale representation. 
(We note that the uniqueness of the second and the third component can be obtained also by applying the previous Lemma \ref{Lemma_estimate}.) 
\paragraph{Step 6} 
Property \eqref{max}  and the characterization of $Y_{S+}$
 as the value function of an optimal stopping problem follow from
 Proposition \ref{pro}
%%part (ii). 
 parts (ii) and (iii).  
The last assertion of Lemma \ref{f} follows from classical results 
(cf., for instance, the last statement in Thm. 20 of \cite[page 429]{DM2}, or \cite{Kob}).
\end{proof}

With the help of the previous two lemmas, we now prove the existence and uniqueness of the solution to the RBSDE from  Definition  \ref{def_solution_RBSDE} in the case of a general Lipschitz driver. % $f$ by using a fixed-point theorem in an appropriate Banach space. 
\begin{theorem}[Existence and uniqueness of the solution]\label{exiuni}
Let  $\xi$  be a left-limited and r.u.s.c. 
%along stopping times
 process in $\mathcal{S}^2$ and let $f$ be a  Lipschitz driver. 
The RBSDE with parameters $(f,\xi)$ from Definition \ref{def_solution_RBSDE} admits a unique solution $(Y,Z,k,A,C)\in \mathcal{S}^2  \times \H^2 \times \H^2_\nu \times \mathcal{S}^2\times \mathcal{S}^2.$\\
Moreover, for all $S\in\stopo$, we have
\begin{equation}\label{yy+}
  Y_S=\xi_S\vee Y_{S+} \quad   \rm{a.s.}
  \end{equation}
Furthermore, if $(\xi_t)$ is  assumed  l.u.s.c. along stopping times, then $(A_t)$ is continuous (or equivalently, the process $(Y_t)$ is l.u.s.c. along stopping times).

\end{theorem}

\begin{Remark} We will see that, as in the right-continuous case, the existence and uniqueness result follows from a fixed point theorem applied in 
an appropriate Banach space. In the right-continuous case, the Banach space  is classically 
the product space $\H^2 \times \H^2\times \H^2_\nu$ equipped with the norm 
 $ \| Y\|_{\beta}^2  +   \| Z\|_{\beta}^2+\|k \|_{\nu,\beta}^2$ (cf, e.g. \cite{ElKaroui97}, \cite{HO2},
 \cite{QuenSul2}). However, this Banach space does not suit our purpose. 
Indeed, let us make the following observation. Let $Y$ be an optional process such that  
$\|Y \|_\beta =0$.
 We then have $Y_t=0$, $0\leq t \leq T$ $dP \otimes dt$-a.e.\, When $Y$ is right-continuous, this implies the  indistinguishability of $Y$ from the null process $0$, that is, the property $Y_t=0$, $0\leq t \leq T$ a.s.\,
 However, if $Y$ is not right-continuous, the implication is not necessarily true. \footnote{However, the property holds 
 for the "triple bar" map $\vvertiii {\cdot}_{\beta}$ on $\mathcal{S}^2$.  More precisely, if $Y \in \mathcal{S}^2$ with  
$\vvertiii {Y}_{\beta}=0$, then $Y_t=0$, $0\leq t \leq T$ a.s.\, because $\vvertiii {\cdot}_{\beta}$
 is a norm on 
$\mathcal{S}^2$. Note that 
$\| \cdot \|_{\beta}$ is only a semi-norm on $\mathcal{S}^2$. }
%  (which implies that $\|\cdot \|_\beta$ 
% is not a norm on $\mathcal{S}^2$). 
 Hence, applying a fixed point theorem in this  space  cannot give us 
 uniqueness  of the solution of our reflected BSDE in the sense of processes, that is, up to indistinguishability. 
% contrary to the right-continuous case (see e.g. \cite{ElKaroui97}, 
% \cite{QuenSul2}), the Banach space $\H^2 \times \H^2\times \H^2_\nu$ equipped with the norm 
% $ \| Y\|_{\beta}^2  +   \| Z\|_{\beta}^2+\|k \|_{\nu,\beta}^2$ is not appropriate to our case. 
% it is not relevant to use $\|\cdot \|_\beta$  for $\tilde Y$ in the proof of the existence and the uniqueness of the solution of the RBSDE. 
% In our framework, we thus have used 
% $\vvertiii{\cdot}_\beta$ instead of $\|\cdot \|_\beta$ for $\tilde Y$.

\end{Remark}

\begin{proof}
 For each $\beta>0$, we denote by $\mathcal{B}_\beta^2$ the space $\mathcal{S}^2 \times \H^2\times \H^2_\nu$ which we  equip with the norm $\|(\cdot,\cdot,\cdot) \|_{\mathcal{B}_\beta^2}$ defined by 
$\| (Y, Z, k)\|_{\mathcal{B}_\beta^2}^2:=\vvertiii {Y}_{\beta}^2  +   \| Z\|_{\beta}^2+\|k \|_{\nu,\beta}^2, $ for  $(Y,Z,k)\in \mathcal{S}^2 \times \H^2\times \H^2_\nu.$ Since $(\H^2, \|\cdot\|_\beta)$ and $(\H^2_\nu,\|\cdot\|_{\nu,\beta})$ are Banach spaces, and
since by Proposition \ref{Prop_Banach_space}, $(\mathcal{S}^2, 
\vvertiii {\cdot}_{\beta})$ is a Banach space, it follows that
$(\mathcal{B}_\beta^2, \|\cdot \|_{\mathcal{B}_\beta})$ is a Banach space.

We define a mapping $\Phi$ from $\mathcal{B}_\beta^2$ into
itself as follows: for a given $(y, z, l) \in \mathcal{B}_\beta^2$, we set $\Phi (y, z, l):= (Y, Z, k)$, where $Y, Z, k$ 
are
%, if we define the process $$ K_t = Y_t- Y_0 - \int^t_0 f(s, U_s, V_s ) \, ds + \int^t_0 ( Z_s   , \, d B_s ), \, 0 \leq t \leq T, $$ then 
the first three components of the solution $(Y, Z, k,A,C)$ to the RBSDE  associated with driver $f(s):= 
f(s, y_s, z_s, l_s)$ and with obstacle $\xi$. 
%Let $(A,C)$ be the associated Mertens process, constructed as in the proof of Lemma \ref{f}. 
 %%%(cf. also Remark \ref{Rmk_uniqueness}) 
The mapping $\Phi$ is well-defined by Lemma \ref{f}.  %%% and Remark \ref{Rmk_uniqueness}.  
%Note that $(Y, Z,k)$ $\in$ $\H_\beta^2$.
%%Let $\dbarA$ be the associated predictable nondecreasing process, constructed as in Lemma \ref{f}.  

%By using some a priori estimates (see Proposition \ref{est}), $\Phi$ can be shown to be a contraction from $\H_\beta^2$ into
%itself. It thus admits an unique fixed point, which corresponds to the solution of RBSDE (\ref{RBSDE}). 
%
%Using the previous a priori estimates (see Proposition \ref{est}),
  
%%%%%%%Let us prove that for $T$ sufficiently small the mapping $\Phi$ is a contraction from $\mathcal{B}^p$ into %%%%%%%$\mathcal{B}^p$.
%Given $(U, V,l) \in {\H}_\beta^2$, let $ (Y, Z,k) := \Phi (U, V, l)$, that is,
%the solution of the RBSDE  associated with driver process $ 
%f^1_s:=f(s, U_s , V_s, l_s)$ (which does not depend on the solution). 
% 
Let $(y',z',l')$ and $(y'',z'',l'')$ be two elements of $\mathcal{B}_\beta^2$. We set $(Y', Z', k')= \Phi (y', z', l')$ and $(Y'', Z'', k'') = \Phi (y'', z'', l'').$  We also set $\tilde Y:=Y'-Y''$, $\tilde Z:=Z'-Z''$, $\tilde k:=k'-k''$, $\tilde y:=y'-y''$, $\tilde z:=z'-z''$, $\tilde l:=l'-l''$.

Let us prove that for a suitable choice of the parameter $\beta>0$ the mapping  $\Phi$ is a contraction from the Banach space $\mathcal{B}_\beta^2$ into itself.  
%%%%%%%%%%%%%%%
By applying Lemma \ref{Lemma_estimate}, we get 
$$\vvertiii{\tilde Y}^2_\beta +\|\tilde Z\|^2_\beta+\|\tilde k\|^2_{\nu,\beta}\leq  6\varepsilon^2(1+4c^2)  \|f(y',z',l')-f(y'',z'',l'')\|^2_\beta\;,  $$
for all $\varepsilon>0$,  for all $\beta\geq\frac 1 {\varepsilon^2}$. By using the Lipschitz property of $f$ and the fact that $(a+b+c)^2\leq 3(a^2+b^2+c^2)$ for all $(a,b,c)\in\R^3$, we obtain
$\|f(y',z',l')-f(y'',z'',l'')\|_{\beta}^2\leq C_{K}(\|\tilde y\|_{\beta}^2+\|\tilde z\|_{\beta}^2+\|\tilde l\|_{\nu,\beta}^2),$
where $C_{K}$ is a positive constant depending on the Lipschitz constant $K$ only.
Thus, for all $\varepsilon>0$,  for all $\beta\geq\frac 1 {\varepsilon^2}$, we have
$   \vvertiii{\tilde Y}^2_\beta +\|\tilde Z\|^2_\beta+\|\tilde k\|^2_{\nu,\beta}$ $\leq$  $6\varepsilon^2C_K(1+4c^2)(\|\tilde y\|_{\beta}^2+\|\tilde z\|_{\beta}^2+\|\tilde l\|^2_{\nu,\beta}).  $
Now, using Fubini's theorem, we get $\|\tilde y\|_{\beta}^2 \leq T \vvertiii{\tilde y}_{\beta}^2$. Hence, we have
%The previous inequality, combined with Remark \ref{Rmk_norms_SandH}, gives
$$   \vvertiii{\tilde Y}^2_\beta +\|\tilde Z\|^2_\beta+\|\tilde k\|^2_{\nu,\beta}\leq  6\varepsilon^2C_K(1+4c^2)(T+1)(\vvertiii{\tilde y}_{\beta}^2+\|\tilde z\|_{\beta}^2+\|\tilde l\|^2_{\nu,\beta}).  $$
Thus, for $\varepsilon>0$ such that $6\varepsilon^2C_K(1+4c^2)(T+1)<1$ and $\beta>0$ such that $\beta\geq\frac 1 {\varepsilon^2}$  the mapping $\Phi$ is a contraction. 
By the Banach fixed-point theorem, we get that $\Phi$ has a unique fixed point in $\mathcal{B}_\beta^2$, denoted by $(Y,Z,k)$, that is, such that $(Y,Z,k)= \Phi (Y,Z,k)$. By definition of the mapping $\Phi$, the process $(Y,Z,k)$ is thus equal to the first three components of the solution $(Y, Z, k, A, C)$ to the reflected BSDE associated with the driver process $g(\omega,t):=f(\omega, t,Y_t(\omega),Z_t(\omega),k_t(\omega))$ and with obstacle $\xi$. It follows that $(Y,Z,k,A,C)$ is the unique solution 
to the RBSDE with parameters $(f,\xi)$.%%\textit{(dans quels espaces??)}

Property \eqref{yy+} follows from Eq. \eqref{max} of  Lemma \ref{f} and from the fact that $(Y,Z,k,A,C)$ 
is equal to the solution of the reflected BSDE associated with the driver process $g(\omega,t):=f(\omega,t,Y_t(\omega),Z_t(\omega),k_t(\omega))$.

The last assertion of the theorem follows  from Lemma \ref{f} (fourth assertion) applied with the process $g(\omega,t):=f(\omega,t,Y_t(\omega),Z_t(\omega),k_t(\omega))$.

\end{proof}

%
%
%\begin{proposition}\label{prop_ajoutee}
%Let $(Y,Z,k,A,C)$ be the solution to the RBSDE with Lipschitz driver $f$ and obstacle $\xi$
% (as in the previous theorem). For all $S\in\stopo$, we have
% $$ Y_S=\xi_S\vee Y_{S+} \quad   \rm{a.s.}$$ 
%\end{proposition}
%
%\begin{proof}
%The result follows from Eq. \eqref{max} of  Lemma \ref{f} and the fact that $(Y,Z,k,A,C)$ is equal to the solution of the reflected BSDE associated with the driver process $g(\omega,t):=f(\omega,t,Y_t(\omega),Z_t(\omega),k_t(\omega))$. 
%\end{proof}
\section{Optimal stopping with $f$-conditional expectations}\label{sec-charact}
\subsection{Formulation of the problem}\label{sec-charact_subsec_1}

Let $T >0$ be the terminal time and $f$ be a predictable Lipschitz driver. 
 Let $(\xi_t, 0 \leq t \leq T )$ %on $[0, T]$, belonging to 
 be a left-limited r.u.s.c. process
% along stopping times process
 in ${\cal S}^2$ modelling a dynamic financial position.  
%The aim of this paper is to study optimal stopping for dynamic risk measures induced by BSDE with jumps. 
The risk of $\xi$ is assessed by a dynamic risk measure equal, up to a minus sign, to the $f$-conditional expectation of $\xi$. More precisely: let $T'\in[0,T]$ be a fixed (for the present) instant before the terminal time $T$;  the gain of the position at $T'$ is equal to $\xi_{T'}$ and the risk at time $t$, where $t\in[0,T']$, is assessed by  $-\mathcal{E}_{t,T'}^f(\xi_{T'}).$ Here, we use the usual notation $\mathcal{E}_{\cdot,T'}^f(\xi_{T'})$ for the first component of the BSDE with driver $f$, terminal time $T'$ and terminal condition $\xi_{T'}$; the random variable $\mathcal{E}_{t,T'}^f(\xi_{T'})$ is referred to as the $f$-conditional expectation of $\xi_{T'}$ at time $t$. The modelling is similar when $T'\in[0,T]$ is replaced by a more general stopping time $\tau\in\stopo$
\footnote{  Recall that a process $Y$ is the solution to the BSDE associated 
with driver $f$, terminal time $\tau$ and terminal condition $\zeta$ (where $\zeta$ is an $\cf_\tau$-measurable square-integrable random variable) if for almost all $\omega\in\Omega$, for all $t\in[0,T]$, $Y_t(\omega)=\bar{Y}_t(\omega)$, where $\bar Y$ denotes the solution to the BSDE associated with driver $f {\bf 1}_{t \leq \tau}$, terminal time $T$ and terminal condition $\zeta$. The process $Y$ is also denoted $\mathcal{E}_{\cdot,\tau}^f(\zeta)$.}
.\\
We are now interested in stopping the process $\xi$ in such a way that the risk be minimal. We are thus led to formulating the following optimal stopping problem (at time $0$): 
\begin{equation} \label{problem_at_time_0}
v(0) = -  {\rm ess} \sup_{\tau \in \T_{0,T}}{\cal E}^f_{0, \tau}(\xi_{\tau}).
\end{equation}
%%%%%%%%%%%
We recall that in our framework (as opposed to the simpler case of a brownian filtration) the monotonicity property of $f$-conditional expectations is not automatically satisfied.   From now on we make the following assumption  on  the driver $f$,   which 
%and  for RBSDEs with jumps (see Theorem \ref{thmcomprbsde} below), 
ensures  the  nondecreasing property of ${\cal E}^f(\cdot)$ by the comparison theorem for BSDEs  with jumps  (cf. \cite[Thm. 4.2]{QuenSul}).
\begin{Assumption}\label{Royer}
%A driver $f$ is said to satisfy Assumption~\ref{Royer} if the following holds:\\
Assume that  $dP \otimes dt$-a.e.\, for each $(y,z, \mathpzc{k}_1,\mathpzc{k}_2)$ $\in$ $ \RB^2 \times (L^2_{\nu})^2$,
$$f( t,y,z, \mathpzc{k}_1)- f(t,y,z, \mathpzc{k}_2) \geq \langle \theta_t^{y,z, \mathpzc{k}_1,\mathpzc{k}_2}  \,,\,\mathpzc{k}_1 - \mathpzc{k}_2 \rangle_\nu,$$ 
with
\begin{equation*}
\theta:  [0,T]  \times \Omega\times \RB^2 \times  (L^2_{\nu})^2  \rightarrow  L^2_{\nu}\,; \, (\omega, t, y,z, \mathpzc{k}_1,\mathpzc{k}_2)\mapsto 
\theta_t^{y,z, \mathpzc{k}_1,\mathpzc{k}_2}(\omega,\cdot)
\end{equation*}
 ${\cal P } ·\otimes {\cal B}(\R^2) \otimes  {\cal B}( (L^2_{\nu})^2 )$-measurable,
 satisfying
$\|\theta_t^{y,z, \mathpzc{k}_1,\mathpzc{k}_2}(\cdot)\|_{\nu}  \leq K\,$  for all $(y,z, \mathpzc{k}_1,\mathpzc{k}_2)$ $\in$ $\RB^2 \times (L^2_{\nu})^2$, $ \,dP\otimes dt $-a.e.\,, where $K$ is a positive constant, and such that  
 \begin{equation}\label{condi}
\theta_t^{y,z, \mathpzc{k}_1,\mathpzc{k}_2} (e)\geq -1,
\end{equation}
for all $(y,z, \mathpzc{k}_1,\mathpzc{k}_2)$ $\in$ $\RB^2 \times (L^2_{\nu})^2$, $\,dP\otimes dt \otimes d\nu(e)-{\rm a.e.}$
\end{Assumption} 

The above assumption is satisfied if, for example, $f$ is of class ${\cal C}^1$ with respect to $k$ such that $\nabla_k f$ is bounded (in $L^2_{\nu}$) and $\nabla_k f \geq -1$ (see 
Proposition A.2. in \cite{DQS2}).

\begin{Remark}\label{Rk_strict_mon}
 The strict comparison theorem for  BSDEs with jumps (cf. Theorem 4.4 in \cite{QuenSul}) ensures that if the inequality \eqref{condi} is strict,  then ${\cal E}^f(\cdot)$ is {\em strictly monotonous} in the following sense: for $\tau\in\stopo$, for  
$\xi^1, \xi^2$ $\in$ $L^2(\cf_\tau)$ such that $\xi^1 \leq \xi^2$ a.s.\,, and for  $S \in\stopo$ such that $S\leq \tau$ a.s., the property  ${\cal E}_{S,\tau}^f(\xi^1)
= {\cal E}_{S,\tau}^f(\xi^2)$ a.s., implies $\xi^1= \xi^2$ a.s.\\
A counter-example to the strict monotonicity of ${\cal E}^f(\cdot)$ in the case where the strict inequality in \eqref{condi} is not assumed is given in  \cite{QuenSul} (cf. also Example \ref{exe} in the Appendix). 
\end{Remark}
As is usual in optimal control, we embed the above problem \eqref{problem_at_time_0} in a larger class of problems. We thus consider for each $S\in$  $\stopo$, the random variable
%%%%The problem is 
%is to find a stopping time $\tau \geq S$ which
%%%to  minimize the risk measure at time $S$.
%%%Let $v(S)$ be the associated value function, defined by
\begin{equation} \label{vvv}
v(S) = -  {\rm ess} \sup_{\tau \in \T_{S,T}}{\cal E}^f_{S, \tau}(\xi_{\tau}),
\end{equation}
which corresponds to the minimal risk measure at time $S$.
Our aim is to characterize   %the minimal risk-measure
 $v(S)$ 
for each $S \in \stopo$, and  to study the  existence 
 of 
an $S$-optimal stopping time $\tau^* \in \T_{S,T}$, i.e. a stopping time $\tau^* \in \T_{S,T}$ 
such that   $v(S) = -{\cal E}^f_{S, \tau^*}(\xi_{\tau^*})$ a.s.  
%%%We may also see the above problem as an optimal stopping game with one agent whose pay-off process is given by $\xi$   and who evaluates his /her pay-off by a (non-linear) expectation. 

\subsection{Characterization of  the value function  %of the optimal stopping problem 
as the solution of  an RBSDE}\label{sec-charact_subsec_2}

In this section, we show that the minimal risk measure $v$ defined by (\ref{vvv})  coincides with $-Y$, where $Y$ is (the first component of) the solution to the reflected BSDE associated with driver $f$ and obstacle
 $\xi$. We also investigate the question of the existence of an  $\varepsilon$-optimal stopping time, and that of the existence of an optimal stopping time (under suitable assumptions on the process $\xi$).
The following terminology will be used in the sequel. 
%%%\begin{definition}\label{defmart}
Let $Y$ be a process in $ \mathcal S^2$. Let $f$ be a predictable Lipschitz driver satisfying Assumption \ref{Royer}.
\begin{itemize}
\item  The process $(Y_t)$ is said to be a strong ${\cal E}^f$-supermartingale (resp ${\cal E}^f$-submartingale), if ${\cal E}^f_{S ,\tau}(Y_{\tau}) \leq Y_{S}$ (resp. ${\cal E}^f_{S ,\tau}(Y_{\tau})\geq Y_{S})$ a.s.\, on $S \leq \tau$,  for all $ S, \tau \in \stopo$. \\
The process $(Y_t)$ is said to be a strong ${\cal E}^f$-martingale if it is both a strong ${\cal E}^f$-super and ${\cal E}^f$-submartingale.
\item Let $S,\tau\in\stopo$ be such that $S\leq \tau$ a.s. The process $Y$ is said to be a strong ${\cal E}^f$-supermartingale (resp. a strong ${\cal E}^f$-submartingale) on $[S, \tau]$ if for all $\sigma,\mu\in\stopo$ such that $S\leq \sigma\leq \mu\leq \tau$ a.s., we have  $Y_\sigma \geq  {\cal E}^f_{\sigma ,\mu}(Y_{\mu})$ a.s. (resp. $Y_\sigma \leq  {\cal E}^f_{\sigma ,\mu}(Y_{\mu})$ a.s.) We say that $Y$ is a strong ${\cal E}^f$-martingale on $[S, \tau]$ if it is both a strong ${\cal E}^f$-super and submartingale on $[S, \tau]$.
%\item Let $S,\tau\in\stopo$ be such that $S\leq \tau$ a.s. We say that 
%\item Let $S,\tau\in\stopo$ be such that $S\leq \tau$ a.s. We say that the process $Y$ is the solution on $[S,\tau]$ to the BSDE associated 
%with driver $f$, terminal time $\tau$ and terminal condition $Y_\tau$ if for almost all $\omega\in\Omega$, for all $t\in[0,T]$ such that $S(\omega)\leq t\leq \tau(\omega)$, $Y_t(\omega)=\bar{Y}_t(\omega)$, where $\bar Y$ denotes the solution to the BSDE associated with driver $f$, terminal time $\tau$ and terminal condition $Y_\tau$.
\end{itemize} 
%\end{definition}
\begin{Remark}\label{Remarque_prealable}
We note that a process $Y\in\mathcal{S}^2$ is a strong ${\cal E}^f$-martingale on $[S,\tau]$ (where $S$, $\tau\in\stopo$ are such that $S\leq\tau$ a.s.)  if and only if, on $[S,\tau]$, $Y$ is indistinguishable from  the solution to the BSDE associated with driver $f$, terminal time $\tau$ and terminal condition 
$Y_\tau$.
%
%$Y$ is the solution on $[S,\tau]$ to the BSDE associated with driver $f$, terminal time $\tau$ and terminal condition 
%$Y_\tau$ 
%\footnote{ \bf  We say that the process $Y$ is the solution on $[S,\tau]$ to the BSDE associated 
%with driver $f$, terminal time $\tau$ and terminal condition $\zeta$ (where $\zeta$ is an $\cf_\tau$-measurable square-integrable random variable) if for almost all $\omega\in\Omega$, for all $t\in[0,T]$ such that $S(\omega)\leq t\leq \tau(\omega)$, $Y_t(\omega)=\bar{Y}_t(\omega)$, where $\bar Y$ denotes the solution to the BSDE associated with driver N EST CE PAS PLUTOT $f {\bf 1}_{t \leq \tau}$? $f$, terminal time $\tau$ and terminal condition $\zeta$.}
%.\\
It follows that for a process $Y\in\mathcal{S}^2$ to be a strong ${\cal E}^f$-martingale on $[S, \tau]$, it is sufficient to have:  $Y_\sigma={\cal E}^f_{\sigma,\tau}(Y_\tau)$ a.s., for all $\sigma\in\stopo$ such that $S\leq \sigma\leq \tau$ a.s.    %% précision enlevée (This statement is due to the consistency property of $f$-expectations.) 
%%We note that a process $Y\in\mathcal{S}^2$ is a strong ${\cal E}^f$-martingale on $[S,\tau]$ (where $S$, $\tau\in\stopo$ are such that $S\leq\tau$ a.s.)  if and only if for almost all $\omega\in\Omega$, for all $t\in[0,T]$ such that $S(\omega)\leq t\leq \tau(\omega)$, $Y_t(\omega)=\bar{Y}_t(\omega)$, where $\bar Y$ denotes the solution to the BSDE associated with driver $f$, terminal time $\tau$ and terminal condition $Y_\tau$. When the latter property is satisfied, we say that the process $Y$ is the solution on $[S,\tau]$ to the BSDE associated 
%with driver $f$, terminal time $\tau$ and terminal condition $Y_\tau$.
	  
%We note that a process $Y\in\mathcal{S}^2$ is a strong ${\cal E}^f$-martingale on $[S,\tau]$ (where $S$, $\tau\in\stopo$ are such that $S\leq\tau$ a.s.)  if and only if $Y$ is the solution on $[S,\tau]$ to the BSDE associated with driver $f$, terminal time $\tau$ and terminal condition $Y_\tau$.  
\end{Remark}
%%%%%%%%%%%
%%%%%%%
\begin{Property}\label{Pty-strong}
Let $f$ be a predictable Lipschitz driver satisfying Assumption \ref{Royer}. Let $S,\tau\in\stopo$ with $S\leq\tau$ a.s. Let $Y$ be a  strong ${\cal E}^f$-supermartingale on $[S, \tau].$ We introduce the following two assertions:

\begin{description}
\item[(i)] The process $Y$ is a strong ${\cal E}^f$-martingale on $[S, \tau].$
\item[(ii)] $Y_S =  {\cal E}^f_{S ,\tau}(Y_{\tau})$ a.s.

%%%%%%% (iii) lien avec la solution d'une EDSR sur $[S,\tau]$  %%%%%%%%%%%%%%
\end{description} 
Assertion ${\bf (i)}$ implies Assertion ${\bf (ii)}$.\\
If, in Assumption \ref{Royer}, we further assume the strict inequality $\theta_t^{y,z, \mathpzc{k}_1,\mathpzc{k}_2} > -1$, then Assertion ${\bf (ii)}$ implies Assertion {\bf (i)}.

\end{Property}

\begin{proof}
The implication $(i)\Rightarrow(ii)$ is due to the definition. Let us show the converse implication.
Let $\sigma\in\stopo$  be such that $S\leq \sigma\leq \tau$ a.s.
By using $(ii)$ and  the consistency property of $f$-expectations, we obtain
%\begin{equation*}
$Y_S =  {\cal E}^f_{S ,\sigma}\big({\cal E}^f_{\sigma ,\tau}(Y_{\tau})\big)$  a.s. 
%\end{equation*}
By using the strong ${\cal E}^f$-supermartingale property of $Y$ and the monotonicity of $f$-expectations, we obtain 
%\begin{equation*} 
${\cal E}^f_{S ,\sigma}\big({\cal E}^f_{\sigma ,\tau}\big(Y_\tau\big)\big)\leq {\cal E}^f_{S ,\sigma}\big(Y_\sigma\big)
\leq Y_S$  a.s. 
%\end{equation*}
From the previous two equations we get
%%%%%%%%%%%%%
%\begin{equation*}
$Y_S =  
{\cal E}^f_{S ,\sigma}\big({\cal E}^f_{\sigma ,\tau}\big(Y_\tau\big)\big)={\cal E}^f_{S ,\sigma}\big(Y_\sigma\big)$  a.s. 
%\end{equation*}
In particular, 
\begin{equation}\label{stri}
{\cal E}^f_{S ,\sigma}\big(Y_\sigma\big)={\cal E}^f_{S ,\sigma}\big({\cal E}^f_{\sigma ,\tau}\big(Y_\tau\big)\big) \text{ a.s. }
\end{equation}
 Since $\theta_t^{y,z, k_1,k_2} > -1$, 
${\cal E}^f (\cdot)$ is strictly monotonous (cf. Remark \ref{Rk_strict_mon}).
%(see Theorem 4.4 in \cite{QuenSul}) can be 
%applied.
From this, together with equality \eqref{stri} and  the inequality $Y_\sigma\geq    {\cal E}^f_{\sigma ,\tau}\big(Y_\tau\big) \text{ a.s.},$ we get  $Y_\sigma=    {\cal E}^f_{\sigma ,\tau}\big(Y_\tau\big) \text{ a.s. }$ The process $Y$ is thus a strong ${\cal E}^f$-martingale on $[S, \tau].$
\end{proof}

We next show a lemma which will be used in the proof of the main result of this section.

\begin{Lemma}\label{lemma_epsilon_optimality}
Let $f$ be a predictable Lipschitz driver satisfying Assumption \ref{Royer} and $\xi$ be a left-limited r.u.s.c.
% along stopping times
 process in $\mathcal{S}^2$.
Let %$(Y,Z,k(\cdot),A,C)$ is the solution of the reflected BSDE (...).
$(Y,Z,k,A,C)$ be the solution to the reflected BSDE with parameters $(f,\xi)$ as in Definition \ref{def_solution_RBSDE}. Let $\varepsilon >0$ and  $S \in \stopo$. Let
$\tau^{\varepsilon}_S$ be defined by
\begin{equation}\label{eq_tau_epsilon_S}
\tau^{\varepsilon}_S:= \inf \{ t \geq S\colon Y_t \leq \xi_t  + \varepsilon\}.
\end{equation}
The following two statements hold:
\begin{description}
\item[(i)] $Y_{\tau^{\varepsilon}_S} \leq \xi_{\tau^{\varepsilon}_S} + \varepsilon \quad   \rm{a.s.}\,$
 %\begin{equation}\label{epsil}
 %Y_{\tau^{\varepsilon}_S} \leq \xi_{\tau^{\varepsilon}_S} + \varepsilon \quad   \rm{a.s.}\,,
%\end{equation}
%\begin{equation}\label{eq_inequality_epsilon}
%Y_{\tau^{\varepsilon}_S}\leq \xi_{\tau^{\varepsilon}_S} + \varepsilon, \text{ a.s. }
%\end{equation}
\item[(ii)] The process $Y$ is a strong $\mathcal{E}^f$-martingale on $[S,\tau^{\varepsilon}_S]$.
%%\item[(ii)] For a.e. $\omega$,  the map $t \mapsto A_t (\omega) + C_{t^-} (\omega)$ is constant on $[S(\omega), 
%\tau^{\varepsilon}_S(\omega)]$.
\end{description}
\end{Lemma} 

We note that $\tau^{\varepsilon}_S$ defined in \eqref{eq_tau_epsilon_S} is a stopping time as the \emph{début} after $S$ of a progressive set. Note also that $\tau^{\varepsilon}_S$ is valued in $[0,T]$ as $Y_T=\xi_T$ a.s. 
\begin{proof}
%By Lemma \ref{f}, the process $(Y_t)$ satisfies to each $S \in \stopo$,
%\begin{eqnarray}\label{ddeux-2}
%Y_{S} = \esssup_{\tau \in \stops} E[ \xi_{\tau} + \int_S^\tau f(u, Y_u, Z_u, k_u)du \mid \Fc_S] \quad {\rm a.s.}
%\end{eqnarray} 
%By results of optimal theory (see KQ), inequality (i) holds. Moreover, the family 
%$(Y_{\tau}+  \int_0^\tau f(u, Y_u, Z_u, k_u)du)_{\tau}$ is a martingale family on $[S, 
%\tau^{\varepsilon}_S]$. In other terms, the process $(Y_{t}+  \int_0^t f(u, Y_u, Z_u, k_u)du)$ is a strong 
%martingale on on $[S, 
%\tau^{\varepsilon}_S]$. This property implies that $A_{\tau^{\varepsilon}_S} = A_S$ a.s. and 
%$C_{(\tau^{\varepsilon}_S)^-} = C_{S^-}$ a.s. Hence, $Y$ is the solution on $[S, 
%\tau^{\varepsilon}_S]$ of the BSDE associated with driver $f$, terminal time $\tau^{\varepsilon}_S$ and terminal condition $Y_{\tau^{\varepsilon}_S}.$   By using Remark \ref{Remarque_prealable}, statement (ii) follows.
%
We first prove statement $(i)$. By way of contradiction, we suppose $P(Y_{\tau^{\varepsilon}_S}> \xi_{\tau^{\varepsilon}_S} + \varepsilon)>0$. We have  $\Delta C_{\tau^{\varepsilon}_S}=C_{\tau^{\varepsilon}_S}- C_{(\tau^{\varepsilon}_S)-}=0$ on the set $\{Y_{\tau^{\varepsilon}_S}> \xi_{\tau^{\varepsilon}_S} + \varepsilon\}$. On the other hand, due to Remark \ref{Rmk_the_jumps_of_C}, $\Delta C_{\tau^{\varepsilon}_S}=Y_{\tau^{\varepsilon}_S}-Y_{(\tau^{\varepsilon}_S)+}.$ Thus, $Y_{\tau^{\varepsilon}_S}=Y_{(\tau^{\varepsilon}_S)+}$ on the set $\{Y_{\tau^{\varepsilon}_S}> \xi_{\tau^{\varepsilon}_S} + \varepsilon\}.$ Hence,
\begin{equation}\label{eq_contradiction}
Y_{(\tau^{\varepsilon}_S)+}> \xi_{\tau^{\varepsilon}_S} + \varepsilon \text{ on the set } \{Y_{\tau^{\varepsilon}_S}> \xi_{\tau^{\varepsilon}_S} + \varepsilon\}.
\end{equation}
 We will obtain a contradiction with this statement.
% To simplify the presentation, we use here Remark \ref{ruscbis}, according to which the process 
%$(\xi_t)$ is right-upper-semicontinuous. \\
Let us fix $\omega\in\Omega$. By definition of $\tau^{\varepsilon}_S(\omega)$, there exists a non-increasing sequence $(t_n)=(t_n(\omega))\downarrow \tau^{\varepsilon}_S(\omega) $ such that $Y_{t_n}(\omega)\leq \xi_{t_n}(\omega) + \varepsilon,$ for all $n\in\N$. Hence,
$\limsupn Y_{t_n}(\omega)\leq \limsupn \xi_{t_n}(\omega)+\varepsilon.$ As the process $\xi$ is r.u.s.c.\,, we have $\limsupn \xi_{t_n}(\omega)\leq \xi_{\tau^{\varepsilon}_S}(\omega)$. On the other hand, as $(t_n(\omega))\downarrow \tau^{\varepsilon}_S(\omega)$, we have $\limsupn Y_{t_n}(\omega)=Y_{(\tau^{\varepsilon}_S)+}(\omega).$ Thus, $Y_{(\tau^{\varepsilon}_S)+}(\omega)\leq \xi_{\tau^{\varepsilon}_S}(\omega)+\varepsilon,$ which is in contradiction with \eqref{eq_contradiction}. We conclude that $Y_{\tau^{\varepsilon}_S}\leq  \xi_{\tau^{\varepsilon}_S} + \varepsilon \text{ a.s.}$\\
Let us now prove statement $(ii)$. By definition of $\tau^{\varepsilon}_S$, we have: for a.e. $\omega\in\Omega $, for all $t\in[S(\omega), \tau^{\varepsilon}_S(\omega)[$, $Y_t(\omega)> \xi_t(\omega)  + \varepsilon.$  %on $[S, \tau^{\varepsilon}_S[$ a.s.\,
Hence, for a.e. $\omega\in\Omega $, the function $t\mapsto A^c_t(\omega) $ is constant on $[S(\omega), 
\tau^{\varepsilon}_S(\omega)[$; by continuity of almost every trajectory of the process $A^c$, $A^c_\cdot(\omega)$ is constant on the closed interval $[S(\omega), 
\tau^{\varepsilon}_S(\omega)]$, for a.e. $\omega$. Furthermore, for a.e. $\omega\in\Omega $, the function  $t\mapsto A^d_t(\omega)$ is constant on $[S(\omega), \tau^{\varepsilon}_S(\omega)[.$
Moreover, $ Y_{ (\tau^{\varepsilon}_S)^-  } \geq \xi_{ (\tau^{\varepsilon}_S)^-  }+ \varepsilon\,$ a.s.\,, which implies that $\Delta A^d _ {\tau^{\varepsilon}_S  } =0$ a.s.\, Finally, for a.e. $\omega\in\Omega $, for all $t\in[S(\omega), \tau^{\varepsilon}_S(\omega)[$, $\Delta C_t(\omega)=C_t(\omega)-C_{t-}(\omega)=0$; %on  $[S, \tau^{\varepsilon}_S[$ a.s.; 
therefore, for a.e. $\omega\in\Omega $, for all $t\in[S(\omega), \tau^{\varepsilon}_S(\omega)[$, $\Delta_+ C_{t-}(\omega)=C_t(\omega)-C_{t-}(\omega)=0$, which implies that, for a.e. $\omega\in\Omega$, the function $t\mapsto C_{t-}(\omega)$ is constant on  $[S(\omega), \tau^{\varepsilon}_S(\omega)[.$ By left-continuity of almost every trajectory of the process  $(C_{t-})$, we get that for a.e. $\omega\in\Omega$, the function $t\mapsto C_{t-}(\omega)$ is constant on the closed interval $[S(\omega), \tau^{\varepsilon}_S(\omega)]$. Thus, for a.e. $\omega\in\Omega$,  the map $t \mapsto A_t (\omega) + C_{t-} (\omega)$ is constant on $[S(\omega), \tau^{\varepsilon}_S(\omega)]$. Hence, $Y$ is the solution on $[S, 
\tau^{\varepsilon}_S]$ of the BSDE associated with driver $f$, terminal time $\tau^{\varepsilon}_S$ and terminal condition $Y_{\tau^{\varepsilon}_S}.$   We conclude by using Remark \ref{Remarque_prealable}.%%
%%The assertion $(ii)$ is thus proved. %%%The process $Y_. + \int_0^{. }f(s)ds $ is thus a martingale on 
%$[S, \tau^{\varepsilon}_S]$.   
\end{proof}
%%\begin{remark}
 %In this proof, we have used Remark \ref{ruscbis}. 
 % Another version of this lemma is given in the Appendix which one can prove without using  this remark.
 %\end{remark}
%
%We now state the main result of this section.
%
 With the help of the previous lemma, we derive the main result of this section.  
\begin{theorem}[Characterization theorem]\label{caracterisation}
Let $T >0$ be the terminal time.
 Let $(\xi_t, 0 
\leq t \leq T )$ be a left-limited r.u.s.c. 
%along stopping times
process in ${\cal S}^2$ and let $f$ be a  predictable Lipschitz driver satisfying Assumption \ref{Royer}.
Let %$(Y,Z,k(\cdot),A,C)$ is the solution of the reflected BSDE (...).
$(Y,Z,k,A,C)$ be the solution to the reflected BSDE with parameters $(f,\xi)$ as in Definition \ref{def_solution_RBSDE}.
\begin{description}
\item[(i)] For each stopping time $S$ $\in$ $\T_0$, we have
 \begin{equation}\label{prixam}
Y_S =  {\rm ess} \sup_{\tau \in \T_{S,T}}{\cal E}^f_{S, \tau}(\xi_{\tau}) \quad   \rm{a.s.}
\end{equation}
%where for $\tau$ $\in$ $\T_S$,
%$X_\cdot(\xi_{\tau}, \tau)$ %,    \pi (\xi_{\tau}, \tau), l(\xi_{\tau}, \tau)     )$ 
%is  the solution of the BSDE associated with terminal time $\tau$,  terminal condition $\xi_{ \tau}$,   and driver $f$.
\item[(ii)] For each  $S \in \stopo$ and each $\varepsilon >0$, the stopping time  $\tau^{\varepsilon}_S $ 
defined by \eqref{eq_tau_epsilon_S} 
%\begin{equation}\label{epstop}
 %\tau^{\varepsilon}_S  = \inf \{ t \geq S,\,\, Y_t \leq \xi_t + \varepsilon\}
 % \end{equation}
  is  {\em $ (L \varepsilon) $-optimal} for problem \eqref{prixam}, that is
 \begin{equation}\label{fifi}
 Y_S \,\,\leq \,\,  {\cal E}^f_{S, \tau^{\varepsilon}_S}(\xi_{\tau^{\varepsilon}_S})
 + L %e^{\frac{\beta T} {2}}
 \varepsilon \quad   \rm{a.s.}\,,
 \end{equation}
where $L$ is a constant which only depends on $T$ and the Lipschitz constant  $K$ of $f$.
%%%%% plus tard%%%%%%%%%
%%%\item[(iii)] 
%%%When $(\xi_t, 0 
%%%%\leq t \leq T )$ is l.u.s.c.  along stopping times,  the stopping time $\tau_S^*$ defined by 
%%%%$$\tau^*_S := \inf \{u \geq S ; \, Y_u = \xi_u \}$$ 
%%%%is optimal for \eqref{prixam}, that is $Y_S =  {\cal E}^f_{S,  \tau^*_S}(\xi_{\tau^*_S})$ a.s.\,
%%%%Hence, the process $(Y_t)$ is a strong ${\cal E}^f$-martingale on $[S, \tau^*_S]$.

 % $\beta= 3 C^2 +2C$, where $C$ is the Lipschitz constant of $f$.\\
%In other words, $\tau^{\varepsilon}_S$
% is a {\em $ (K \varepsilon) $-optimal stopping time} for \eqref{prixam}.
 %, with $\varepsilon ' =  e^{\frac{\beta T} {2}} \varepsilon$.
 \end{description}
%where on the interval 
%$[S, \tau]$, the process $(X_{s}(\xi_{\tau}, \tau),
%\pi_{s}(\xi_{\tau}, \tau), l_s(\xi_{\tau}, \tau))$ satisfies the BSDE
%\begin{eqnarray*} 
% -dX_{s}= f(s,X_{s}, \pi_{s}, l_s)ds - \pi_{s}dW_{s} - \int_{ U} l_s(u) \tilde{N}(ds,du); 
%&& X_{\tau}= \xi_{\tau}.
% \end{eqnarray*}

\end{theorem}

\begin{Remark}\label{Rmk_AA}
This result still holds when the assumption of existence of left limits for the process $\xi$ is relaxed (cf. also Remark \ref{Rmk_left_uppersemicontinuous envelope}). 
\end{Remark}
In the case where $\xi$ is right-continuous, we recover   Theorem 3.2 of \cite{QuenSul2}.  
\begin{proof}
%, that is
% \begin{eqnarray*} 
% -dX^{ \tau}_{s}= f(s,X^{ \tau}_{s}, \pi^{ \tau}, l^{ \tau})ds  - \pi^{ \tau}dW_{s} - \int_{ U} l^{ \tau}_s(u) \tilde{N}(ds,du); 
%&& X^{ \tau}_{\tau}= \xi_{\tau}.
% \end{eqnarray*}
%Note that the process $(Y,Z,k(\cdot))$
% satisfies~:
% \begin{eqnarray*} 
% -dY_{s}= f(s,Y_{s}, Z_{s}, k_s)ds + dA_s - Z_{s}dW_{s} - \int_{ U} k_s(u) \tilde{N}(ds,du); 
%&& Y_{\tau}= Y_{\tau}.
% \end{eqnarray*}
Let $\varepsilon>0$ and let $\tau$ $\in$ $\T_{S,T}$. 
By Proposition \ref{compref}  in the Appendix,  the process $(Y_t)$ is a strong ${\cal E}^f$-supermartingale.  
Hence, for each $\tau \in \T_{S,T}$, we have 
$$Y_{S} \geq {\cal E}^f_{S ,\tau}(Y_{\tau})\geq {\cal E}^f_{S ,\tau}(\xi_{\tau})  \quad\text{ a.s.}\,, $$ 
where the second inequality follows from the inequality $Y \geq \xi$ and the monotonicity property of 
${\cal E}^f(\cdot)$ (with respect to terminal condition).
By taking the supremum over $\tau \in \T_{S,T}$, we get
 \begin{equation}\label{prixame}
Y_S  \geq  {\rm ess} \sup_{\tau \in \T_{S,T}} {\cal E}^f_{S ,\tau}(\xi_{\tau})  \quad \text{ a.s.}
\end{equation}
It remains to show the converse inequality. 
Due to part $(ii)$ of the previous Lemma \ref{lemma_epsilon_optimality} we have  
 %\begin{equation}\label{ma}
$Y_S =  {\cal E}^f_{S, \tau^{\varepsilon}_S}(Y_{\tau^{\varepsilon}_S}) \quad \text{ a.s. }$ 
%\end{equation}
%By the flow property of BSDEs, the process 
%$(Y_t, S\leq t \leq 
%\tau^{\varepsilon}_S )$ is a strong ${\cal E}$-martingale. 
%\end{itemize}
%Let us prove inequality \eqref{fifi}.
%\begin{lemma}\label{eps}
%For each $\varepsilon$ $>0$ and each $S$ $\in$ $\T_0$, we have
%\begin{equation}\label{fifi}
%Y_S \,\,\leq \,\,  X_S ( \xi_{\tau^{\varepsilon}_S}, \tau^{\varepsilon}_S)  + e^{\frac{\beta T} {2}}\varepsilon \quad   \rm{a.s.}\,
%\end{equation}
%
%
% \end{lemma}
% 
% \dproof
From this equality, together with part $(i)$ of Lemma \ref{lemma_epsilon_optimality} and the monotonicity property of ${\cal E}^f(\cdot)$, we derive 
\begin{equation}\label{fi}
  Y_S =  {\cal E}^f_{S, \tau^{\varepsilon}_S}(Y_{\tau^{\varepsilon}_S})
    \leq  {\cal E}^f_{S, \tau^{\varepsilon}_S}(\xi_{\tau^{\varepsilon}_S}+ \varepsilon)  \leq 
     {\cal E}^f_{S, \tau^{\varepsilon}_S}(\xi_{\tau^{\varepsilon}_S})
 + L %e^{\frac{\beta T} {2}}
 \varepsilon \quad   \rm{a.s.},
\end{equation}
where the last inequality follows from the  estimates on BSDEs (cf. Proposition A.4 in \cite{QuenSul}). Inequality (\ref{fifi}) thus holds. %, which ends the proof of Lemma \ref{eps}.
%\fproof \\
% {\bf End of  proof of Theorem \ref{caracterisation}.}
%By Lemma \ref{eps}, we have 
From \eqref{fi} we also deduce
%\begin{equation*}
$Y_S \,\,\leq \,\, {\rm ess} \sup_{\tau \in \T_{S,T}}{\cal E}^f_{S ,\tau}(\xi_{\tau}) + L \varepsilon $    a.s.
%\end{equation*}
As $\varepsilon$ is an arbitrary positive number, we get 
$% \begin{eqnarray*}
Y_S \,\, \leq  \,\,{\rm ess} \sup_{\tau \in \T_{S,T}}{\cal E}^f_{S ,\tau}(\xi_{\tau})$  
a.s.\,\,%\end{eqnarray*} 
By \eqref{prixame} this inequality is an equality.
\end{proof}

We  now investigate the question of the existence of optimal stopping times for the optimal stopping  problem \eqref{prixam}. We first provide an optimality criterion for the problem \eqref{prixam}.
%

%We now provide an optimality criterion for the problem \eqref{prixam}.
%
\begin{Proposition}[Optimality criterion] \label{optcri}
Let $(\xi_t, 0 
\leq t \leq T )$ be a left-limited r.u.s.c.
%  along stopping times
 process in ${\cal S}^2$ and let $f$ be a predictable Lipschitz driver satisfying Assumption \ref{Royer}. %  with 
%$\theta_t^{y,z, \mathpzc{k}_1,\mathpzc{k}_2} > -1$.
Let $S \in \stopo$ and  $\hat{\tau} \in  \T_{S,T}.$ 
If $Y$ is a strong ${\cal E}^f$-martingale on $[S,{\hat \tau}]$ with 
$Y_{\hat \tau}= \xi_{\hat \tau}$ a.s., then the stopping time $\hat{\tau}$ is $S$-optimal (i.e. $Y_S =  {\cal E}^f_{S, \hat \tau}(\xi_{\hat \tau})$ a.s.). The converse statement also holds true,  if, in addition, the  inequality from  Assumption \ref{Royer} is strict 
(that is,   $\theta_t^{y,z, \mathpzc{k}_1,\mathpzc{k}_2} > -1$).   %%Suppose that in Assumption \ref{Royer}, we have 
%for each $l$ $\in$ ${\cal L}^2_\nu$,
 %%$
 %%\theta_t^{x,\pi, l_1, l_2} > - 1.$  
 %%\\
%The stopping time $\hat{\tau}$ is $S$-optimal, i.e. $Y_S =  {\cal E}^f_{S, \hat \tau}(\xi_{\hat \tau})$ a.s.\,,
%if and only if $Y$ is a strong ${\cal E}^f$-martingale on $[S,{\hat \tau}]$ with 
%$Y_{\hat \tau}= \xi_{\hat \tau}$ a.s. %$(Y_s,  \,S \leq s \leq  {\hat \tau})$ is a strong ${\cal E}^f$-martingale with 
%$Y_{\hat \tau}= \xi_{\hat \tau}$ a.s.
%\\
%{\bf Moreover, the strict inequality $\theta_t^{y,z, \mathpzc{k}_1,\mathpzc{k}_2} > -1$ is not needed for the "if part".}
\end{Proposition}

\begin{proof} The first claim is immediate.
Let us prove the second (and last) claim.
Assume the strict inequality in Assumption \ref{Royer}.   Let $\hat{\tau}$ be $S$-optimal, i.e. $Y_S =  {\cal E}^f_{S, \hat \tau}(\xi_{\hat \tau})$ a.s.
Since by Theorem \ref{caracterisation} and by Proposition \ref{compref}, $Y$ is a strong ${\cal E}^f$-supermartingale, we have
$$Y_S \geq  {\cal E}^f_{S, \hat \tau}(Y_{\hat \tau}) \geq {\cal E}^f_{S, \hat \tau}(\xi_{\hat \tau})= Y_S\quad a.s.\,, $$ 
where the last inequality holds because $Y \geq \xi$.
It follows that $Y_S =  {\cal E}^f_{S, \hat \tau}(Y_{\hat \tau})$ a.s.
Since $\theta_t^{y,z, \mathpzc{k}_1,\mathpzc{k}_2} > -1$, Property \ref{Pty-strong} can be applied, which yields that
 $Y$ is 
  a strong ${\cal E}^f$-martingale on $[S,{\hat \tau}]$. Moreover, since ${\cal E}^f_{S, \hat \tau}(Y_{\hat \tau}) = {\cal E}^f_{S, \hat \tau}(\xi_{\hat \tau})$ 
a.s.\, with $Y_{\hat \tau} \geq \xi_{\hat \tau}$ a.s.\,, the strict monotonicity of ${\cal E}^f$ 
%strict comparison theorem for BSDEs %with jumps (\cite{QuenSul}, Th 4.4) 
implies that $Y_{\hat \tau} = \xi_{\hat \tau}$ a.s.\,     
%The "if part" is immediate.
\end{proof}
We note that, even in the case where $\xi$ is right-continuous, the large inequality   $\theta_t^{y,z, \mathpzc{k}_1,\mathpzc{k}_2} \geq -1$ from Assumption \ref{Royer} is not sufficient for  the last statement of the above proposition to hold true; a counter-example is given in the Appendix (cf. Example \ref{exe}).  

%\begin{Remark}\label{Rk_strict inequality}
%Note that the strict inequality $\theta_t^{y,z, \mathpzc{k}_1,\mathpzc{k}_2} > -1$ is not needed for the "if part" in the above proposition to hold true. 
%\end{Remark}
In Theorem \ref{caracterisation} (ii), we have shown the existence of an $L\varepsilon$-optimal stopping time for problem \eqref{problem_at_time_0}. Under an additional assumption of left upper-semicontinuity  along stopping times of the process $\xi$, we will show the existence of an optimal stopping time. %More precisely,  we prove that the first time when the value process $Y$ "hits" $\xi$ is optimal.  
To this purpose, we first give a lemma which is to be compared with Lemma \ref{lemma_epsilon_optimality}.
\begin{Lemma}\label{lemma_optimality}
Let $f$ be a predictable Lipschitz driver satisfying Assumption \ref{Royer}.
Let $(\xi_t, 0 
\leq t \leq T )$ be a left-limited r.u.s.c. process in ${\cal S}^2$ which we assume also to be 
% along stopping times 
l.u.s.c. along stopping times. Let 
%$(Y,Z,k(\cdot),A,C)$ is the solution of the reflected BSDE (...).
$(Y,Z,k,A,C)$ be the solution to the reflected BSDE with parameters $(f,\xi)$.
Let $S \in  \T_{0,T}.$ We define $\tau_S^*$  by
\begin{equation}\label{eq_lemma_optimality}
\tau^*_S := \inf \{u \geq S\colon Y_u = \xi_u \}.
\end{equation} 
The following assertions hold:
\begin{description}
\item[(i)] $Y_{\tau^*_{S}} = \xi_{\tau^*_{S}}$ a.s.
\item[(ii)] The process $Y$ is a strong $\mathcal{E}^f$-martingale on $[S,\tau^*_S]$.  
%For a.e. $\omega$,  the map $t \mapsto A_t (\omega) + C_{t^-} (\omega)$ is constant on $[S(\omega), 
%\tau^*_S(\omega)]$.
\end{description}
\end{Lemma}
\begin{proof}
To prove the first statement we note that $ Y_{\tau^*_{S}} \geq \xi_{\tau^*_{S}}$ a.s., since $Y$ is (the first component of) the solution to the RBSDE with barrier $\xi$. We show that $ Y_{\tau^*_{S}} \leq \xi_{\tau^*_{S}}$ a.s.  by using the assumption of right-upper semicontinuity on the process $\xi$; the  arguments are similar to those used in the proof of part $(i)$ of Lemma \ref{lemma_epsilon_optimality} and are left to the reader.\\
Let us prove the second statement.  By definition of $\tau^*_S$, we have that for a.e. $\omega\in\Omega,$   $Y_t(\omega) >\xi_t(\omega)$ on  $[S(\omega), \tau^*_S(\omega)[$; hence, for a.e. $\omega$, the trajectory   $A^c(\omega)$ is constant on $[S(\omega), \tau^*_S(\omega)[$ and even on the closed interval $[S(\omega), \tau^*_S(\omega)]$ due to the continuity.  On the other hand, due to the assumption of l.u.s.c. along stopping times on the process $\xi$, we have $A(\omega)=A^c(\omega)$ for a.e. $\omega$ 
(see Theorem \ref{exiuni}). Thus, for a.e. $\omega$, $A(\omega)$ is constant on $[S(\omega), \tau^*_S(\omega)].$ We show that $C_{t-}(\omega)$ is constant on $[S(\omega), \tau^*_S(\omega)]$ by the same arguments as those of the proof of part $(ii)$ of Lemma \ref{lemma_epsilon_optimality}. We conclude by using Remark \ref{Remarque_prealable}.
\end{proof}

\begin{Remark} \label{sautA}
We see from the above proof that the assumption of l.u.s.c. of $\xi$ in Lemma \ref{lemma_optimality} can be replaced by the assumption $\Delta A_{\tau_S^*}=0$. The assumption $\Delta A_{\tau_S^*}=0$ is weaker than the assumption of l.u.s.c. of $\xi$ as illustrated in  Example  \ref{toy2}  of the Appendix.   
% By the above proof, for a.e. $\omega\in\Omega,$ the trajectory   $A(\omega)$ is constant on the closed interval $[S(\omega), \tau^*_S(\omega)]$ if and only if  $\Delta A_{\tau_S^*}(\omega) =0$. Now, by using  Remark \ref{Remarque_prealable}, we get that if  
%$\Delta A_{\tau_S^*} =0$ a.s.\,, then 
%$Y$ is a strong $\mathcal{E}^f$-martingale on $[S,\tau^*_S]$.
%  Suppose now that
%$\theta_t^{y,z, \mathpzc{k}_1,\mathpzc{k}_2} > -1$. By the strict comparison theorem for BSDEs with jumps (cf. Theorem 4.4 in \cite{QuenSul}), we derive that $Y$ is a strong $\mathcal{E}^f$-martingale on $[S,\tau^*_S]$ if and only if $\Delta A_{\tau_S^*}=0$ a.s.\, 
\end{Remark}
By the previous lemma %, the characterization Theorem \ref{caracterisation} 
and the  first statement ("the optimality criterion")  from Proposition \ref{optcri}, we derive the following existence result.
\begin{Proposition}\label{exist}
Let $f$ be a predictable Lipschitz driver satisfying Assumption \ref{Royer}.
Let $(\xi_t, 0 
\leq t \leq T )$ be a left-limited r.u.s.c. process in ${\cal S}^2$ which we assume also to be 
%along stopping times
l.u.s.c. along stopping times. %process in ${\cal S}^2$. 
%We define $\tau_S^*$  by 
%$$\tau^*_S := \inf \{u \geq S ; \, Y_u = \xi_u \}.$$ 
%The following assertions hold:
%\begin{description}
%\item[(i)] 
%\item[(ii)] 
Let $S \in  \T_{0,T}$. The stopping time $\tau_S^*$ defined in \eqref{eq_lemma_optimality} 
%%$$\tau^*_S := \inf \{u \geq S ; \, Y_u = \xi_u \}$$ 
is optimal for problem \eqref{prixam}, that is 
$Y_S =  {\rm ess} \sup_{\tau \in \T_{S,T}}{\cal E}^f_{S, \tau}(\xi_{\tau}) = {\cal E}^f_{S,  \tau^*_S}
(\xi_{\tau^*_S})$ a.s.\,
%If moreover $\theta_t^{y,z, \mathpzc{k}_1,\mathpzc{k}_2} > -1$, $\tau_S^*$ is then the minimal optimal stopping time.

%\end{description}
%Hence, the process $(Y_t)$ is a strong ${\cal E}^f$-martingale on $[S, \tau^*_S]$
\end{Proposition}
%\begin{proof} {\bf The optimality of $\tau_S^*$ follows from the third statement of the previous Lemma \ref{lemma_optimality} %, the characterization Theorem \ref{caracterisation}
%and the  first statement ("the optimality criterion")  from Proposition \ref{optcri}.\\
% Suppose moreover
%that $\theta_t^{y,z, \mathpzc{k}_1,\mathpzc{k}_2} > -1$. Let $\hat \tau$ $\in$ in ${\cal T}_{[S,T]}$ be an optimal stopping time for problem \eqref{prixam}. By "the optimality criterion"  from Proposition \ref{optcri}, 
%$Y_{\hat \tau}= \xi_{\hat \tau}$ a.s.\,Hence, by the definition of $\tau^*_S$, we get $\hat \tau \geq \tau^*_S$ a.s.
%\,The stopping time $\tau_S^*$ is thus the minimal optimal stopping time. }
%\end{proof}
\begin{Remark}  %Suppose that $\xi$ is not necessarily  
%along stopping times
%l.u.s.c. along stopping times. Let 
%$(Y,Z,k(\cdot),A,C)$ is the solution of the reflected BSDE (...).
%$(Y,Z,k,A,C)$ be the solution to the reflected BSDE with parameters $(f,\xi)$. Let $S \in  \T_{0,T}.$\\
We note that, due to Remark \ref{sautA} and to the optimality criterion,  the optimality of $\tau_S^*$ %for $Y_S$ 
in the above proposition  still holds if we relax the assumption of l.u.s.c. of $\xi$ to the (weaker) assumption  
$\Delta A_{\tau_S^*} =0$ a.s. We recall that, by Remark \ref{Rmk_the_jumps_of_C}, the condition 
$\Delta A_{\tau_S^*} =0$ a.s. is equivalent to  $Y$ being  left-continuous along stopping times at $\tau_S^*$. If the condition $\Delta A_{\tau_S^*} =0$ a.s. is violated, the stopping time $\tau_S^*$ might not be optimal (cf. Example \ref{toy2} from the Appendix). 

%Suppose now that
%$\theta_t^{y,z, \mathpzc{k}_1,\mathpzc{k}_2} > -1$. By Remark \ref{sautA}, the converse statement then holds true. We even have the equivalence
% between the condition $\Delta A_{\tau_S^*} =0$ a.s.\,and the existence of an optimal stopping time for $Y_S$.
%This statement follows from the fact that if $\hat \tau$ is an optimal stopping time for $Y_S$, then so is $\tau_S^*$.  
%Indeed, in this case, as shown in the above proof, 
%% by the optimality criterium, $Y_{\hat \tau}= \xi_{\hat \tau}$ a.s\, and $Y$ is  a strong $\mathcal{E}^f$-martingale on  $[S,\hat \tau]$.
%%This implies that 
%${\hat \tau} \geq \tau_S^*$ a.s. Hence, $Y$ is  a strong $\mathcal{E}^f$-martingale 
%on $[S,\tau^*_S]\subset [S, \hat \tau]$. Since $Y_{\tau_S^*}= \xi_{\tau_S^*}$ a.s\,, we derive that 
%$\tau_S^*$ is optimal for $Y_S$.

%A simple example illustrating the above results  is given  in Remark \ref{toy2} in the Appendix.
 \end{Remark}

%\begin{definition}
%Let $(\xi_t, 0 
%\leq t \leq T )$ be a
% process in ${\cal S}^2$ and let $f$ be a  predictable Lipschitz driver satisfying Assumption \ref{Royer}.
% A process $(Y_t, 0 
%\leq t \leq T )$ in ${\cal S}^2$ is said to be the {\em ${\cal E}^f$- Snell envelope} of $\xi$ if it is the smallest strong 
%${\cal E}^f$-supermartingale greater than or equal to $\xi$. 
%\end{definition}
%We show the following property:
%
%\begin{proposition}\label{snell}
%Let $T >0$ be the terminal time.
% Let $(\xi_t, 0 
%\leq t \leq T )$ be a left-limited r.u.s.c. 
%%along stopping times 
%process in ${\cal S}^2$
% and let $f$ be a predictable Lipschitz driver satisfying Assumption \ref{Royer}.
%Let $(Y,Z,k,A,C)$ be the solution to the reflected BSDE with parameters $(\xi,f)$ as in Definition \ref{def_solution_RBSDE}.
%The process $Y$ is the ${\cal E}^f$-Snell envelope of $\xi$.
%\end{proposition}

We show the following property.  %which generalizes to our framework a well-known result 
%of the classical optimal stopping theory (cf., e.g., the end of the Appendix in \cite{Marc})
\begin{Proposition}\label{snell}
Let $T >0$ be the terminal time.
 Let $(\xi_t, 0 
\leq t \leq T )$ be a left-limited r.u.s.c. 
%along stopping times 
process in ${\cal S}^2$
 and let $f$ be a predictable Lipschitz driver satisfying Assumption \ref{Royer}.
Let $(Y,Z,k,A,C)$ be the solution to the reflected BSDE with parameters $(\xi,f)$ as in Definition \ref{def_solution_RBSDE}.
The process $Y$ is the ${\cal E}^f$-{\em Snell envelope} of $\xi$, 
that is, \emph{the smallest strong 
${\cal E}^f$-supermartingale greater than or equal to $\xi$.}
\end{Proposition}
\begin{Remark}\label{Rmk_AB}
This result still holds when $\xi$ is not left-limited 
(cf. Remarks \ref{Rmk_left_uppersemicontinuous envelope} and \ref{Rmk_AA}).
\end{Remark}
From Proposition \ref{snell} and Theorem \ref{caracterisation}, we deduce that the "value process" of the   optimal stopping problem  \eqref{vvv} is characterized as the ${\cal E}^f$-{\em Snell envelope}
of the reward process $\xi$. 
In the particular case of a classical (linear) expectation (that is, when $f=0$), we recover a characterization from the classical optimal stopping theory stating that the "value process" of the  "classical" linear optimal stopping problem coincides with the {\em Snell envelope} of $\xi$, which is smallest strong 
supermartingale greater than or equal to $\xi$  (cf, e.g., 
\cite{ALM}).
\begin{proof}
By Proposition \ref{compref} in the Appendix, the process $Y$ is a strong 
${\cal E}^f$-supermartingale. Moreover, since $Y$ is (the first component of) the solution to the reflected BSDE with parameters $(f,\xi)$, it is  greater than or equal to $\xi$ (cf. Def. \ref{def_solution_RBSDE}). \\
It remains to show the minimality property. 
Let $Y'$ be another ${\cal E}^f$-supermartingale greater than or equal to $\xi$.  Let $S \in\stopo$.
For each $\tau \in \T_{S,T}$, we have 
$Y'_{S}$ $ \geq $ ${\cal E}^f_{S ,\tau}(Y'_{\tau})$ $\geq$ $ {\cal E}^f_{S ,\tau}(\xi_{\tau})$ a.s.\,,
where the second inequality follows from the inequality $Y' \geq \xi$ and the monotonicity property of 
${\cal E}^f$ with respect to the terminal condition.
By taking the supremum over $\tau \in \T_{S,T}$, we get
% \begin{equation}\label{prixame_1}%%%%% étiquette changée car répétition
%Y'_S  \geq  {\rm ess} \sup_{\tau \in \T_{S,T}} {\cal E}^f_{S ,\tau}(\xi_{\tau}) = Y_S \quad a.s.,
%\end{equation}
%where the last equality follows from the above characterization theorem (Theorem \ref{caracterisation}). 
%The desired result follows.
%
$Y'_S$ $ \geq$  ${\rm ess} \sup_{\tau \in \T_{S,T}} {\cal E}^f_{S ,\tau}(\xi_{\tau})$ $=$ $Y_S$ a.s.\,,
where the last equality follows from Theorem \ref{caracterisation}. 
The desired result follows.
\end{proof}
%\begin{Remark}
%%
%{\bf Note that when the driver does not depend on $y,z, \mathpzc{k}$, that is $f=f(t)$, this result leads to the  characterization of the process $(Y_t + \int_0^t f(s)ds)$, where $Y$ is the value function of the classical 
%optimal stopping problem \eqref{deux-2}, as the {\em Snell envelope} of the process
% $(\xi_t + \int_0^t f(s)ds)$, that is 
%the smallest strong 
%supermartingale greater than or equal to  $(\xi_t + \int_0^t f(s)ds)$  
%(see \cite{EK} or the end of the Appendix in \cite{Marc}).}
%%
%\end{Remark}

\section{Additional results}\label{sec_add}
\subsection{${\cal E}^f$-Mertens decomposition of ${\cal E}^f$-strong supermartingales}\label{subsec_Mertens}

We now  show an ${\cal E}^f$-Mertens decomposition for ${\cal E}^f$-strong supermartingales, which generalizes  Mertens decomposition to the case of $f$-expectations.
%%%%%%%%%%%
%%% partie enlevée et remplacée par un lemme sans hypothèse ladlag 
We first show the following lemma.
\begin{lemma}\label{rusc}
Let $(Y_t)$ $\in\mathcal{S}^2$  be a strong ${\cal E}^f$-supermartingale (resp. ${\cal E}^f$-submartingale).
Then, $(Y_t)$ is right upper-semicontinuous  (resp. right lower-semicontinuous). 
%%%% (je ne suis pas sure de l'assertion qui suit:
%%When  $(Y_t)$ a strong ${\cal E}^f$-martingale, it is a cadlag process (à vérifier).
\end{lemma}

\begin{proof} Suppose that $(Y_t)$ is a 
strong ${\cal E}^f$-supermartingale.
Let $\tau \in \stopo$ and let $(\tau_n)$ be a nonincreasing sequence of stopping times with $\lim_{n \rightarrow + \infty} \tau_n = \tau$ a.s. and  for all $n\in\N$, $\tau_n > \tau$ a.s. on 
$\{\tau < T\}$. Suppose that $\lim_{n \rightarrow + \infty} Y_{\tau_n}$ exists a.s.\,
The random variable $\lim_{n \rightarrow + \infty} Y_{\tau_n}$ is $\cf_\tau$-measurable as the filtration is right-continuous. 
Let us show that 
$$Y_{\tau} \geq \lim_{n \rightarrow + \infty} Y_{\tau_n} \quad {\rm a.s.}$$
Since $(Y_t)$ is a strong ${\cal E}^f$-supermartingale and the sequence $(\tau_n)$ is nonincreasing, we have, for  
all $n\in\N$,
${\cal E}^f_{\tau ,\tau_n}(Y_{\tau_n})\leq {\cal E}^f_{\tau ,\tau_{n+1}}(Y_{\tau_{n+1}})\leq Y_\tau.$ We deduce that the sequence of random variables $({\cal E}^f_{\tau ,\tau_n}(Y_{\tau_n}))_{n \in\N}$   is nondecreasing (hence, converges a.s.) and its limit (in the a.s. sense) satisfies $Y_{\tau} \geq \lim_{n \rightarrow + \infty} \uparrow {\cal E}^f_{\tau ,\tau_n}(Y_{\tau_n})$ a.s.\, This observation, combined with the continuity property of BSDEs with respect to terminal time and terminal condition (cf.   \cite[Prop. A.6]{QuenSul}) gives
$$Y_{\tau} \geq \lim_{n \rightarrow + \infty}  {\cal E}^f_{\tau ,\tau_n}(Y_{\tau_n})= {\cal E}^f_{\tau ,\tau}(\lim_{n \rightarrow + \infty} Y_{\tau_n})= \lim_{n \rightarrow + \infty} Y_{\tau_n}\quad {\rm a.s.}$$ 
% we derive that for each $n$, we have $Y_{\tau} \geq {\cal E}_{\tau ,\tau_n}(Y_{\tau_n})$ a.s.\, Also, one can show that the sequence of random variables $({\cal E}_{\tau ,\tau_n}(Y_{\tau_n}))_{n \in {\mathbb N}}$ is non decreasing, and thus converges a.s.\, We thus get 
%$Y_{\tau} \geq \lim_{n \rightarrow + \infty} \uparrow {\cal E}_{\tau ,\tau_n}(Y_{\tau_n})$ a.s.\, Hence, by using the continuity property of BSDEs with respect to terminal time and terminal condition (see Proposition A.6 in [Quenez-Sulem BSDEs with jumps]), we obtain
%$$Y_{\tau} \geq \lim_{n \rightarrow + \infty}  {\cal E}_{\tau ,\tau_n}(Y_{\tau_n})= {\cal E}_{\tau ,\tau}(\lim_{n \rightarrow + \infty} Y_{\tau_n})= \lim_{n \rightarrow + \infty} Y_{\tau_n}\quad {\rm a.s.}\,$$ 
This result, together with a result of the general theory of processes (cf. \cite[Prop. 2, page 300]{DelLen2}), ensures that the optional process $(Y_t)$ is right-upper semicontinuous.\\
%Reference: C. Dellacherie and E. Lenglart 1980-81, S\'eminaire de Probabilit\'es {\em 
%Sur des probl\`emes de r\'egularisation, de recollement et d'interpolation en th\'eorie des processus}, 
%p.298-313. 
\end{proof}

%\begin{lemma}\label{rusc}
%Let $(Y_t)$ be a process in ${\cal S}^2$ which we assume to be a  strong ${\cal E}^f$-supermartingale (resp. ${\cal E}^f$-submartingale).
%Then, $(Y_t)$ is right-upper semi continuous  (resp. left-upper semi continuous). In other terms,
%$Y_t \geq Y_{t+}$, $0 \leq t \leq T$, a.s. \, (resp. \emph{à vérifier} $Y_t \leq Y_{t+}$).\\ 
%If  $(Y_t)$ is a strong ${\cal E}^f$-martingale, then it is a cadlag process.
%\end{lemma}

%\begin{proof} Suppose that $(Y_t)$ is a 
%strong ${\cal E}^f$-supermartingale.
%Let us show that $Y_t \geq Y_{t+}$, $0 \leq t \leq T$, a.s. 
%Let $\tau \in \stopo$. There exists a nonincreasing sequence of stopping times $(\tau_n)$ such that $\lim_{n %\rightarrow + \infty} \tau_n = \tau$ a.s. and
%for each $n$, $\tau_n > \tau$ a.s. on 
%$\tau < T$. Since $(Y_t)$ is a strong ${\cal E}^f$-supermartingale, we derive that for each $n$, we have $Y_{\tau} \geq {\cal E}^f_{\tau ,\tau_n}(Y_{\tau_n})$ a.s.\, We thus get 
%$Y_{\tau} \geq \lim_{n \rightarrow + \infty} \uparrow {\cal E}^f_{\tau ,\tau_n}(Y_{\tau_n})$ a.s.\,
%By using the continuity of (\textit{reflected???}) BSDEs with respect to terminal time and terminal condition (cf. ...), we get
%$Y_{\tau} \geq   {\cal E}_{\tau ,\tau}(\lim_{n \rightarrow + \infty} Y_{\tau_n})= Y_{\tau^+}$ a.s.\,, where by convention we have set $Y_{T^+} = Y_T$.
%Hence, $Y_{\tau} \geq   Y_{\tau^+}$ a.s.\, By a result of the general theory of processes (cf. also Proposition \ref{optional_section} in the Appendix), this implies $Y_t \geq Y_{t^+}$, $0 \leq t \leq T$, a.s.
%\end{proof}
 \begin{theorem}[${\cal E}^f$-Mertens decomposition]\label{Thm_Mer}
Let $(Y_t)$ be a process  in ${\cal S}^2$. Let $f$ be a predictable Lipschitz driver satisfying Assumption \ref{Royer}.
The process  $(Y_t)$ is a strong ${\cal E}^f$-supermartingale (resp. ${\cal E}^f$-submartingale) if and only if 
there exists a nondecreasing (resp. nonincreasing) right-continuous predictable process $A$ in ${\cal S}^2$ with $A_0=0$ and 
a nondecreasing (resp. nonincreasing) right-continuous adapted purely discontinuous process $C$ in ${\cal S}^2$ with $C_{0-}=0$, as well as  two processes $Z \in   \H^2 $ and $k \in  \mathbb{H}^2_\nu$,  such that
%\\
%$
%  -d  Y_t  \displaystyle =  f(t,Y_{t}, Z_{t}, k_t)dt + dA_t + d C_{t-}-  Z_t dW_t-\int_E k_t(e)\tilde{N}(dt,de),
%$\\
  %
%  {\bf 
%  (J' AI SUPPRIME UNE EQUATION)
%  }
  %
% in the sense that, for each $\tau\in\stopo$, the equality
%in the sense that, 
a.s. for all $t \in [0,T]$,
% $$-d  Y_t  \displaystyle =  f(t,Y_{t}, Z_{t},k_t)dt + dA_t + d C_{t-}- Z_t dW_t-\int_E k_t(e)\tilde N(dt,de). $$
% \begin{equation}\label{dy}
%  -d  Y_t  \displaystyle =  f(t,Y_{t}, Z_{t}, k_t)dt + dA_t + d C_{t-}-  Z_t dW_t-\int_E k_t(e)\tilde{N}(dt,de)
%  \end{equation}
  %
%  {\bf 
%  (J' AI SUPPRIME UNE EQUATION)
%  }
  %
% in the sense that, for each $\tau\in\stopo$, the equality
 $$
 Y_{t}= Y_{T}+ \int_{t} ^T f(s,Y_{s}, Z_{s},k_s)ds + A_T - A_{t} + C_{T-}- C_{t-}- \int_{t} ^T Z_{s}dW_{s}-\int_{t} ^T \int_E k_{s}(e)\tilde{N}(ds,de).
$$
%%% (trivial)Also, $(Y_t)$ is a strong ${\cal E}^f$-martingale if and only if it is the solution of a BSDE associated with driver $f$.  
This decomposition is unique.
\end{theorem}

\begin{proof} The "if part"  has been shown in Proposition \ref{compref} of the Appendix. Let us show the "only if" part. Suppose that $(Y_t)$ is a 
strong ${\cal E}^f$-supermartingale. Hence, $(Y_t)$ is clearly the ${\cal E}^f$-Snell envelope of $(Y_t)$, that is the 
smallest strong ${\cal E}^f$-supermartingale 
greater or equal to $(Y_t)$. By the characterization of the solution of a reflected BSDE as the ${\cal E}^f$-Snell envelope of the obstacle 
process (cf. Proposition \ref{snell} and Remark \ref{Rmk_AB}), we derive that the process $(Y_t)$ coincides with the solution of the reflected BSDE associated with the obstacle $(Y_t)$ (which is r.u.s.c. by Lemma \ref{rusc}). The desired conclusion follows.\\
 The uniqueness of the processes $Z$, $k$, $A$, $C$ of the decomposition follows from the uniqueness of the solution of the reflected BSDE.
%Let $S$ $\in$ $\stopo$. 
%Since $(Y_t)$ is a strong ${\cal E}^f$-supermartingale, it follows that for each $\tau \in \T_{S,T}$, we have $Y_{S} \geq {\cal E}^f_{S ,\tau}(Y_{\tau})$ a.s.\, We thus get
%  $Y_S  \geq  {\rm ess} \sup_{\tau \in \T_{S,T}} {\cal E}^f_{S ,\tau}(Y_{\tau})$  a.s.\,
%  Now, by definition of the essential supremum, $Y_S  \leq  {\rm ess} \sup_{\tau \in \T_{S,T}} {\cal E}^f_{S ,\tau}(Y_{\tau}) $ a.s.\, (since $S \in \T_{S,T}$). 
% % since $S \in \T_S$.
% The previous two inequalities imply  
% % \Rightarrow \quad
%  $ Y_S  =  {\rm ess} \sup_{\tau \in \T_{S,T}} {\cal E}^f_{S ,\tau}(Y_{\tau})$ a.s.\,
%  By our characterization theorem  (Theorem \ref{caracterisation}) combined with Remark \ref{Rmk_AA}, 
%the process $(Y_t)$ coincides with the solution of the reflected BSDE associated with the obstacle $(Y_t)$ (which is r.u.s.c. by Lemma \ref{rusc}).  The result follows. 
\end{proof}

When $Y$ is right-continuous, the process $C$ of the ${\cal E}^f$-Mertens decomposition is equal to $0$. 
In this case, the previous theorem reduces to the so-called 
${\cal E}^f$-Doob-Meyer decomposition  (cf. Proposition A.6 in \cite{DQS2}; cf. also \cite{R} and \cite{Peng_Doob_Meyer}). \\
Through different techniques, a similar result to the above Theorem \ref{Thm_Mer} has been established in the recent paper \cite{Bouchard} (in the Brownian framework).\\
%The previous theorem corresponds to Proposition A.6 in \cite{DQS2} in the right-continuous case (see also \cite{R}
   
\begin{Remark}
It follows from the previous theorem that strong $\mathcal{E}^f$-supermartingales and strong $\mathcal{E}^f$-submartingales have left and right limits.
\end{Remark}
\subsection{Comparison theorem for RBSDEs }\label{subsec_comp}
%We  now state a comparison theorem for RBSDEs with jumps.
\begin{theorem}[Comparison]\label{thmcomprbsde}
\label{comparison} Let $\xi^{1}$, $\xi^{2}$  be two obstacles. Let $f^1$and $f^2$ 
 be  predictable Lipschitz drivers   satisfying Assumption~\ref{Royer}. Let $(Y^{i}, Z^{i}, k^i,  A^i, C^i)$  be
the solution of the RBSDE associated with $(\xi ^{i},f^{i})$ , $i=1,2$.
Suppose  that $\xi_t ^{2}\le \xi_t ^{1}$, $0\leq t \leq T$ a.s. and that
$f^{2}(t,Y_t^2, Z_t^2, k_t^2) \le f^{1}(t,Y_t^2, Z_t^2, k_t^2)$, $0\leq t \leq T$ $dP\otimes dt$-a.s.\,\\
%$f^{2}(t,y,z,\mathpzc{k}) \le f^{1}(t,y,z,\mathpzc{k}),\,\, \text{ for all  }  (y,z,\mathpzc{k}) \in \R^2\times L^2_\nu  \;\;,  dP\otimes dt-a.e.\,$\\
%\begin{itemize}
%\item   $\xi_t ^{2}\le \xi_t ^{1}$, $0\leq t \leq T$ a.s.
%
%\item  $f^{2}(t,y,z,\mathpzc{k}) \le f^{1}(t,y,z,\mathpzc{k}),\,\, \text{ for all  }  (y,z,\mathpzc{k}) \in \R^2\times L^2_\nu  \;\;,  dP\otimes dt-a.e.\,$
%\end{itemize}
Then, 
$Y_{t}^{2}\le
Y_{t}^{1}, \,\,\forall t\in [0,T] \text{ a.s. } $
\end{theorem}

\begin{proof} {\bf Step 1}: Let us first consider the case where, along with the assumptions of the theorem, the following additional assumption holds:  %$\xi_\cdot ^{2}\le \xi _\cdot ^{1}$ and 
$f^{2}(t,y,z,k) \le f^{1}(t,y,z,k)$ for all  $(y,z,k) \in \R^2 \times L^2_\nu$ $dP\otimes dt$-a.s. 
Let $S$ $\in$ $\stopo$. %For $i= 1,2$ and for each $\tau \in \T_{S,T}$, let us denote by ${\cal E}^{f^i}$ the 
 %conditional (non-linear) expectation associated with driver $f^i$.
By the comparison theorem for BSDEs,  for each $\tau$ in $\T_{S,T}$ we have
${\cal E}^{f^2}_{S, \tau}(\xi^2 _{\tau}) \leq {\cal E}^{f^1}_{S, \tau}(\xi^1 _{\tau}) $ a.s.\,
By taking the essential supremum over $\tau\in\T_{S,T}$
 and  by using Theorem \ref{caracterisation}, we get $Y^{2}_{S}\leq Y^{1}_{S}\,$ a.s. \\
%$$Y^{2}_{S}= {\rm ess} \sup_{\tau \in \T_{S,T}}{\cal E}^{f^2}_{S, \tau}(\xi^2 _{\tau})  \leq {\rm ess} \sup_{\tau \in \T_{S,T}}{\cal E}^{f^1}_{S, \tau}(\xi^1 _{\tau})=Y^{1}_{S}\,\, \quad {\rm a.s.   } \, $$
%Since this inequality holds for each $S$ $\in$ $\stopo$, the result follows.\\
{\bf Step 2}: %Suppose that only assumptions of the above theorem hold. 
Let us now place ourselves under the assumptions of the theorem (without the additional assumption on $f^1$ and $f^2$ from Step 1).   
Let $\delta f$ be the process defined by 
$\delta f_t:= f^{2}(t,Y_t^2, Z_t^2, k_t^2)-f^{1}(t,Y_t^2, Z_t^2, k_t^2) $. Note that $(Y^2, Z^2, k^2)$ 
is the solution of the reflected BSDE associated with obstacle $\xi ^{2}$ and 
driver $f^1(t,y,z,k) + \delta f_t$. Now, by assumption, we have $f^1(t,y,z,k)+ \delta f_t $ $\leq $ $f^1(t,y,z,k) $  for all $(y,z,k)$. By Step 1 
applied to the drivers $f^1$ and
$f^1(t,y,z,k) + \delta f_t$ (instead of $f^2$), we get  $Y^2 \leq Y^1$. 
\end{proof}

\section{Further developments}\label{sec_extension}
In our ongoing work (cf. \cite{Imkeller2}), we study the case of doubly reflected BSDEs where the barriers  are not right-continuous.
\appendix
\section{}
The following observation is given for the convenience of the reader. 
%%%%%%%page 45 des notes%%%%%%%
\begin{Remark}\label{esssup_cadlag_process}
Let $Y$ be a right-continuous (or left-continuous) adapted process. Then, $\sup_{t\in[0,T]}Y_t=\sup_{t\in[0,T]\cap \Q}Y_t$ a.s., which implies that $\sup_{t\in[0,T]}Y_t$ is a random variable. Moreover, due to the definition of the essential supremum, we have
$\sup_{t\in[0,T]}Y_t=\esssup_{t\in[0,T]}Y_t=\esssup_{\tau\in\stopo}Y_\tau$ a.s.
\end{Remark}
%The following definition can be found in \cite[Appendix 1, Def.1]{DM2}.
%
\begin{Definition}\label{Def_surmartingale_forte}
Let $(Y)_{t\in[0,T]}$ be an optional process. We say that $Y$ is a {\em strong (optional) supermartingale} if
$Y_\tau$ is integrable for all $\tau\in\stopo$ and $Y_S\geq E[Y_\tau\mid \cf_S]$ a.s., for all $S,\tau\in\stopo$ such that $S\leq \tau$ a.s.
%\begin{itemize}
%\item $Y_\tau$ is integrable for all $\tau\in\stopo$ and
%\item $Y_S\geq E[Y_\tau\mid \cf_S]$ a.s., for all $S,\tau\in\stopo$ such that $S\leq \tau$ a.s. 
%\end{itemize}
\end{Definition}
%Note that if $Y$ is a right-continuous supermartingale (in the usual sense), then $Y$ is a strong (optional) supermartingale in the sense of the above definition (due to the optional sampling theorem for right-continuous supermartingales).  The converse statement does not necessarily hold true, that is,  a strong (optional) supermartingale is not necessarily right-continuous. \\
%%However,  almost all trajectories of a  strong optional supermartingale have left and right limits (cf. \cite{DM2}, Appendix 1, Thm. 4, page 408).\\ 
We  recall a decomposition of strong optional supermartingales, known as \emph{Mertens decomposition} 
(see e.g. \cite[Theorem 20, page 429, combined with Remark 3(b), page 205]{DM2} and \cite[Appendix 1, Thm.20, equalities 
(20.2)]{DM2}).
%(cf. \cite[equation (31.4), page 388]{Meyer_cours}, 
%or \cite[Theorem 20, page 429, combined with Remark 3(b), page 205]{DM2}, 
%or \cite[Prop. 2.26]{EK}, or \cite[Theorem 4, page 527]{Lenglart}).

\begin{Theorem}[Mertens decomposition]\label{thm_Mertens_decomposition}
Let $Y$ be a strong optional supermartingale of class $(D)$. 
There exists a unique right-continuous left-limited uniformly integrable  martingale $(M_t)$, a unique predictable right-continuous nondecreasing process $(A_t)$ with $A_0=0$ and $E[A_T] < \infty$, and a unique right-continuous adapted nondecreasing process $(C_t)$, which is purely discontinuous, with $C_{0-}=0$ and $E[C_T] < \infty$, such that
%There exists a uniformly integrable martingale $M$ and a non-decreasing predictable process $\dbarA$ such that $E[\dbarA_T]<\infty$ and  
\begin{equation}\label{eq_thm_Mertens_decomposition}
Y_t = M_t - A_{t} - C_{t-}, \,\,\, 0 \leq t \leq T  \,\,\, {\rm a.s.} 
%\dbarY_S=M_S-\dbarA_S \text{ a.s. for all } S\in\stopo.
\end{equation}
In particular, all trajectories of $Y$ have left and right limits.
%We also  have $\Delta C_t= Y_t - Y_{t+}$ and 
%$\Delta A_S = Y_{S-}- \,^{p}Y_S$ for all predictable stopping times $S\in\stopo$, where $^{p}Y$ is the predictable 
%projection of $Y$. 
%The processes $M$ and $\dbarA$ are unique.\\
%%The process $A$ is continuous if and only if $Y$ is left-continuous along stopping times in expectation.
%{\bf In particular, the process $A$ is right-continuous if and only if  $Y$ is right-continuous. Moreover,  $A$ is left-continuous 
% if and only if $Y$ is {\em regular} (i.e. for any non-decreasing sequence of stopping times $(S_n)$, we have $\limn E[Y_{S_n}]=E[Y_{S}],$ where $S:=\limn S_n$).  }
%%%%Moreover, the process $\dbarA$ is cag  if and only if $\dbarY$ is left-continuous in expectation (i.e. for any non-decreasing sequence of stopping times $(S_n)$, we have $\limn E[\dbarY_{S_n}]=E[\dbarY_{S}],$ where $S:=\limn S_n$).  
 % i.e. si et seulement si $\dbarY$ est régulière.
\end{Theorem}
%The following remark can be found in \cite[Appendix 1, Thm.20, equalities 
%%(20.2)]{DM2}.
%%\begin{Remark}\label{Rmk_jumps of C}
%%We  have $\Delta C_t= Y_t - Y_{t+}$ and 
%%$\Delta A_S = Y_{S-}- \,^{p}Y_S$ for all predictable stopping times $S\in\stopo$, where $^{p}Y$ is the predictable projection of $Y$. 
%%\end{Remark}

\begin{Remark}\label{Rmk_jumps of C}
%From the Mertens decomposition, it follows that $\Delta C_t= Y_t - Y_{t+}$. Hence, $Y_t\geq Y_{t+}$, for all $t\in[0,T)$, which implies
% that $Y$ is necessarily r.u.s.c.\,
% Moreover, for all predictable stopping times $S\in\stopo$, we have $Y_{S^-} \geq E[Y_S | {\cal F}_{S^-}]= Y_S$, where the last equality follows from the fact that $ {\cal F}_{S^-}=  {\cal F}_{S}$, because that filtration 
% is left-quasi-continuous.\\
Since the filtration in our framework is quasi-left-continuous, martingales have only totally inaccessible jumps. 
From this and from   Mertens decomposition \eqref{eq_thm_Mertens_decomposition}, we deduce that,  for each predictable stopping time $\tau$,   $Y_{\tau}-Y_{\tau-}= - (A_{\tau}-A_{\tau-})$ a.s.\,
%, which 
%implies that $Y_{\tau}\leq Y_{\tau-}$
%It follows that for each $\tau$ $\in$ $\stopo$, the process $Y$ is left-continuous along stopping times at $\tau$ if and only if the process $\Delta A_{\tau}=0$ a.s.
\end{Remark}

\begin{Remark}
%[The particular right-continuous case]
    By  Mertens decomposition \eqref{eq_thm_Mertens_decomposition}, we get $\Delta C_t$ $=$ 
$C_t-C_{t-}$
$=$ $ Y_t - Y_{t+}$. Hence, $Y_t\geq Y_{t+}$, for all $t\in[0,T)$, which implies
 that $Y$ is necessarily r.u.s.c.\, Moreover, $Y$ is right-continuous if and only if  $C\equiv 0$. \end{Remark}
 
By using this remark, we recover the well-known Doob-Meyer decomposition for \emph{right-continuous} supermartingales of class $(D)$. 
Indeed, let $Y$ be a \emph{right-continuous} supermartingale (in the usual sense) of class $(D)$. Then $Y$ is a strong (optional) supermartingale in the sense of the above definition (due to the optional sampling theorem for right-continuous supermartingales).    Mertens decomposition of $Y$  reduces to $Y=M-A$ (where $M$ and $A$ are as  above), as   $C\equiv 0$.  This corresponds to  Doob-Meyer decomposition of $Y$.  %Mertens decomposition can be seen as a generalization of Doob-Meyer decomposition for \emph{right-continuous} supermartingales of class $(D)$.   Indeed, let $Y$ be a \emph{right-continuous} supermartingale  of class $(D)$.  Then, as already mentioned, $Y$ is a strong optional supermartingale and  in  virtue of the above Theorem \ref{thm_Mertens_decomposition}, $Y$ can be uniquely decomposed as follows: 
%$Y=M-A-C_{-}$, where $M$, $A$, and  $C$ are as  above. Moreover, we have $\Delta C:= C_t-C_{t-}= Y_t - Y_{t+}=0$ (the last equality being due to the right-continuity of $Y$). As $C$ is a right-continuous,  purely discontinuous process such that $C_0-=0$, we conclude that $C\equiv 0$.
%We thus  recover Doob-Meyer decomposition of the right-continuous supermartingale $Y$. 

%The process $\dbarA$ defined by $dbarA_t=A_t+C_{t-}$ is sometimes referred to as the \textit{ Mertens process} associated with the strong optional supermartingale $Y$. 
%%%%%%%%%%%%
The following result from potential theory can be found in \cite{DM2}.  
\begin{Theorem}[Dellacherie-Meyer]\label{thm_fondamental_generalized}
Let $K$ be a non-decreasing predictable process (which is not necessarily right-continuous). Let $U$ be the \emph{potential} of the  process $K$, i.e. 
$$U_t:= E[
K_T-K_t\,|\,F_t]$$ for all $t\in[0,T]$. Assume that there exists a non-negative $\cf_T$-measurable random variable $X$ such that  $U_S\leq E[X|\cf_S]$ a.s. for all $S\in\stopo$. Then, there exists a constant $c>0$ such that
\begin{equation}\label{eq_norms_inequality}
E[K_T^2]\leq c E[X^2].
\end{equation}      
%where $c>0$ is a positive constant. 
\end{Theorem}

\begin{proof}
%%% a rédiger mieux 
%For this result the reader is referred to \cite[Theorem VI.99]{DM2} whose generalization to the case of  a process which is not necessarily cad nor cag can be found in \cite[Appendix 1, Paragraph 18]{DM2}. 
%OR 
For the proof of  the result the reader is referred to Paragraph 18 in \cite[Appendix 1]{DM2} generalizing Theorem VI.99 of the same reference
 to the case of a non-decreasing process which is not necessarily right-continuous nor left-continuous.
%  {\bf The arguments are the same.}
\end{proof}
%For a given strong optional supermartingale $Y$, %we denote by $K$ the sum of the two nondecreasing processes from Mertens decomposition of $Y$; more precisely,  $K_t=A_t+C_{t-}$  
%we define $K$  by $K_t:=A_t+C_{t-}$, where $A$ and $C$ are the two nondecreasing processes of Mertens decomposition of $Y$ from equation \eqref{eq_thm_Mertens_decomposition}. The process $K$ is sometimes referred to as the \textit{Mertens process} associated with $Y$.\\ 
By using the previous theorem, we obtain the following integrability property of 
the {\em Mertens process}
associated with a strong optional supermartingale, which is used in the proof of Lemma \ref{f}.
\begin{Corollary}[{\em Mertens process} of a strong supermartingale: a useful estimate]\label{corollary_DM} 
Let $Y$ be a strong optional supermartingale of class $(D)$ such that: for all $S\in\stopo$, 
$|Y_S|\leq E[X|\cf_S]$ a.s., where $X$ is a non-negative $\cf_T$-measurable random variable. 

Let us consider the {\em Mertens process} of $Y$, that is the process $(A_t + C_{t-})$, where $A$ and $C$ are the two nondecreasing processes of Mertens decomposition of $Y$ from equation \eqref{eq_thm_Mertens_decomposition}.
%Let $K$ be the sum of the two non-decreasing processes of  Mertens decomposition of 
%the strong supermartingale
%%the so-called {\em Mertens process} associated with the strong supermartingale
% $Y$, 
%% defined by
%i.e.  $K_t:=A_t+C_{t-}$, where 
%Let $A$ and $C$ be the two nondecreasing processes of Mertens decomposition of $Y$ from equation \eqref{eq_thm_Mertens_decomposition}.  % from equation \eqref{eq_thm_Mertens_decomposition}.  
There exists a constant $c>0$ such that 
%\begin{equation}\label{eq_corollary_DM}
%E\left[\left(K_T\right)^2\right]\leq c E[X^2],
%\end{equation} 
\begin{equation}\label{eq_corollary_DM}
E\left[\left(A_T+C_{T-}\right)^2\right]\leq c E[X^2].
\end{equation}      
\end{Corollary}
\begin{proof}
Let us introduce the notation $K_t:=A_t+C_{t-}$ ($K$ is the Mertens process of $Y$). Note that $K$ is a non-decreasing predictable process (which is not necessarily right-continuous).\\
Let $S\in\stopo.$ From Mertens decomposition, we have
$Y_S=M_S-K_S \text{ a.s. }$
and  $Y_T=M_T-K_T \text{ a.s. }$ By subtracting the second equation from the first, and by taking conditional expectations, we derive that 
%Hence, $E[\dbarY_T|\cf_S]=E[M_T|\cf_S]-E[K_T|\cf_S] \text{ a.s. }$
%Taking conditional expectations in  the second equation gives $E[Y_T|\cf_S]=E[M_T|\cf_S]-E[K_T|\cf_S] \text{ a.s. }$ By subtracting this equation from the first, we obtain 
%%By subtracting the second equation from the first, and by taking conditional expectations, we get
%$Y_S-E[Y_T|\cf_S]=M_S-E[M_T|\cf_S]+E[K_T|\cf_S]-K_S$ a.s.
%Thus, 
$Y_S-E[Y_T|\cf_S]=E[K_T- K_S|\cf_S]$ a.s. Hence, the process $(U_t)$ defined by $U_t:=Y_t-E[Y_T|\cf_t]$ is the potential associated with the non-decreasing predictable process $K$. Now, we have
\begin{equation}\label{eq2_corollary_DM}
|U_S|=|Y_S-E[Y_T|\cf_S]|\leq  |Y_S|+E[|Y_T||\cf_S] \leq E[2X|\cf_S] \text{ a.s. }\,, 
\end{equation}
where the last inequality follows from the assumption.
%,
%where we have used the fact that $M$ is a martingale, together with the optional sampling theorem.
%From this equality we easily get
%\begin{equation}\label{eq2_corollary_DM}
%\begin{aligned}
%|E[K_T|\cf_S]-K_S|=|Y_S-E[Y_T|\cf_S]|\leq  |Y_S|+E[|Y_T||\cf_S] \text{ a.s. }
%\end{aligned}
%\end{equation}
%By  using this observation and the assumption, we get the following upper bound for the potential at time $S$
%of the process $K$:
%$$|E[K_T|\cf_S]-K_S|\leq E[X|\cf_S]+ E[E[X|\cf_T]|\cf_S]=E[2X|\cf_S] \text{ a.s. }$$
By applying Theorem \ref{thm_fondamental_generalized}, there exists a constant $c>0$ such that
$E[K_T^2]\leq c E[X^2]$, which is the desired conclusion.
%we obtain the desired conclusion.
\end{proof}
%%%%partie abrégée%%%%%%%%%
%\begin{Remark}\label{Rmk_norms_increasing_process}
%Let $p>1$. For a non-decreasing adapted process $(A)_{t\in[0,T]}$ the following observation holds:\\
%$A$ is in  ${\cal S}^{p}$ if and only if $A_T$ is in $L^p$. In that case, we have the equality $\vertiii{A}^p_{{\cal S}^{p}}=\|A_T\|^p_{L^p}.$ Indeed, due to the non-decreasingness of $A$ and to the non-decreasingness of the function $x\mapsto x^p$, we have $A_S^p\leq A_T^p$ a.s. for all $S\in\stopo$. Hence, $\esssup_{S\in\stopo} A_S^p\leq A_T^p$ a.s. The converse inequality, namely  $A_T^p\leq \esssup_{S\in\stopo} A_S^p$, is trivially true.
%\end{Remark}

%For the convenience of the reader, we recall the following result, which is 
%a direct consequence of the optional section theorem (cf. \cite[Theorem IV.84]{DM1}). 
%\begin{Proposition}\label{optional_section}
%Let $X$ and $Y$ be two optional processes such that 
%$X_S\geq Y_S$ a.s. for all $S\in\stopo$. Then, $X\geq Y$ up to an evanescent set. 
%\end{Proposition}
%The proof is analogous to those of the proof of Theorem IV.86 of \cite{DM1} and is left to the reader.
%%%%%%%%%Nous n'avons pas besoin de la proposition qui suit%%%%%%%%
%\begin{Proposition}\label{martingale-system}
%Let $X$ be an integrable random variable.  The family of random variables
%$\{E[X|\cf_S]\colon S\in\stopo\}$ can be aggregated by a cadlag adapted process $(X_t)$
%\end{Proposition}
%\subsubsection{Change of variables}
%

We recall the change of variables formula  for {\em optional strong semimartingales} 
 which are not necessarily right-continuous. The result  can be seen as a generalization of the classical Itô formula and can be found in
\cite[Theorem 8.2]{Galchouk} (cf. also \cite[Chapter VI, Section 3, page 538]{Lenglart}). We recall the result in our framework in which the underlying filtered probability space satisfies the usual conditions. 
\begin{Theorem}[Gal'chouk-Lenglart]\label{Thm_Ito}
Let $n\in\N$. Let $X$ be an $n$-dimensional {\em optional strong semimartingale},
i.e. $X=(X^1,\ldots, X^n)$ is an $n$-dimensional optional process with decomposition $X^k=X_0^k+M^k +A^k+B^k$, for all $k\in\{1,\ldots,n\}$, where 
$M^k$ is a (cadlag) local martingale, $A^k$ is a right-continuous adapted process of finite variation such that $A_0=0$, and $B^k$ is a left-continuous adapted process of finite variation which is purely discontinuous and such that $B_0=0$. Let $F$ be a twice continuously differentiable function on $\R^n$. Then, 
\begin{equation*}
\begin{aligned}
F(X_t)= &\,\, F(X_0)+\sum_{k=1}^{n}\int_{]0,t]} D^k F(X_{s-})d(A^k+M^k)_s 
+\frac 1 2 \sum_{k,l=1}^{n}\int_{]0,t]} D^kD^l F(X_{s-})d<M^{kc},M^{lc}>_s\\
\quad &+ \sum_{0<s\leq t}\left[ F(X_s)-F(X_{s-})-\sum_{k=1}^{n} D^k F(X_{s-}) \Delta X_s^k\right]
+\sum_{k=1}^{n}\int_{[0,t[} D^k F(X_{s})d(B^k)_{s+}\\
 &+\sum_{0\leq s<t}\left[ F(X_{s+})-F(X_{s})-\sum_{k=1}^{n} D^k F(X_{s}) \Delta_+ X_s^k\right],\quad 0 \leq t \leq T 
 \quad {\rm a.s.}\,,
\end{aligned}
\end{equation*}
where $D^k$ denotes the differentiation operator with respect to the $k$-th coordinate, and $M^{kc}$ denotes the continuous part of $M^k$.
\end{Theorem}

\begin{Corollary}\label{Cor_Ito}
Let $Y$ be a one-dimensional optional strong semimartingale with decomposition $Y=Y_0+M+A+B$, where 
$M$, $A$, and $B$ are as in the above theorem. Let $\beta>0$. Then, almost surely, for all $t\geq 0$,
\begin{equation*}
\begin{aligned}
\e^{\beta t}Y_t^2 = \,\, &Y_0^2+\int_{]0,t]}\beta\e^{\beta s} Y_{s}^2 ds+2\int_{]0,t]} \e^{\beta s}Y_{s-}d(A+M)_s
+\int_{]0,t]} \e^{\beta s}d<M^{c},M^{c}>_s\\
&\quad + \sum_{0<s\leq t}\e^{\beta s}(Y_s-Y_{s-})^2+
\int_{[0,t[} 2\e^{\beta s}Y_{s}d(B)_{s+}
+\sum_{0\leq s<t}\e^{\beta s}(Y_{s+}-Y_{s})^2.
\end{aligned}
\end{equation*}
\end{Corollary}
\begin{proof} It suffices to apply 
%change of variables formula from
Gal'chouk-Lenglart's formula  with $n=2$, $F(x,y)=xy^2$, $X^1_t=\e^{\beta t}$, and $X^2_t=Y_t$. 
Indeed, by applying Theorem \ref{Thm_Ito} and by noting  that the local martingale part and the purely discontinuous part of $X^1$ are both equal to $0$, we obtain
\begin{equation*}
\begin{aligned}
\e^{\beta t}Y_t^2&=Y_0^2+\int_{]0,t]}\beta\e^{\beta s} Y_{s}^2 ds+
2\int_{]0,t]} \e^{\beta s}Y_{s-}d(A+M)_s
+\frac 1 2\int_{]0,t]} 2\e^{\beta s}d<M^{c},M^{c}>_s\\
&+ \sum_{0<s\leq t}\e^{\beta s}\big(Y_s^2-(Y_{s-})^2-2Y_{s-}(Y_s-Y_{s-})\big)\\
&+\int_{[0,t[} 2\e^{\beta s}Y_{s}d(B)_{s+}
+\sum_{0\leq s<t}\e^{\beta s}\big((Y_{s+})^2-(Y_{s})^2-2Y_s(Y_{s+}-Y_{s})\big).
\end{aligned}
\end{equation*}
The desired expression follows as $Y_s^2-(Y_{s-})^2-2Y_{s-}(Y_s-Y_{s-})=(Y_s-Y_{s-})^2$ and 
$(Y_{s+})^2-(Y_{s})^2-2Y_s(Y_{s+}-Y_{s})=(Y_{s+}-Y_s)^2.$
\end{proof}
%
%

%%%% "Comparison theorem"
%\subsubsection{Comparison}
%MODIFICATIONS CAS AVEC UN PROCESSUS DE POISSON
\begin{Proposition}[BSDE with "generalized" driver vs. BSDE]\label{compref} %%%satisfies Assumption $\eqref{Royer}$. 
Let $f$ be a predictable Lipschitz driver satisfying Assumption  \ref{Royer}. Let $A$  be a nondecreasing (resp. nonincreasing) right-continuous predictable process in ${\cal S}^2$ with $A_0=0$
 and let $C$ be a nondecreasing (resp. nonincreasing) right-continuous adapted purely discontinuous process in ${\cal S}^2$ with $C_{0-}=0$. Let $(Y,Z,k) \in {\cal S}^2 \times \mathbb{H}^2\times \mathbb{H}^2_\nu$  satisfy a.s. for all $t \in [0,T]$,
% \begin{equation}\label{dy}
%  -d  Y_t  \displaystyle =  f(t,Y_{t}, Z_{t}, k_t)dt + dA_t + d C_{t-}-  Z_t dW_t-\int_E k_t(e)\tilde{N}(dt,de)
%  \end{equation}
  %
%  {\bf 
%  (J' AI SUPPRIME UNE EQUATION)
%  }
  %
% in the sense that, for each $\tau\in\stopo$, the equality
 \begin{equation}\label{dy}
 Y_{t}= Y_{T}+ \int_{t} ^T f(s,Y_{s}, Z_{s},k_s)ds + A_T - A_{t} + C_{T-}- C_{t-}- \int_{t} ^T Z_{s}dW_{s}-\int_{t} ^T \int_E k_{s}(e)\tilde{N}(ds,de).
 \end{equation}
Then the process $(Y_t)$ is a strong ${\cal E}^f$-supermartingale (resp. ${\cal E}^f$-submartingale).

%%%%on enlève cette partie pour l'instant: cf. Remarque sur les E-surmartingales fortes
%Moreover, let  $ \tau \in \T_0$ be such that $\tau \leq \theta$ a.s.\,
%Suppose that  $Y_{\tau} = {\cal E}_{\tau ,\theta}(Y_{\theta}) $ a.s.\, Then, $A_s +  C_{s^-}$ is constant on a.s. %$[\tau, \theta]$  (\textit{quel est exactement le sens de cette phrase???)}. The process $Y$ is thus a strong  
%${\cal E}^f$-martingale on $[\tau, \theta]$. 
\end{Proposition}

\begin{proof} We address the case where $A$ and $C$ are nondecreasing. Let  $ \tau, \theta \in \T_0$ be such that $\tau \leq \theta$ a.s.\,
Let us show that $Y_{\tau} \geq {\cal E}^f_{\tau ,\theta}(Y_{\theta}) $ a.s.\\
We denote by $(X,\pi,l)$ the solution to the BSDE associated with driver $f$, terminal time $\theta$, and terminal condition $Y_{\theta}$; then ${\cal E}^f_{\tau ,\theta}(Y_{\theta})=  X_{\tau }$ a.s. (by definition of ${\cal E}^f$).  \\ 
%that is
% \begin{eqnarray*}
% -dX^{\tau}_{s}= g(s,X^{\tau}_{s}, \pi_s^{\tau}, k_s^{\tau})ds  - \pi_s^{\tau} dW_{s} - \int_{\R^*} k^{\tau}_s(u) \tilde{N}(ds,du);
%&& X^{\tau}_{\tau}= Y_{\tau}.
% \end{eqnarray*}
Set $\bar Y_t = Y_t - X_t$, $\bar Z_t = Z_t - \pi_t$ and $\bar k_t=k_t-l_t$. Then
  $$
  -d \bar Y_t  \displaystyle =  h_t dt + dA_t + d C_{t-}- \bar Z_t dW_t-\int_E \bar k_t(e) \tilde N(dt,de), \quad
   \bar Y_{\theta}  = 0,
   $$
   where $h_t:=f (t, Y_{t-}, Z_t, k_t) - f(t, X_{t-}, \pi_t, l_t)$. 
By the same arguments as those of  the proof of the comparison theorem for BSDEs with jumps (cf.  \cite[Thm. 4.2]{QuenSul}, or \cite{R}), using Assumption \ref{Royer} on $f$, we can show  that 
\begin{equation}\label{eq2}
h_t \geq  \delta_t \bar Y_t + \beta_t \bar Z_t + \langle \gamma_t\,,\, \bar k_t  \rangle_\nu, \;\; 0 \leq t \leq T,\quad 
dP \otimes dt-{\rm a.e.}\;
\end{equation}
where $\gamma_t:= \theta_t^{X_{t-}, \pi_t, k_t, l_t}$ and where $\delta$ and $\beta$ are predictable bounded processes (which can be expressed as increment rates of $f$ with respect to $y$ and $z$).\\
Let $\Gamma_{\tau,\cdot}$ be  the unique solution of the following forward SDE
\begin{equation}\label{eq4b}
d \Gamma_{\tau,s}  = \displaystyle \Gamma_{\tau,s-} \left[ \delta_s ds + \beta_s d W_s + \int_E\gamma_s(e)\tilde N(ds,de)\right] ; \;\; 
\Gamma_{\tau,\tau}  = 1. 
\end{equation} 
%%%\int_{E} \gamma_s(e) \tN(ds,de)
Suppose for a while that we have shown 
\begin{equation}\label{eq3}
  \Gamma_{\tau,\tau}\bar Y_\tau \geq E \left[  \int_\tau^\theta \Gamma_{\tau,s-}\, dA_s +  \int_\tau^\theta \Gamma_{\tau,s}\,dC_{s} \mid \FC_{\tau} \right]\;  {\rm a.s.}
\end{equation}
Then, since $\Gamma_{\tau,s} \geq 0$ and $\Gamma_{\tau,\tau}=1$, we have $\bar Y_\tau \geq 0$ a.s., that is $Y_{\tau} \geq X_{\tau}={\cal E}^f_{\tau ,\theta}(Y_{\theta}) $ a.s.\,, which is the desired result.
It remains to show \eqref{eq3}. 
To simplify the notation, we denote  $\Gamma_{\tau,s}$ by $\Gamma_s$ for $s \geq \tau$. 
%%By the  It\^o product formula (REF ...),  we have: A VOIR
%%%\begin{align*}
%-d(\bar Y_s \Gamma_s) & =- \bar Y_{s^-}  \Gamma_{s^-}[ \delta_s ds + \beta_s dW_s +  \int_\RB \gamma_s(u) \tN (du,ds) ]  
%+\Gamma_{s^-} h_s ds  +...\\
% & \quad \Gamma_{s^-}[ dA_s + d C_{s^{-}} -\bar Z_s dW_s -   \int_\RB \bar k_s(u) \tN (du,ds)] - \bar Z_s \Gamma_s \beta_s ds - \Gamma_{s^-} \int_\RB \gamma_s(u) \bar k_s(u) N (du,ds).
% \end{align*}
%Using inequality (\ref{eq2}) together with the non negativity of $\Gamma$, and doing ....computations, we derive that 
%%$$
%%-d(\bar Y_s \Gamma_s)
 %% \geq   \Gamma_s (dA_s + d C_{s^{-}}) - d M_s, 
%% $$
%% where $M$ is a martingale (since $\Gamma_{t,.}$ $\in$ $S^2$ and since $(\delta_t)$ and $(\beta_t )$ are bounded). By integrating between $\tau$ and $\theta$  and by taking the conditional expectation, we derive inequality (\ref{eq3}). 
%%The proof is thus complete.  
%%%%%%%%%%%%%%   
%%%%%%%%%%%%%%%%%%%%%%%%%%%%%%%
%%%%Nous enlevons la partie qui sui pour l'instant (cf. aussi Remarque sur les E-surmartingales fortes)%%%%% 
%Moreover, in the case when $\gamma_t >-1$, then $\Gamma_{\tau,.} > 0$. If we have additionally that 
%$Y_{\tau} = {\cal E}_{\tau ,\theta}(Y_{\theta}) $ a.s.\,, then inequality \eqref{eq3} yields that 
%the non decreasing process $A_s +  C_{s^-}$ is constant on a.s. $[\tau, \theta]$, which implies
%$Y=X$ on $[\tau, \theta]$. The second assertion thus follows.
%%%%%%%%%%%%%%%%%%%%%%%%%%%%%%%
%To simplify the notation, let us denote $\Gamma_{\tau,s}$ by $\Gamma_s$ for $s \geq \tau$.
We use that $\bar Y$ is a strong optional semimartingale with decomposition  $\bar Y=M^1+A^1+B^1$, where $M^1_t=\int_0^t \bar Z_s dW_s+\int_0^t\int_E  \bar k_s(e) \tilde N(ds,de)$, $A^1_t:=-\int_0^t h_s ds-A_s$, and $B^1_t:=-C_{t-}$,  
and we apply Gal'chouk-Lenglart's formula  from Theorem \ref{Thm_Ito} with $n:=2$, $X^1:=\bar Y$, $X^2:= \Gamma$, and $F(x^1,x^2):=x^1x^2$. We obtain    %It\^o product formula (REF ...),  we have: A VOIR
%%%%%%%%%%%%%%%
\begin{equation}\label{eq_Ito_00}
\begin{aligned}
\Gamma_\tau\barY_\tau&= -\int_\tau^\theta \Gamma_{s-}(\bar Z_s+\bar Y_{s-}\beta_s) dW_s-  
\int_\tau^\theta \Gamma_{s}(\bar Y_{s-}\delta_s+\bar Z_s\beta_s -h_s)ds\\
&+\int_\tau^\theta \Gamma_{s-}dA_s+\int_\tau^\theta \Gamma_{s}dC_s-\int_\tau^\theta\int_E \Gamma_{s-}(\bar k_s(e)+\bar Y_{s-}\gamma_s(e))\tilde N(ds,de)
-\sum_{\tau \leq s\leq \theta} \Delta \Gamma_s\Delta Y_s. 
\end{aligned}
\end{equation}
%Now, $\Delta \Gamma_s= \Gamma_{s-}\gamma_s\Delta N_s$  and $\Delta \bar Y_s= -\Delta A_s +\bar k_s\Delta N_s$.
By using the fact that $A_\cdot$ and $N(\cdot,de)$ 
do not have common jumps, we get $\sum_{\tau \leq s\leq \theta} \Delta \Gamma_s\Delta Y_s=\int_\tau^\theta \int_E\Gamma_{s-} \gamma_s(e) \bar k_s(e) N(ds,de).$ By replacing this expression in equation \eqref{eq_Ito_00} and by doing some computations, we obtain
\begin{equation}
\begin{aligned}
\Gamma_\tau\barY_\tau&= -\int_\tau^\theta \Gamma_{s-}(\bar Z_s+\bar Y_{s-}\beta_s) dW_s-  
\int_\tau^\theta \Gamma_{s}(\bar Y_{s-}\delta_s+\bar Z_s\beta_s +\langle\gamma_s,\bar k_s\rangle_\nu-h_s)ds\\
&+\int_\tau^\theta \Gamma_{s-}dA_s+\int_\tau^\theta \Gamma_{s}dC_s-\int_\tau^\theta\int_E \Gamma_{s-}(\bar k_s(e)+\bar Y_{s-}\gamma_s(e)+\gamma_s(e)\bar k_s(e))\tilde N(ds,de). 
\end{aligned}
\end{equation}
%By taking the conditional expectation, we get 
Now, the stochastic integral with respect to "$dW_s$" in the above equation is a martingale (since $\Gamma$ $\in$ $S^2$, $\bar Z\in\H^2$, $\bar Y\in\mathcal S^2$, and $\beta$ is bounded). The stochastic integral with respect to the Poisson random measure is also a martingale. By taking the conditional expectation and by using the inequality \eqref{eq2}, we derive \eqref{eq3}.  
%By using the inequality (\ref{eq2}),  %together with the non-negativity of $\Gamma$, 
%we  obtain  
%$$
%d(\bar Y_s \Gamma_s)
%  =   \Gamma_{s}(\bar Z_s+\bar Y_{s-}\beta_s)dW_s-\Gamma_s (dA_s + d C_{s}). 
% $$
The proof is thus complete.
\end{proof}

%Before giving the proof of  Proposition \ref{Prop_Banach_space}, we make a remark which will be used in the proof.
%\begin{Remark}\label{Rmk_monotone_convergence}
%Let $(\phi_n)$ be a non-negative sequence of optional processes. We assume that the sequence $(\phi_t^n(\omega))$ is non-decreasing for all $t\in[0,T]$, for all $\omega\in\Omega$. We set $\phi_t(\omega):= \lim_n \phi_t^n(\omega)=\sup_n \phi_t^n(\omega).$ Then, we have $\limup E(\esssup_{\tau\in\stopo} |\phi^n_\tau |^2)=E(\esssup_{\tau\in\stopo} |\phi_\tau |^2).$ The result is a consequence of the monotone convergence theorem, combined with the following observation:\\
%$\esssup_{\tau\in\stopo} \sup_n |\phi_\tau^n|^2 =\sup_n\esssup_{\tau\in\stopo}  |\phi_\tau^n|^2  $ a.s.   
%\end{Remark}
\noindent \textbf{Proof of Proposition \ref{Prop_Banach_space}:}
We first show that 
$\vertiii{\cdot}_{{\cal S}^{2}}$ is a norm on the space of optional processes. 
The positive homogeneity  and the triangular inequality are easy to check. Suppose now that $\phi\in{\cal S}^{2}$ is such that 
$\vertiii{\phi}_{{\cal S}^{2}}=0.$ Then, $\esssup_{S\in\stopo} |\phi_S |^2=0$ a.s., which, by definition 
of the essential supremum, implies that 
$|\phi_S|^2=0$ a.s. for all $S\in\stopo$. By a classical result of the General Theory of Processes (\cite[Theorem IV.84]{DM1}),
%applying the optional section theorem (cf. also Prop. \ref{optional_section}),
 we obtain that the process $\phi$ is indistinguishable from the null process, that is
 $\phi_t=0$, $0 \leq t \leq T$ a.s. 
 We conclude that $\vertiii{\cdot}_{{\cal S}^{2}}$ is a norm on ${\cal S}^{2}$.\\
Let us prove that the space $({\cal S}^{2}, \vertiii{\cdot}_{{\cal S}^{2}})$ is complete. We only sketch the proof since its main steps are similar to those of the  proof of the completeness of the space $(L^2, \|\cdot\|_{L^2})$.  Let $(\phi^n)$  be a Cauchy sequence in ${\cal S}^{2}$ for the norm  $\vertiii{\cdot}_{{\cal S}^{2}}.$ We extract a subsequence  $(\phi^{n_{k}})_{k\in\N}$ such that 
%\begin{equation}\label{eq_1_Prop_Banach_space}
%\vertiii{\phi^{n_{k+1}}-\phi^{n_{k}}}_{{\cal S}^{2}}\leq \frac 1 {2^k}, \text{ for all } k\in\N.
%\end{equation}
$\vertiii{\phi^{n_{k+1}}-\phi^{n_{k}}}_{{\cal S}^{2}}\leq \frac 1 {2^k}$, for all  $k\in\N$.
Setting $g_t^n:= \sum_{k=1}^n |\phi_t^{n_{k+1}}-\phi^{n_{k}}|$ for each $n$, 
by the triangular inequality, we derive that 
%
%For $n\in\N$, for $t\in[0,T]$, for  $\omega\in\Omega$, we set
%$g_t^n(\omega):= \sum_{k=1}^n |\phi_t^{n_{k+1}}(\omega)-\phi^{n_{k}}(\omega)|.$ We note that $g^n\in \mathcal{S}^2$ for all $n\in\N.$ Moreover, the sequence $(\vertiii{g^n}_{{\cal S}^{2}})$ is bounded. Indeed, by the triangular inequality and property \eqref{eq_1_Prop_Banach_space}, we have 
%
$\vertiii{g^n}_{{\cal S}^{2}}\leq \sum_{k=1}^n \vertiii{ \phi_t^{n_{k+1}}-\phi^{n_{k}}}_{{\cal S}^{2}}\leq  \sum_{k=1}^n \frac 1 {2^k}\leq \sum_{k=1}^\infty \frac 1 {2^k}=1. $
We set $g_t:=\limup g_t^n$, for all $t\in[0,T]$ (the limit exists in $[0,+\infty]$ as the sequence $(g_t^n)_n$ is non-negative non-decreasing).  
Being the limit of optional processes, the process $g$ is optional.  
%Moreover, $g \in\mathcal{S}^2.$ 
Since $\esssup_{\tau\in\stopo} \sup_n |g_\tau^n|^2 =\sup_n\esssup_{\tau\in\stopo}  |g_\tau^n|^2  $ a.s.\,, using the monotone convergence theorem, we derive that $\vertiii{g}_{{\cal S}^{2}}=\limup\vertiii{g^n}_{{\cal S}^{2}}$. 
 As the sequence $(\vertiii{g^n}_{{\cal S}^{2}})$ is bounded by $1$, we get $\vertiii{g}_{{\cal S}^{2}}\leq 1.$ 
We then adapt the arguments from the proof of the completeness of $(L^2, \|\cdot\|_{L^2})$ to show that 
 $\lim_n \vertiii{g-g^n}_{{\cal S}^{2}}=0,$ 
and that $\vertiii{\phi-\phi^{n_{l}}}_{{\cal S}^{2}}\underset{l\to\infty}{\longrightarrow}0$, which concludes the proof.
\fproof

\noindent The following result of the optimal stopping theory is used 
%at the beginning of 
in the proof of Lemma \ref{f}. 
%By some results of optimal stopping theory (cf., e.g., \cite{EK} or Proposition B.1 in \cite{Kob}), we have 
\begin{Proposition} \label{pro}
 Let $(\barY (S))$ be the family defined 
 for $S \in \stopo$ by 
\begin{eqnarray} \label{ddeux-2deux}
\overline Y  (S):= \esssup_{\tau \in \stops} E[ \xi_{\tau} + \int_S^\tau f(u)du \mid \Fc_S], 
\end{eqnarray} 

 \begin{description}
 \item[(i)]  There exists a ladlag optional process $(\barY_t)_{t\in[0,T]}$ which aggregates the family $(\barY (S))$ (i.e. $\barY_S=\barY (S),$ for all $S\in\stopo$).\\
Moreover, the process $(\barY_t + \int_0^t f(u) du)_{t\in[0,T]}$ is a strong supermartingale.
 \item[(ii)] We have 
% $\dbarY_S= (\xi _S + \int_0^S f(u)du)\vee \dbarY_{S+} \quad {\rm a.s.},$
$\barY_S=\xi_S\vee \barY_{S+}$ a.s. 
   for all $S$.
 \item[(iii)] Furthermore, $\barY _{S+}=\esssup_{\tau >S} E[ \xi_{\tau} + \int_S^\tau f(u)du \mid \Fc_S] \quad {\rm a.s.},$  for all $S$.
 \end{description} 

%where  $(\dbarY (S))$ is the family defined in \eqref{Lemma_eq_numero_ajoute}.
\end{Proposition}
\begin{Remark}\label{sautY}
%Since by (ii), we have $\barY_S=\xi_S \vee \barY_{S+}$ a.s.\,, 
It follows from $(ii)$ that 
%%%\Delta C_t=\Delta C_t{\bf 1}_{\{Y_t= \xi_t\}}
$\Delta_+ \barY_{S}= {\bf 1}_{\{\barY_S= \xi_S\}}\Delta_+ \barY_{S}$ a.s.
\end{Remark}
%For completeness, we give hereafter a short (and not very difficult) proof  (cf. \cite{Maingueneau} when $\xi$ is left and right limited, and 
%\cite[Sect.B]{Kob} in the general case).  
\begin{proof} %For completeness, we give a short proof.
For completeness, we give here a short proof  (cf. \cite{Maingueneau} when $\xi$ is left- and right-limited, and 
\cite[Sect.B]{Kob} in the general case).  
For $S \in \stopo$, 
%we define $\overline Y  (S)$ by 
%\begin{eqnarray}\label{ddeux-2}
%\overline Y  (S):= \esssup_{\tau \in \stops} E[ \xi_{\tau} + \int_S^\tau f(u)du \mid \Fc_S], 
%\end{eqnarray} 
we define $\dbarY(S)$ by 
\begin{equation}\label{Lemma_eq_numero_ajoute}
\dbarY(S):=\barY(S)+\int_0^S f(u)du=\esssup_{\tau \in \stops} E[ \xi_{\tau} + \int_0^\tau f(u)du \mid \Fc_S].
\end{equation}
where the equality follows from the definition of $\overline Y  (S)$ (see \eqref{ddeux-2deux}).
%We note that the process $(\xi_t+\int_0^{t} f(u)du)_{t\in[0,T]}$ is progressive. Therefore, the family $(\dbarY(S))_{S\in\stopo}$ is a supermartingale family (cf. \cite[Remark 1.2 with Prop. 1.5]{Kob}). This observation, combined with \cite[Remark (b), page 435]{DM2}, gives the existence of a strong optional supermartingale (which we denote again by $\dbarY$) such that $\dbarY_S=\dbarY(S)\, \rm{ a.s.}$ for all $S\in\stopo$.  
%By Proposition \ref{pro} in the Appendix, there exists a strong optional ladlag supermartingale  (which we denote again by $\dbarY$) such that $\dbarY_S=\dbarY(S)\, \rm{ a.s.}$ for all $S\in\stopo$, and we also have 
%$\dbarY_S= (\xi _S + \int_0^S f(u)du)\vee \dbarY_{S+} \quad {\rm a.s.},$  for all $S$.
%By \eqref{Lemma_eq_numero_ajoute}, we thus derive that $\barY(S)=\dbarY(S)-\int_0^S f(u) du=\dbarY_S-\int_0^S f(u) du\, \rm{ a.s.} $ for all $S\in\stopo$.  
%%On the other hand, we know that almost all trajectories of the strong optional supermartingale $\dbarY$ are ladlag (cf. \cite{DM2}). 
%Thus, we get that the ladlag optional  process $(\barY_t)_{t\in[0,T]}=(\dbarY_t-\int_0^t f(u) du)_{t\in[0,T]}$ aggregates the family $(\barY(S))_{S\in\stopo}, $ {\bf and that $\barY_S=\xi_S\vee \barY_{S+}$ a.s. for all $S\in\stopo$.
%}
For $S \in \stopo$, define 
\begin{equation}\label{barbarY}
\dbarY ^+(S):=\esssup_{\tau >S} E[ \xi_{\tau} + \int_0^\tau f(u)du \mid \Fc_S].
\end{equation}
By some well-known results of optimal stopping theory (cf., e.g. \cite[Prop. D.3]{KS2} 
or \cite[Prop. 1.12]{Kob}), the family of random variables 
$(\dbarY ^+(S))$ is a supermartingale family which is right-continuous along stopping times in expectation. By classical results (cf., e.g., \cite{EK} 
 or \cite[Prop. 4.1]{KQR}), there exists a process $(\dbarY ^+_t)$ which aggregates this family.
By \cite[Prop. D.3]{KS2} (cf. also \cite[Prop. 1.9]{Kob}), we have 
\begin{equation}\label{E.v}
\dbarY (S)= (\xi _S + \int_0^S f(u)du)\vee \dbarY^+(S) \quad {\rm a.s.},
\end{equation}
for all $S\in\stopo$.
It follows that the process $(\dbarY_t)$ defined by
\begin{equation}\label{E.v2}
\dbarY_t:= (\xi_t +  \int_0^t f(u)du)\vee \dbarY^+_{t}
\end{equation}
aggregates the family $(\dbarY (S))$. Since $(\dbarY (S))$ is a supermartingale family, $(\dbarY_t)$ is a strong supermartingale.
Now, we know  (cf., e.g.,  \cite[Prop. 4.14]{Kob}, combined with \cite[Appendix A1, paragraph 1]{Marc})  that $\dbarY^+(S)=\dbarY(S+)$, for all $S\in\stopo,$  where $\dbarY (S+)$ denotes the right-hand limit of $\dbarY$ along stopping times at $S$, as defined, for instance,  in \cite[Def. 4.5]{Kob}. On the other hand, we know that the process $(\dbarY_{t})$ aggregates the family $(\dbarY(S))$, which entails that the process $(\dbarY_{t+})$ aggregates the family $(\dbarY(S+))$.  By using  Eq. \eqref{E.v}, we conclude that
\begin{equation}\label{(ii)}
\dbarY_S= (\xi _S + \int_0^S f(u)du)\vee \dbarY_{S+} \quad {\rm a.s.}
\end{equation}
for all $S$. By \eqref{Lemma_eq_numero_ajoute}, we  derive $\barY(S)=\dbarY(S)-\int_0^S f(u) du=\dbarY_S-\int_0^S f(u) du$ a.s., 
for all $S\in\stopo$.
%On the other hand, we know that almost all trajectories of the strong optional supermartingale $\dbarY$ are ladlag (cf. \cite{DM2}). 
The ladlag optional  process $(\barY_t)_{t\in[0,T]}=(\dbarY_t-\int_0^t f(u) du)_{t\in[0,T]}$ thus aggregates the family $(\barY(S))_{S\in\stopo}. $  
Moreover,  $(\barY_t+\int_0^t f(u) du)= (\dbarY_t)$ is a strong supermartingale, which gives $(i)$.
By using \eqref{(ii)}, we derive $(ii)$. 
%$\barY_S=\xi_S\vee \barY_{S+}$ a.s. for all $S\in\stopo$,
By  using \eqref{barbarY}, we obtain $(iii)$.
\end{proof}
%% version abrégée de la remarque 
%\begin{Remark} {\bf POURRAIT ON SUPPRIMER CETTE REMARQUE AFIN DE RACCOURCIR LE PAPIER?}
%In \cite{EK}, part $(i)$ of the above result  is shown by using a different  approach: 
%One shows first that the family $(\dbarY(S))_{S\in\stopo}$ is a supermartingale family (cf. \cite{EK},  or \cite[Remark 1.2 with Prop. 1.5]{Kob}), then one applies a  very fine aggregation result from the general theory of processes (cf. \cite[Remark (b), page 435]{DM2}), ensuring that any integrable supermartingale family can be aggregated by a strong supermartingale.
%%Now, a fine aggregation result of the general theory of processes (cf. \cite[Remark (b), page 435]{DM2}) ensures that any integrable supermartingale family can be aggregated by a strong supermartingale.
%%Applying this result to the supermartingale family $(\dbarY(S))_{S\in\stopo}$ gives the existence of a strong optional supermartingale which 
%%aggregates this family.
%\end{Remark}

\begin{example}[A toy example] \label{exe_intro}
 % we can compute explicitly the solution to the RBSDE with   
Let $(\xi_t)$ be a deterministic continuous decreasing bounded function. We set $Y_t:= \sup_{s\geq t} \xi_s=\xi_t$ and $A_t:=\xi_0-\xi_t$, for all $t\in[0,T]$.  It is well-known (cf.  \cite{ElKaroui97}, or the classical Skorokhod's problem as recalled  in  \cite{KS2}) that  $(Y,0,0,A)$ is the unique solution to the RBSDE with driver $f\equiv 0$ and (continuous) obstacle $\xi$. % is given by 
%$(Y,Z,k,A)= (\sup_%From classical results on RBSDEs with continuous obstacle (cf. \cite ) 
Let us now change the obstacle $\xi$ at a single point $t_0\in[0,T)$. More precisely, we consider a function  $\bar\xi$ such that  $\bar\xi_t=\xi_t$, for $t\neq t_0$, and $\bar\xi_{t_0}>\xi_{t_0}$. We note that $\bar\xi$ is r.u.s.c. but not right-continuous. %Before presenting a general result on existence and uniqueness of the 
In this very simple example,  we can  compute explicitly a solution to the RBSDE (defined in Definition \ref{def_solution_RBSDE}) with parameters  $(0,\bar\xi)$. We set $\bar Y_t:=\sup_{s\geq t} \bar\xi_s,$ for $t\in[0,T].$   We first rewrite $\bar Y$ in a different manner.   For $t>t_0$, we have $\bar Y_t=\xi_t=Y_t$. %Hence, $\bar Y_{t_0+}=\xi_{t_0}=Y_{t_0}.$ 
For $t\leq t_0$, we have  $\bar Y_t=\max(\sup_{s\geq t, s\neq t_0}  \xi_s, \bar\xi_{t_0})= \max(\sup_{s\geq t}  \xi_s, \bar\xi_{t_0})=\max(\xi_t,\bar\xi_{t_0})$. We set $t_1:= \sup\{s\geq 0: \xi_s\geq \bar\xi_{t_0}\}$, with the convention $\sup(\varnothing)=0$. We note  that $\bar Y_{t}=\bar\xi_{t_0}$, for $t\in[t_1,t_0],$ and $\bar Y_{t}=\xi_{t}=Y_t$, for $t\in[0,t_1)$.  We define $\bar C_t:=  (\bar\xi_{t_0}-\xi_{t_0}){\bf 1}_{t\geq t_0}$, for $t\in[0,T]$. We see that $\bar C$ is non-decreasing, cad-lag, purely discontinuous (in fact, $\bar C$ has one single jump), and it  satisfies the minimality condition \eqref{RBSDE_C}. We  now consider the following  two cases: $(i)$ the case $\bar\xi_{t_0}\geq\xi_{t_0}$, and $(ii)$ the case $\bar\xi_{t_0}<\xi_{t_0}.$  In the case $(i)$, we have  $t_1=0$;  we set  $\bar A_t:=0$,  for $t\in[0,t_0]$, and 
$\bar A_t:=\xi_{t_0}-\xi_t$, for $t\in(t_0,T]$. In the case $(ii)$, we have $t_1>0$ and $\xi_{t_1}=\bar\xi_{t_0} $; 
we define $\bar A_t$ by $\bar A_t:=\xi_0-\xi_t$, for $t\in[0,t_1)$, $\bar A_t:=\xi_0-\xi_{t_1}$, for $t\in[t_1,t_0]$, and $\bar A_t:=\xi_0-\xi_{t_1}+\xi_{t_0}-\xi_t$, for $t\in(t_0,T]$. %We see that $\bar C$ is non-decreasing, cad-lag, purely discontinuous, and it  satisfies the minimality condition \eqref{RBSDE_C}. 
In both cases, the function $\bar A$ is non-decreasing, continuous, and  it  satisfies the minimality condition \eqref{RBSDE_A}.  Moreover, it can be easily checked that,  $\bar Y_t=\bar \xi_T+ \bar A_T-\bar A_t+\bar C_{T-}-\bar C_{t-}$, for all $t\in[0,T]$. We conclude that $(\bar Y, 0, 0, \bar A, \bar C)$ is a solution to the reflected BSDE with parameters $(0,\bar \xi)$. We prove in Lemma \ref{f} that $(\bar Y, 0, 0, \bar A, \bar C)$ is the unique solution.   We notice that  $\bar Y$ has a jump on the right at $t_0$;  the  size $\Delta_+ \bar Y_{t_0}$ of the jump satisfies  $\Delta_+ \bar Y_{t_0}:= \bar Y_{t_0+}-\bar Y_{t_0}=\xi_{t_0}-\bar\xi_{t_0}=-(C_{t_0}-C_{t_{0-}}).$   
%The reader could check some of the properties stated in the subsequent propositions/theorems on this toy example. 
   
% in such a way that the 
\end{example}
\begin{example}[Counter-example]\label{exe} 
%Suppose that $\gamma$ is a real constant, $\delta = \beta = 0$, $\varphi=0$, 
Let $\nu(du) := \delta_1(du)$, where $\delta_1$ denotes the Dirac measure at $1$. 
The process $N_t := N([0,t] \times \{1\})$ is then a Poisson process with parameter $1$, and we have
$\tN_t := \tN([0,t] \times \{1\})  = N_t - t.$ 
Let the driver $f$ be given by
$
f(t,y,z,\ell): = \langle -1, \ell \rangle_\nu=- \ell (1) .
$
 We introduce  the associated adjoint process $\Gamma_{t,.} $, defined for each $r\in [t,T]$ by
$
\Gamma_{t,r}   
= {\bf 1}_{\{N_r - N_t=0\}}e^{  r-t}.
$
Let the pay-off process $\xi$ be given by
$\xi_t:= {\bf 1}_{  \{N_t \geq 1\}} e^{-t}, \text{ for all } t\in[0,T].$ Note that $\xi$ is adapted and right-continuous. 
%Let ${\cal E}^f_{s, \tau}(\xi_{\tau})$ be the solution of the BSDE associated 
%with driver $f$ terminal time $\tau \geq t$ and terminal condition $\xi_{\tau}$.
%%%%%%%%%%%%%
By the representation property for linear BSDEs with jumps (\cite[Thm.3.4]{QuenSul}) and  classical computations, we get
%, for all $t\in[0,T]$, for all $\tau\in\stopt$,
 $${\cal E}^f_{t, \tau}(\xi_{\tau}) = E[ \Gamma_{t, \tau}\xi_{\tau} \, | \, {\cal F}_t]=  \e^{-t}E[{\bf 1}_{  \{N_\tau-N_t=0\}}{\bf 1}_{  \{N_\tau \geq 1\}} \, | \, {\cal F}_t]=\e^{-t}{\bf 1}_{ \{N_t \geq 1\}}E[{\bf 1}_{  \{N_\tau-N_t=0\}}\, | \, {\cal F}_t],$$
 for all $t\in[0,T]$, for all $\tau\in\stopt$.
We deduce that
%We denote by $Y$ the value 
  $Y_t: =   {\rm ess} \sup_{\tau \in \T_{t,T}}{\cal E}^f_{t, \tau}(\xi_{\tau})=\e^{-t}{\bf 1}_{ \{N_t \geq 1\}}=\xi_t,$ for all $t \in [0,T]$ (as $E[{\bf 1}_{  \{N_\tau-N_t=0\}}\, | \, {\cal F}_t]\leq 1$ and the upper bound is attained for $\tau=t$).  Let us focus on the optimal stopping problem at time $t=0$. 
The above computations imply that,  for $t=0$,  
$Y_0:={\rm ess} \sup_{\tau \in \T_{0,T}}{\cal E}^f_{0, \tau}(\xi_{\tau})=\xi_0=0$. Moreover,  the essential supremum (at time $0$) is attained at any stopping time $\tau\in\stopo$ (indeed,     ${\cal E}^f_{0, \tau}(\xi_{\tau})=0$, for all $\tau\in\stopo$).  This is true, in particular, for the stopping time $\hat{\tau}$ defined by $\hat{\tau}:=T.$ However, we will see that the process $Y$ (computed above) is not an $\mathcal{E}^f$-martingale on $[0,\hat{\tau}]$. To do so, let us denote by $X$ the (first component) of the solution to the BSDE with driver $f$, terminal time $T$, and terminal condition $Y_T=\xi_T$.  For $u\in[0,\hat{\tau}]=[0,T]$, we have 
$X_u={\cal E}^f_{u,T}(\xi_{T})=\e^{-u}{\bf 1}_{ \{N_u \geq 1\}}E[{\bf 1}_{  \{N_T-N_u=0\}}\, | \, {\cal F}_u]=
\e^{-u}{\bf 1}_{ \{N_u \geq 1\}}P[ N_T-N_u=0]=\e^{-T}{\bf 1}_{ \{N_u \geq 1\}}.$
Hence, for $u\in(0,T)$,  we have $Y_u=\e^{-u}>\e^{-T}=X_u$ on the set $ \{N_u \geq 1\}$.   Hence, the processes  $X$ and $Y$
 are not indistinguishable. \\
 Let us also note that in this example ${\cal E}^f$ is not strictly monotonous. To see this, we consider  
 $\xi^1:=0$ and $\xi^2:=\xi_T=\e^{-T}{\bf 1}_{ \{N_T \geq 1\}}$.  We have $\xi^1\leq \xi^2$ and
 ${\cal E}^f_{0,T}(\xi^1)={\cal E}^f_{0,T}(0)=0={\cal E}^f_{0,T}(\xi_T)={\cal E}^f_{0,T}(\xi^2)$. However, $\xi_T\neq 0$ with a positive  probability.  
\end{example}

\begin{example}\label{toy2}
  Let $(\xi_t)$ be an RCLL deterministic 
%continuous decreasing
 bounded function, increasing on $[0, t_0[$, decreasing on $[t_0, T]$, and supposed to be continuous on $[0,T]$ except at  $t_0 \in ]0, T[$ with 
 $\xi_{t_0} < \xi_{t_0{-}}$. Note that the function $\xi$ is  not l.u.s.c. at time $t_0$.
 We set $Y_t:= \sup_{s\geq t} \xi_s$ and $A_t:=Y_0-Y_t$, for all $t\in[0,T]$.  By the classical 
 Skorokhod's problem (cf. also \cite{ElKaroui97}), $(Y,0,0,A)$ is the unique solution to the RBSDE with driver $f\equiv 0$ and obstacle $\xi$. 
  We have $Y_t =\xi_{t_0{-}}$, if $t< t_0$, and $Y_t = \xi_{t}$, if $t\geq t_0$. 
 Let  $\tau^*_0 := \inf \{u \geq 0 \colon Y_u = \xi_u \}$. We have $\tau^*_0=t_0$. Note that here $\Delta A_{\tau^*_0}=\Delta A_{t_0}=\xi_{t_0-}-\xi_{t_0}>0$.  
 However, $\tau^*_0=t_0$ is not optimal for 
 $Y_0= \sup_{s\geq 0} \xi_s= \xi_{t_0{-}}$ because $\xi_{t_0} < \xi_{t_0{-}}$. 
 In fact, there does not exist an optimal stopping time for $Y_0$.
 
 Let us now consider the case  where,  %on $[t_0, T]$,
  instead of being decreasing on $[t_0, T]$, the function $\xi$ is increasing on $[t_0, T]$ with 
 $\xi_T= \xi_{t_0{-}}$. Note that, again,  the function $\xi$ is not l.u.s.c. For each $t \in [0,T]$, $Y_t= \sup_{s\geq t} \xi_s=  \xi_{t_0{-}}$. The process $A$ is  constant equal to $0$, and $\tau^*_0=T$ is optimal for 
 $Y_0$ (and also for $Y_t$, for all $t \in [0,T]$).  
%  Note that $Y_t=Y_0-A_t$, where $A$ is the non decreasing function given by $A_t = 0$ if $t< t_0$, and $A_t = \xi_{t_0^{-}} - \xi_{t}$ 
% if $t\geq t_0$. Note that $A$ is not continuous at time $t_0$.
 
\end{example}


\begin{thebibliography}{ab}
\bibitem{ALM} Alario-Nazaret M., J.-P. Lepeltier,  and B. Marchal (1982): Dynkin games, in Stochastic Differential Systems (Bad Honnef Workshop on stochastic processes), {\em Lecture Notes in Control and Information Sciences} 43, Springer-Verlag, Berlin,  23-32.
 

%%\bibitem{Bayraktar} E. Bayraktar and S. Yao: Optimal stopping for Non-linear Expectations. 
 %%{\em Stochastic Processes and Their Applications} (2011), 121 (2), 185-211 and 212-264.

%%\bibitem{Bayraktar-2}  E. Bayraktar, I. Karatzas  and S. Yao: Optimal Stopping for Dynamic Convex Risk Measures,  {\em Illinois Journal of Mathematics }, 54 (3), 1025-1067 (Fall 2010).

%%\bibitem{BBP}  Barles G., R. Buckdahn and E. Pardoux: Backward Stochastic Differential Equations and integral-partial differential equations,  {\em Stochastics and Stochastics Reports}, 1997.
\bibitem{BEK}
Barrieu P.  and N. El Karoui (2004):   Optimal derivatives design
under dynamic risk measures, {\em Mathematics of Finance}, Contemporary
Mathematics (A.M.S. Proceedings), 13-26.

\bibitem{Bayraktar-2}  
Bayraktar E., I. Karatzas,  and S. Yao (2010): Optimal Stopping for Dynamic Convex Risk Measures,  {\em Illinois Journal of Mathematics }, 54 (3), 1025-1067.



%\bibitem{Bion} Bion-Nadal J. (2008): Dynamic Risk measures: time consistency and risk measures from BMO martingales, 
 %{\em Finance} \& {\em Stochastics}, 12, 219-244. 
 
\bibitem{Bismut2} Bismut J.-M. (1973): Conjugate convex functions in optimal stochastic control, \textit{Journal of Mathematical Analysis and Applications} 44(2), 384-404.
%{\bf J'AI SUPPRIME le 2eme  Bismut pour raccourcir}
%\bibitem{Bismut1} Bismut J.-M. (1976): Théorie probabiliste du contrôle des diffusions, \textit{Memoirs of the American Mathematical Society} 4(167).  
%
\bibitem{Bouchard} 	   
	Bouchard B., D. Possamaï, and  X. Tan (2016): A general Doob-Meyer-Mertens decomposition for $g$-supermartingale system,  \textit{Electronic Journal of Probability} 
    21, paper no. 36, 21 pages. %preprint, available at  http://arxiv.org/abs/1505.00597.

\bibitem{CM} Cr\'epey S. and A. Matoussi (2008): Reflected and doubly reflected BSDEs with jumps: a priori estimates and comparison, {\em Annals of Applied Probability},  18(5), 2041-2069.

%\bibitem{D} Delbaen F. (2006): The structure of $m$-stable sets and in particular of the set of risk neutral measures, 
 %{\em Lecture Notes in Mathematics},  Vol. 1874, 215-258.

\bibitem{DelLen2} Dellacherie C. and E. Lenglart (1981): Sur des problèmes de régularisation, de recollement et d'interpolation en théorie des processus, \textit{Sém. de Proba. XVI}, lect. notes in Mathematics, 920, 298-313, Springer-Verlag.
 
\bibitem{DM1}
Dellacherie C. and  P.-A. Meyer (1975):
 \textit{Probabilit\'e et Potentiel, Chap. I-IV}. Nouvelle \'edition. Hermann. 
% {\bf MR}{0488194}

\bibitem{DM2}
Dellacherie C. and  P.-A. Meyer (1980):
 \textit{Probabilit\'es et Potentiel, Th\'eorie des Martingales, Chap. V-VIII}. Nouvelle \'edition. Hermann.
%{\bf MR}{0566768}

\bibitem{DQS2}  
% 
Dumitrescu R., M.-C. Quenez,  and A. Sulem (2016): Generalized Dynkin Games and Doubly reflected BSDEs with jumps, to appear in {\em Electronic Journal of Probability}.

%


%\bibitem{DQS}
%Dumitrescu R., Quenez M.-C.  and A. Sulem  (2014):
%Optimal stopping for dynamic risk measures with jumps and obstacle problems, arXiv:1404.4600.


\bibitem{EK}
El Karoui N. (1981):
Les aspects probabilistes du contr\^ole stochastique. \textit{\'Ecole d'\'et\'e de Probabilit\'es de Saint-Flour IX-1979 Lect. Notes in Math.}
\textbf{876}, 73-238.
{\bf MR}{0637469 }

%\bibitem{KHM}
%El Karoui N., Hamadène S. and A. Matoussi (2009):
%Backward Stochastic Differential 
%Equations and Applications, In: Indifference Pricing: Theory and Applications, Ed. R. Carmona, Princeton Series in Financial Engineering, 267-320.
%



\bibitem{ElKaroui97}
El Karoui N., Kapoudjian C.,
 Pardoux E., Peng S. and M.-C. Quenez (1997):
Reflected solutions of Backward SDE's and related obstacle problems 
for PDE's,
{\em The Annals of Probability}, 25(2), 702-737.


\bibitem{KPQ}
El Karoui N., Peng S. and M.C. Quenez (1997): 
Backward Stochastic Differential Equations in Finance, {\em Mathematical Finance} 7(1), 1-71. 
 
\bibitem{EQ96}
El Karoui N. and M.-C. Quenez (1997): %(1996):
Non-linear Pricing Theory and Backward Stochastic Differential 
Equations,  
%{\em Financial Mathematics}, 
Lect. Notes in Mathematics 1656, %Bressanone, 1996, 
Ed. W. Runggaldier, Springer.

  

\bibitem{Essaky} Essaky H. (2008): Reflected backward stochastic differential equation with jumps and RCLL obstacle. Bulletin des Sciences Math\'ematiques  132, 690-710.

%%\bibitem{FS} F\"ollmer H. and A. Schied (2011)~: {\em Stochastic Finance. An introduction in discrete-time}, Third edition, 
%Berlin, de Gruyter,  Studies in Mathematics.

%\bibitem{FR2}
%  Frittelli M.  and E. Rosazza-Gianin (2004): {\it Dynamic convex risk
%measures}, In G. Szeg\" o ed., Risk Measures in the 21st Century,
%John Wiley \& Sons, Hoboken, NJ,   227-248.

\bibitem{Galchouk} 
Gal'chouk L. I. (1981)~: Optional martingales,
{\em Math. USSR Sbornik} 40(4), 435-468.

\bibitem{Imkeller2} Grigorova M., P. Imkeller, Y. Ouknine, and  M.-C. Quenez (2016): 
Doubly
Reflected BSDEs and Dynkin games: beyond the right-continuous case, working paper. 

\bibitem{Ham} Hamadène S. (2002) Reflected BSDE's with discontinuous barrier and application, \textit{Stochastics and Stochastic Reports} 74(3-4), 571-596.  

\bibitem{HO1}  Hamad\`ene S.  and Y. Ouknine (2003): Backward stochastic differential equations with jumps and random obstacle, {\em Electronic Journal of  Probability} 8, 1-20.

\bibitem{HO2}  Hamad\`ene S.  and Y. Ouknine (2015): Reflected backward SDEs with general jumps, {\em Teor. Veroyatnost. i Primenen.} , 60(2),  357-376. 



%\bibitem{KS1}
%I. Karatzas\ and\ S. E. Shreve, {\it Brownian motion and stochastic calculus}, Graduate Texts in Mathematics, 113, Springer, New York, 1988. MR0917065 (89c:60096)
%
\bibitem{KS2} 
Karatzas I. and S. E. Shreve, Methods of mathematical finance, Applications of Mathematics (New York), 39, Springer, New York, 1998. 
%MR1640352 (2000e:91076)


%% \bibitem{IW} Ikeda  N. and S. Watanabe: {\em Stochastic Differential Equations and Diffusion Processes},  North Holland, Kodansha, 1989.

%%%\bibitem{JS} Jacod J. and A.  Shiryaev: {\em Limit Theorem for Stochastic Processes}, 
%%Springer, 2003.

%\bibitem{KS} Kl\"oppel S. and M. Schweizer (2007): Dynamic utility based good deals bonds, {\em Statistics and Decisions}  25, 285-309.

\bibitem{Kob} Kobylanski M. and M.-C. Quenez (2012): 
Optimal stopping time problem in a general framework, 
\textit{Electronic Journal of Probability}  17, 1-28.

\bibitem{ERRATUM} Kobylanski M. and M.-C. Quenez (2016): 
Erratum: Optimal stopping time problem in a general framework, available at https://hal.archives-ouvertes.fr/hal-01328196. 

\bibitem{Marc} Kobylanski M., M.-C. Quenez,  and M. Roger de Campagnolle (2014): 
    Dynkin games in a general framework,
    \textit{Stochastics} 86(2), 304-329.

\bibitem{KQR}
Kobylanski M., M.-C. Quenez\ and\ E. Rouy-Mironescu, Optimal multiple stopping time problem, Ann. Appl. Probab. {\bf 21} (2011), no.~4, 1365--1399. MR2857451 (2012h:60130)



%\bibitem{KQ} Kobylanski, M. and Quenez, M.-C. (2012): Optimal stopping time problem in a general framework, {\em Electron.J.Probab.} 17, No.72, 1-28.  

%%\bibitem{KLQT} 
% Kobylanski M., Lepeltier J.P., Quenez M.-C. and  S. Torres: Reflected BSDE with superlinear
%quadratic coefficient. Probability and Mathematical Statistics 2002, 22, 51-83.

%\bibitem{LeGall}
%Le Gall, J.-F. (2013): Mouvement brownien, martingales et calcul stochastique, Collection: 
%Mathématiques et Applications, Vol. 71, Springer Verlag Berlin Heidelberg. 

\bibitem{Lenglart}
Lenglart E. (1980): Tribus de Meyer et théorie des processus, Séminaire de probabilités de Strasbourg XIV 1978/79, 
Lecture Notes in Mathematics Vol. 784, 500-546.  

\bibitem{Maingueneau}  
Maingueneau M. A. (1977): Temps d'arr\^et optimaux et th\'eorie g\'en\'erale, S\'eminaire de Probabilit\'es, XII de Strasbourg, 1976/77, 457--467, Lecture Notes in Math., 649 Springer, 
Berlin. 
%MR0520020 (81j:60055)

%\bibitem{Meyer_cours} Meyer P.-A. (1976): Un cours sur les intégrales stochastiques (exposés 1 à 6), Séminaire de probabilités de Strasbourg X, 245-400. 

%\bibitem{Nikeghbali}
%Nikeghbali A. (2006)~: An essay on the general theory of stochastic processes,
%{\em Probab. Surv.} 3, 345-412. 

\bibitem{ouk98}
Ouknine Y. (1998)~: Reflected backward stochastic differential 
equation with jumps, {\em Stochastics and Stoch. Reports} 65, 111-125.

\bibitem{Pape90} Pardoux E. and  S. Peng (1990): Adapted solution of backward stochastic differential equation,
\textit{Systems \& Control Letters} 14, 55–61.

\bibitem{Pape92} Pardoux E. and S. Peng (1992): 
Backward Stochastic Differential equations and Quasilinear Parabolic 
Partial Differential equations,
{\em Lect. Notes in CIS} {176}, 200-217.


\bibitem{Peng_Doob_Meyer}  Peng S. (1999): Monotonic limit theorem of BSDE and nonlinear decomposition theorem of Doob-Meyer's type, 
\textit{Probability Theory and Related Fields},
113(4), 473-499.

\bibitem{Pe04}  Peng S. (2004): 
Nonlinear expectations, nonlinear evaluations and risk measures, 
165-253, {\em Lecture Notes in Math.}, 1856, Springer, Berlin.

%\bibitem{Pro05}  Protter P. (2005): Stochastic integration and differential equations, 2nd edition, {\it Springer Verlag}.


\bibitem{QuenSul} Quenez M-C. and A. Sulem  (2013): BSDEs with jumps, optimization and applications 
to dynamic risk measures. {\em Stochastic Processes and Their Applications} 123, 0-29.

\bibitem{QuenSul2} Quenez M.-C. and A. Sulem (2014): Reflected BSDEs and robust optimal stopping for dynamic risk measures with jumps, \textit{Stochastic Processes and their Applications} 124(9), 3031-3054.

\bibitem{Gianin} Rosazza-Gianin E. (2006): Risk measures via g-expectations, \textit{Insurance: Mathematics and Economics} 39(1),  19-34.
%
%{\bf (J'AI SUPPRIME LA REFERENCE AVEC FRITELLI POUR RACCOURCIR)}
%
%Reflected BSDEs with jumps and application to optimal stopping for dynamic risk measures,   Manuscript  (2011)

\bibitem{R} Royer M.  (2006):  Backward stochastic differential equations  with jumps and related non-linear expectations, {\em Stochastic Processes and Their Applications} 116, 1358-1376.

%%\bibitem{Tang} Tang S.H. and X. Li: Necessary conditions for optimal control of stochastic systems with random jumps, {\em SIAM J. Cont. and Optim.} 32, (1994), 1447--1475. 


%\bibitem{YinMao} J. Yin and X. Mao (2008): The adapted solution and comparison theorem for backward stochastic differential equations with Poisson jumps and applications. {\em Journal of Mathematical Analysis and Applications} 346, 345--358. 


\bibitem{Tang} Tang S.H. and X. Li (1994): Necessary conditions for optimal control of stochastic systems with random jumps, {\em SIAM J. Cont. and Optim.} 32,  1447-1475. 


\end{thebibliography}
\end{document}